\documentclass[a4paper,10pt]{amsart}

\usepackage[disable]{todonotes}
\usepackage[utf8]{inputenc}
\usepackage[english]{babel}
\usepackage[T1]{fontenc}
\usepackage{amsfonts}
\usepackage{amsmath}
\usepackage{amsthm}
\usepackage{amssymb}
\usepackage{hyperref}
\usepackage{syntonly}
\usepackage{mathtools}
\usepackage{algorithm}
\usepackage{algpseudocode}
\usepackage{tikz,tikz-cd}
\usepackage{dsfont}
\usepackage{braket}
\usepackage{faktor}
\usepackage{array}
\usepackage{lscape}
\usepackage{rotating}
\usepackage{epigraph}
\usepackage{subfig}
\usepackage{indentfirst}
\usepackage{booktabs}
\usepackage{multirow}
\usepackage{caption}
\usepackage{tabu}
\usepackage{tabularx}
\usepackage{xcolor}
\usepackage{pgf}
\usepackage{tikz}
\usepackage{caption}
\usepackage{color}
\usepackage{enumitem}
\usepackage{mathtools}
\usepackage{float}
\usepackage{graphicx}
\usepackage{relsize}
\usepackage[displaymath, tightpage]{preview}
\usepackage[toc,page]{appendix}
\usetikzlibrary{matrix}
\usetikzlibrary{arrows}

\usepackage{mathrsfs}
\def\Moh{\mathscr{M}}

\makeatletter
\newcommand{\eqnum}{\refstepcounter{equation}\textup{\tagform@{\theequation}}}
\makeatother

\theoremstyle{plain}
\newtheorem{theorem}{Theorem}[section]
\newtheorem{lemma}[theorem]{Lemma}
\newtheorem{proposition}[theorem]{Proposition}
\newtheorem{corollary}[theorem]{Corollary}

\theoremstyle{definition}
\newtheorem{definition}[theorem]{Definition}

\theoremstyle{remark}
\newtheorem{remark}[theorem]{Remark}

\newtheorem{example}[theorem]{Example}
\newtheorem{claim}{Claim}

\newcommand{\Pmap}{\mathbf{P}}

\title[Realization spaces of matroids over hyperfields]
      {Realization spaces of matroids over hyperfields}

\author[Emanuele Delucchi]{Emanuele Delucchi}
\address{(Emanuele Delucchi) Department of Mathematics, University of Fribourg, Chemin du Mus\'ee 23, CH-1700, Fribourg, CH.}
\email{emanuele.delucchi@unifr.ch}

\author[Linard Hoessly]{Linard Hoessly}
\address{(Linard Hoessly) Department of Mathematics, University of Fribourg, Chemin du Mus\'ee 23, CH-1700, Fribourg, CH.}
\email{linard.hoessly@unifr.ch}

\author[Elia Saini]{Elia Saini}
\address{(Elia Saini) Department of Mathematics, University of Fribourg, Chemin du Mus\'ee 23, CH-1700, Fribourg, CH.}
\email{elia.saini@unifr.ch}

\begin{document}

\maketitle

\begin{abstract}
    We study realization spaces of matroids over hyperfields (in the sense of Baker and Bowler \cite{BB16}). More precisely, given a matroid $M$ and a hyperfield $\mathbb H$ we determine the space of all
    $\mathbb H$-matroids over $M$.
This can be seen as the matroid stratum of the hyperfield Grassmannian  in the sense of Anderson-Davis \cite{AD17}.

    We give different descriptions of these realization spaces (e.g., in terms of Tutte groups or projective classes), allowing for explicit computations. When the hyperfield at hand is topological, the realization spaces have a natural topology. In this case, our models carry the correct homeomorphism type.
        
    As applications of our methods we obtain    a theorem on the existence of phased matroids that are not realizable over $\mathbb C$ and are not chirotopal, as well as a result on the diffeomorphism type of complex hyperplane arrangements whose underlying matroid is uniform.
 \end{abstract}

\section*{Introduction}
\emph{Matroids over hyperfields} were introduced by Baker and Bowler in \cite{BB16}. They unify several flavors of matroid theory, including oriented matroids \cite{Zie}, valuated matroids \cite{DW92a} and phased matroids \cite{AD12}. Accordingly, they have applications to different areas of mathematics such as tropical geometry, Berkovich theory and classical algebraic geometry \cite[\S 1]{BB16}. 

A matroid over a hyperfield $\mathbb H$ (for short: a $\mathbb H$-matroid) can be defined as a class of Grassmann-Pl\"ucker functions on a (finite) ground set $E$ with values in $\mathbb H$. Hyperfields are field-like objects where addition is allowed to be multivalued. An ordinary matroid corresponds to a matroid over the Krasner hyperfield, a matroid over the sign hyperfield corresponds to an oriented matroid and a matroid over the tropical hyperfield is a valuated matroid. (See Section \ref{section1b} for precise definitions and examples.)
Matroids over hyperfields admit the following ``functorial'' property: given a morphism of hyperfields $f:\mathbb{H}_1\to \mathbb{H}_2$ and an $\mathbb{H}_1$-Matroid $M$, there is an induced $\mathbb H_2$-Matroid $f_\star(M)$.
 The Krasner hyperfield is the terminal object in the category of hyperfields and, accordingly, we can define 
the {\em underlying matroid} of any $\mathbb{H}$-Matroid $M$ as the push-forward $k_\star(M)$ with respect to the unique map $k:\mathbb{H}\to \mathbb K$ (see also discussion after Theorem \ref{TheoremA}).
 In this paper we study the following question:
 \begin{center} 
\emph{What is the space of all matroids over a given hyperfield \\ with a fixed underlying matroid?}
 \end{center}
 We call such spaces {\em realization spaces}. In the case of the hyperfield of signs, these are the discrete, finite sets of oriented matroids studied by Gel'fand, Rybnikov and Stone \cite{GRS95}, who provided 
 different characterizations of them, up to a canonical operation on oriented
 matroids called \emph{reorientation}
 (compare \cite[\S 3.1 and Remark 3.2.3]{Zie},  \cite[p. 121]{GRS95}). 

In general, realization spaces of matroids over hyperfields are not finite; moreover, many hyperfields of interest carry a topological structure which induces a topology on the realization spaces. Our aim is then to model the homeomorphism type of such spaces.
 
 Generalizing the notion of reorientation 
 to the context of hyperfields, we consider the notion of \emph{rescaling class} of a matroid over a hyperfield and we give several descriptions of (the homeomorphism type of) the space of  rescaling classes of matroids over a fixed hyperfield and with a prescribed underlying matroid. Although our work is inspired by \cite{GRS95}, we will see that working in the generality of hyperfields and accounting for topology introduces new challenges. The reward is, then, a better structural understanding as well as a wider array of applications, of which we will outline a sample.
 
 \subsection*{Hyperfields}
 The idea of multivalued algebraic objects goes back at least to 1934, when Marty introduced the notion of hypergroups \cite{Mar34}.  In particular, in 1956 Krasner introduced hyperrings and hyperfields in order to develop some technical tools 
 in the study of approximations of valued fields, see \cite[pp.\ 139-140]{Kra56}. For the formal definition of a hyperfield see \S \ref{section1b} below.
 Ever since their first appearance, algebro-geometric properties of hyperrings have been investigated \cite{DS06,PR87}.
 In \cite{CC11}, Connes and Consani showed that Connes' ad\`{e}le class space of a global field has a hyperring structure, they investigated the connection between ``vectorspaces'' over the Krasner hyperfield and finite projective geometries, 
 and they began the study of multivalued algebraic geometry on hyperrings. 
 For a good overview and the connection to tropical geometry we refer to \cite{Vir10b}. 
 In \cite{Jun15a} and \cite{Jun17},
 Jaiung Jun further developed the theory of algebraic geometry over hyperrings by introducing integral hyperring schemes and used hyperrings in order to generalize the classical notion of valuations.

 \subsection*{Matroids over hyperfields}
Baker and Bowler presented several equivalent (or, in matroid theory parlance, ``cryptomorphic'') descriptions of matroids over hyperfields -- such as via
circuits, dual pairs and Grassmann-Pl\"ucker functions -- as well as a duality theory which depends on the choice of an involution of the hyperfield at hand. A special
feature of this theory is the distinction of two notions of $\mathbb H$-matroids,
namely strong and weak $\mathbb H$-matroids (see \S\ref{section1c} below). Anderson
contributed vector axioms in the strong case \cite{And16}.  For more details on definitions and
examples on matroids over hyperfields we refer to Section
\ref{section1b}.
Note that the follow-up paper of Baker and Bowler \cite{BB17} extends this theory to even more general algebraic structures (for further work in this vein see Pendavingh \cite{RP18}).

\subsection*{Grassmannians} When the hyperfield is a classical field, the space we aim at describing is known as the {\em matroid stratum} of the corresponding Grassmannian, or the  {\em realization space} of the given matroid over the field at hand, going back to \cite{GGMS}. In general, our spaces are related to the {\em hyperfield realization spaces} appearing in Anderson and Davis' work on hyperfield Grassmannians (see \cite{AD17} and Remark \ref{rem:Gdn}.(2)), where  the notion of a {\em topological hyperfield} has been introduced.

In the special case of the sign hyperfield we recover the results of  \cite{GRS95}.
 Moreover, specializing to the tropical hyperfield our work amounts to describing the space of projective equivalence classes of valuated matroids \cite{DW92a} with prescribed underlying matroid. This is a quotient of the matroid's {\em Dressian}, see \cite[\S 4.4]{MaSt}, hence the corresponding specialization of our results fits into the line of research studying the structure of Dressians, see \cite{HeJeJoSt,HeJoSt}. We do not pursue it here, but   we mention as a sample the question of whether one of our descriptions could improve on the upper bound on the dimension of the Dressian of uniform matroids given in \cite[Theorem 31]{JoSc}.

 \subsection*{Results}
Since our goal is to obtain descriptions for the space of all matroids over a given hyperfield $\mathbb H$ with a given underlying matroid $M$, we first verify that the different equivalent definitions of matroids over hyperfields 
 give rise to natural bijections (resp.\ homeomorphisms) between the corresponding spaces, allowing us to properly 
 define 
 \def\lemm{4.7em}
 \begin{enumerate}[label=$\mathcal R_{\mathbb H}(M)$, leftmargin=\lemm]
\item \label{GG0} ``the'' (topological) space of rescaling classes of $\mathbb H$-matroids with underlying matroid $M$. 
 \end{enumerate}

 We then extend the definitions of \cite{GRS95} introducing, in Section \ref{section3},
 
 \begin{enumerate}[label=$\mathcal P_{\mathbb H}(M)$, leftmargin=\lemm]
  \item \label{GG1} the space of \emph{hyperfield projective classes} of a matroid $M$, defined in terms of circuits and 
                   cocircuits of $M$.
 \end{enumerate}
 
 Other than in the oriented matroid case (see \cite[Theorem 1]{GRS95}), for matroids over general hyperfields 
 the space \ref{GG1} needs \emph{not} be in bijection with the space of rescaling classes. 
 In Section \ref{section3b} we characterize algebraically those hyperfields for which this one-to-one correspondence (which, in the topological case, is a homeomorphism) holds. We name the corresponding class of hyperfields \emph{WAM hyperfields} and show that the class of  non-WAM hyperfields is non-empty.

 Then we offer an alternative characterization of \ref{GG0} by proving, in Theorem \ref{big}  ,
 that this space is in bijection with  (and, in the topological case, homeomorphic to)
 
 \begin{enumerate}[label=$\mathcal H_{\mathbb H}(M)$, leftmargin=\lemm]
  \item \label{GG2} a subspace of the set of group homomorphisms from $\mathbb{T}_{M}^{(0)}$ 
                   to the multiplicative group $\mathbb{H^{\ast}}$. Here $\mathbb{T}_{M}^{(0)}$ denotes 
                    the inner Tutte group of $M$, that is a finitely generated abelian group introduced by Dress and Wenzel in 
                   \cite{DW89} and subsequent papers \cite{DW91,Wen89a, Wen89b} 
  as an algebraic counterpart of Tutte's homotopy theory \cite{Tut65}.
 \end{enumerate}
 
 Exploiting different presentations of the inner Tutte group we then prove (Theorem \ref{onetoone}) that 
 the spaces \ref{GG0} and \ref{GG2} are in bijection with (resp.\ homeomorphic to)  
 \begin{enumerate}[label=$\mathcal G_{\mathbb H}(M)$, leftmargin=\lemm]
  \item \label{GG3} the space of $\mathbb{H}^{\ast}$-\emph{cross-ratios}, and 
          \end{enumerate}
           \begin{enumerate}[label=$\mathcal G^R_{\mathbb H}(M)$, leftmargin=\lemm]
  \item \label{GG4} the space of \emph{reduced} $\mathbb{H}^{\ast}$ \emph{cross-ratios}, 
                   which affords easier geometric considerations, obtained by studying a presentation of the inner Tutte group that eliminates redundancies.
 \end{enumerate}

Geometric and algebraic properties of such spaces can be used to tackle specific problems. For example, working with \ref{GG4}, we derive an explicit characterization of rescaling classes as solution of systems of equations. This allows in Proposition \ref{finiteH} to give upper bounds on the number of weak matroids over finite hyperfields with underlying matroid $M$ in terms of circuits of $M$.

As a final structural result, with Theorem \ref{rescalingspaceembeddings} we prove that if $\mathbb H_1$ is a sub-hyperfield of $\mathbb H_2$ then the space of $\mathbb H_1$-rescaling classes over a fixed matroid $M$ embed into that of $\mathbb H_2$-rescaling classes.

 \subsection*{Applications} Our methods and results allow us to use topological and geometric techniques in order to obtain the following applications. 
 \begin{enumerate}[label=(\arabic{*})]
  \setcounter{enumi}{5}
 \item[--] There exist phased matroids that are neither realizable over $\mathbb C$ nor arising from the ``complexification'' of an oriented matroid (Theorem \ref{SuperMAIN}).
 \item[--] The diffeomorphism type of the complement manifold of any two arrangements of hyperplanes in complex space with uniform underlying matroid is determined by the underlying matroid itself  (Corollary \ref{cor:UHA}).
 \end{enumerate}

\subsection*{Plan}
In Section \ref{section1} we recall the basics of matroids over
hyperfields. Then we define rescaling classes of matroids over hyperfields in
Section \ref{section2}. In Section \ref{section3} we study projective
classes of matroids over hyperfields, we introduce the class of WAM hyperfields and we explain how these
particular hyperfields relate to the correspondence between projective
classes and rescaling classes.

In Section
\ref{section5} we give characterizations of spaces of rescaling  classes based on Tutte groups. (The definitions of those Tutte groups as well as other technical ingredients of the proofs are given in an Appendix.)   
Finally, in Section \ref{sec:App} we derive the stated applications to phased matroids and hyperplane arrangements. 

\subsection*{Remark on the ArXiv history}
The roots of this paper lie in the study of phasing spaces of matroid by the first and third author, 
see \cite{DS15}. While still appearing as a separate ArXiv entry, \cite{DS15} is now encompassed and superseded by the present work, which adopts the wider point of view of matroids over hyperfields. 
            
\subsection*{Acknowledgements}

We thank Laura Anderson, Christopher Eppolito, Jaiung Jun and Thomas Zavslasky 
for the warm hospitality and the very productive discussions during a visit at
SUNY Binghamton.
We also thank Alex Fink, Ivan Martino and Rudi Pendavingh for 
the opportunity to discuss an earlier version of our work. Moreover, 
we thank Richard Randell for feedback about Section \ref{Section9}, Peter Michor for advice regarding Lemma \ref{LEMLEMLEM}, Alberto Cavallo for pointing out \cite{Hir76}. We are also grateful to  Michael Joswig and Benjamin Schr\"oter for sharing their expertise on Dressians in friendly discussions during the 2018 special semester on tropical geometry at the Institute Mittag-Leffler. We are indebted to Nathan Bowler for pointing out an error in an earlier version of this paper.

All authors have been supported by the Swiss National Science Foundation Professorship grant PP00P2\_150552/1.

\section{Basics on matroids over hyperfields}\label{section1}

We recall 
some basic definitions and results about matroids and matroids over hyperfields. 
For a thorough treatment of matroid theory we point to Oxley's book \cite{Oxl92}, 
for basics on hyperfields we refer to Viro \cite{Vir10b} while the foundations of matroids over hyperfields are laid out in the 
preprint by Baker and Bowler \cite{BB16}.

\subsection{Matroids}\label{section1a}
A matroid $M$ is a pair $(E,\mathfrak{B})$, 
where $E$ is a finite set and $\mathfrak{B}\subseteq2^{E}$ 
is a collection of subsets of $E$ satisfying the following two conditions:
\begin{enumerate}[label=(B\arabic{*})]
 \item \label{B1} $\emptyset\neq\mathfrak{B}$;
 \item \label{B2} For all $B_1,B_2\in\mathfrak{B}$ and  $b_1\in B_1\setminus B_2$, there exists $b_2\in B_2\setminus B_1$ such that $(B_1\setminus\{b_1\})\cup\{b_2\}\in\mathfrak{B}$
\end{enumerate}
The set $E$ is called the ground set of $M$. The members of $\mathfrak{B}$ are the \emph{bases} of $M$.
The collection of subsets of elements of $\mathfrak{B}$ are the \emph{independent sets} of $M$. 
A subset of $E$ that is not independent is 
called \emph{dependent}. Minimal inclusion dependent sets are called \emph{circuits} and 
the family of circuits of $M$ will be denoted by $\mathfrak{C}$. 
We will write $E(M)$, $\mathfrak{B}(M)$ and $\mathfrak{C}(M)$ when we will need to specify which matroid we are considering.

The \emph{rank} of a subset $S\subseteq E$ is defined by
\begin{equation*}
 \operatorname{rk}(S)=\max\left\lbrace
                         |S\cap B|
                         \mid
                         B\in\mathfrak{B}
                        \right\rbrace
\end{equation*}
and we define the \emph{rank of the matroid} $M$ as $\operatorname{rk}(M):=\operatorname{rk}(E)$. 
A subset $S$ of $E$ is \emph{spanning} if $\operatorname{rk}(S)=\operatorname{rk}(M)$.

\begin{remark}[Cryptomorphisms]
Our definition in terms of bases can be replaced by 
a list of requirements for any of the set systems described by an italicized word above (and many more). 
This availability of different reformulations is a
distinctive feature of matroid theory. The rules allowing to switch between 
these reformulations are called ``cryptomorphisms''.
\end{remark}

\begin{remark}[Duality]
\def\rk{\operatorname{rk}}
The family of complements of bases of a matroid $M$ is the collection of bases of a matroid $M^{\ast}$ called 
\emph{dual} to $M$. The rank function $\rk^{\ast}$ of $M^{\ast}$ is related to that of $M$ 
via $\rk^{\ast}(A) = \rk(E\setminus A) + \vert A \vert - \rk(E)$.

Circuits of $M^{\ast}$ are called \emph{cocircuits} of $M$. 
We write $\mathfrak{C}^{\ast}(M)$ (or simply $\mathfrak{C}^{\ast}$ if no confusion can arise) for the family of cocircuits of $M$. 
\end{remark}

\begin{remark}[Representability]
A matroid is called representable if its ground set $E$ maps into a vector space $V$ so that a subset of $E$ is independent if and only if the corresponding vectors are linearly independent.
\end{remark}

\begin{example}[The Fano matroid] \label{ex:fano} 
The Fano matroid is defined on the ground set $E=\{1,2,\ldots,7\}$ 
by the circuit set
\begin{equation*}
 \mathfrak{C}=\left\lbrace
                \{1,2,3\},\{2,5,7\},\{1,4,7\},\{1,5,6\},\{3,4,5\},\{3,6,7\},\{2,4,6\}
              \right\rbrace.
\end{equation*}
It is representable over $\mathbb{F}$ if and only if the characteristic of $\mathbb F$ is two 
\cite[Proposition 6.4.8]{Oxl92}.
\end{example}

\begin{remark}[Minors]
Given a subset $T$ of the ground set $E$ of a matroid $M$, the collection of all inclusion maximal sets of the form $B\cap T$ where $B$ is a basis of $M$ satisfies the axioms (B1) and (B2). Thus, it is the collection of bases of a matroid, called \emph{restriction} of $M$ to $T$ and denoted by $M[T]$. 
The \emph{contraction} of $T$ in $M$ is the matroid $(M^{\ast}[E\setminus T])^{\ast}$. 
A \emph{minor} 
of $M$ is any matroid that can be obtained from $M$ through a sequence of restrictions and contractions.
\end{remark}

\begin{remark}[Connectivity]
Given matroids $M_{1}$ and $M_{2}$ with ground sets $E_{1}$ and $E_{2}$,  the \emph{direct sum} 
of $M_{1}$ and $M_{2}$ is the matroid $M_{1}\oplus M_{2}$ 
with ground set $E_{1}\cup E_{2}$ and bases 
\begin{equation*}
\mathfrak{B}(M_{1}\oplus M_{2})=
 \left\lbrace
 B_{1}\cup B_{2}\mid B_{1}\in\mathfrak{B}(M_{1})\text{ and }B_{2}\in\mathfrak{B}(M_{2})
 \right\rbrace.
\end{equation*}
\end{remark}

We say that $M$ is \emph{disconnected} if there exists a proper non-empty subset $T$ 
of the ground set $E$ such that $M=M[T]\oplus M[E\setminus T]$. We call $M$ \emph{connected} 
otherwise. A \emph{connected component} 
of $M$ is a maximal inclusion subset $T$ of $E$ such that $M[T]$ is connected. The decomposition of a matroid $M$ as a direct sum of connected 
matroids is unique (up to permutations), see e.g., \cite[Corollary 4.2.13]{Oxl92}. Therefore, the number $c_{M}$ of connected components of a matroid $M$ is well-defined.

\subsection{Hyperfields}\label{section1b}
Given a set $S$, a \emph{hyperoperation} $\boxplus$ on 
$S$ is a map from $S\times S$ to the collection of non-empty 
subsets of $S$. If $A$ and $B$ are non-empty subsets of $S$, we set 
\begin{equation*}
 A\boxplus B=\bigcup_{a\in A,b\in B}(a\boxplus b)
\end{equation*}
and we say that $\boxplus$ is \emph{commutative} if $a\boxplus b=b\boxplus a$ for all $a$, $b\in S$. 
We call $\boxplus$ \emph{associative} if $a\boxplus(b\boxplus c)=(a\boxplus b)\boxplus c$ for all $a$, $b$, $c\in S$. 

A commutative \emph{hypergroup} is a tuple $(G,\boxplus,0)$, where $G$ is a set and $\boxplus$ is a commutative and associative hyperoperation 
on $G$ such that
\begin{enumerate}[label=(H\arabic{*})]
  \item \label{H1} $0\boxplus x=\{x\}$ for all $x\in G$;
  \item \label{H2} For each $x\in G$ there is a unique element of $G$ 
                   (denoted by $-x$ and called the \emph{hyperinverse} of $x$) such that $0\in x\boxplus-x$;
  \item \label{H3} $x\in y\boxplus z$ if and only if $z\in x\boxplus-y$.
 \end{enumerate}

Given a commutative monoid $(R,\odot,1)$, an element $r\in R$ and a non-empty subset $A$ of $R$ we define
\begin{equation*}
 r\odot A=\{r\odot a\mid a\in A\}.
\end{equation*}

A commutative \emph{hyperring} is a tuple $(R,\boxplus,\odot,0,1)$ such that
\begin{enumerate}
 \item $(R,\boxplus,0)$ is a commutative hypergroup;
 \item $(R,\odot,1)$ is a commutative monoid;
 \item $0\odot x=x\odot0=0$ for all $x\in R$ (Absorption rule);
 \item $a\odot(x\boxplus y)=(a\odot x)\boxplus(a\odot y)$ for all $a$, $x$, $y\in R$ (Distributive law).
\end{enumerate}

\begin{definition}\label{dh}
 A \emph{hyperfield} is a commutative hyperring $(\mathbb{H},\boxplus,\odot,0,1)$ 
 such that $0\neq1$ and all non-zero 
 elements of $\mathbb{H}$ have an inverse with respect to $\odot$.    
\end{definition}

When no confusion arises, we denote a hyperfield by its underlying set $\mathbb{H}$ and we write 
$\mathbb{H}^{\ast}$ for the set of its non-zero elements. 
We will often denote $x\odot y^{-1}$ by $\frac{x}{y}$. Moreover, we will at times not distinguish between a singleton set and its unique element (e.g., writing axiom (H1) as $0\boxplus x = x$).

A \emph{sub-hyperfield} $\mathbb{H}_{1}$ of a hyperfield $\mathbb{H}_{2}$ is a subset 
$\mathbb{H}_{1}\subseteq\mathbb{H}_{2}$ 
that itself is a hyperfield with respect to the operations induced by $\mathbb{H}_{2}$.

A \emph{hyperfield homomorphism} is a map $f:\mathbb{H}_{1}\longrightarrow\mathbb{H}_{2}$ such that:
\begin{itemize}
 \item $f(0)=0$;
 \item $f(1)=1$;
 \item $f(x\boxplus y)\subseteq f(x)\boxplus f(y)$ for any $x$, $y\in\mathbb{H}_{1}$; 
 \item $f(x\odot y)=f(x)\odot f(y)$ for any $x$, $y\in\mathbb{H}_{1}$.
\end{itemize}

An \emph{involution} of the hyperfield $\mathbb{H}$ is a hyperfield homomorphism 
$\tau:\mathbb{H}\longrightarrow\mathbb{H}$ such that $\tau\circ\tau=\operatorname{id}_{\mathbb{H}}$.
According to this definition the identity map of $\mathbb{H}$ is an involution.

In the following statement we summarize for later reference some elementary algebraic properties of hyperfields. 
\begin{proposition}\label{algbas}
 For any hyperfield $\mathbb{H}$ with an involution $\tau$:
 \begin{enumerate}[label=(A\arabic{*})]
  \item \label{A1} $(-1)\odot f=f\odot(-1)=-f$ for all $f\in\mathbb{H}$;
  \item \label{A2} $\tau(-1)=-1$.
 \end{enumerate}
\end{proposition}

\begin{proof}
 Since $\mathbb{H}^{\ast}$ is an abelian group, we obviously have 
  $(-1)\odot f=f\odot(-1)$.
 Thus, for (A1) it suffices 
 to see that $f\odot(-1)=-f$ and this follows immediately from 
 \begin{equation*}
  0=0\odot f=f\odot 0\in f\odot\{0\}=f\odot(1\boxplus(-1))=(f\odot1)\boxplus(f\odot(-1))=f\boxplus(f\odot(-1)).
 \end{equation*}
 Similarly, to prove that $\tau(-1)=-1$ it is enough to compute
 \begin{equation*}
  0=\tau(0)\in\tau(\{0\})=\tau(1\boxplus(-1))\subseteq\tau(1)\boxplus\tau(-1)=1\boxplus\tau(-1).
 \end{equation*}
 By the uniqueness of inverses (H2), this implies (A2).
\end{proof}

The following notion of a {\em topological} hyperfield has been recently introduced by Anderson and Davis \cite{AD17}.

\begin{definition}\label{def:toph} A {\em topological hyperfield} is a hyperfield $(\mathbb H,\boxplus,\odot)$ with a topology on $\mathbb H$ such that $\mathbb H^*$ is open, the multiplication map $\odot: \mathbb H\times \mathbb H \to \mathbb H$ is continuous, and the multiplicative inverse map $(\cdot)^{-1}: \mathbb H^* \to \mathbb H^*$ is continuous.

A homomorphism of topological hyperfields is a hyperfield homomorphism that is continuous with respect to the given topology.  Accordingly, when talking about topological hyperfields we consider only continuous involutions. In particular, every involution of a topological hyperfield is a homeomorphism.
\end{definition}

\begin{example}\label{example1}
 We list here some relevant hyperfields. See for instance \cite{BB16,AD17,Vir10b}  for more details and examples (but be aware of the at times diverging notations).
 \begin{itemize}
    \setlength{\parskip}{0pt}
  \item Every field $(\mathbb F,+,\cdot)$ defines a hyperfield $(\mathbb F,\boxplus,\odot)$ with $x\boxplus y =\{x+y\}$ and $x\odot y = x \cdot y$.
  \item The \emph{Krasner hyperfield} \index{hyperfield!Krasner} 
        $\mathbb{K}$ is defined on the set $\{0,1\}$ with the usual multiplication rule 
        and hyperaddition law given by: 
        \begin{itemize}
         \item[$\bullet$] $0\boxplus x=x\boxplus0=\{x\}$ if $x=0$ or $x=1$;
         \item[$\bullet$] $1\boxplus1=\{0,1\}$.
        \end{itemize}
        The involution $\tau$ is the identity.
  \item The \emph{hyperfield of signs} \index{hyperfield!of signs}
        $\mathbb{S}$ is defined on the set $\{0,1,-1\}$ with the usual multiplication rule and 
        hyperaddition law given by setting: 
        \begin{itemize}
         \item[$\bullet$] $1\boxplus1=\{1\}$;
         $(-1)\boxplus(-1)=\{-1\}$;
         \item[$\bullet$] $x\boxplus0=0\boxplus x=\{x\}$;
         \item[$\bullet$] $1\boxplus(-1)=(-1)\boxplus1=\{0,1,-1\}$.
        \end{itemize}        
        The involution $\tau$ is the identity.
  \item The \emph{phase hyperfield} \index{hyperfield!phase} 
        $\mathbb{P}$ is defined on the set $S^{1}\cup\{0\}$, where $S^{1}$ is the complex unit circle, 
        with usual multiplication rule and hyperaddition law given by setting: 
        \[ x \boxplus y :=
  \begin{cases}
    x       & \quad \text{if } y=0 \\
        y       & \quad \text{if } x=0 \\
\{x,-x,0\}  & \quad  \text{if } x=-y\neq 0 \\
     \left\lbrace\left.\frac{\alpha x+\beta y}{\lVert\alpha x+\beta y\rVert}
                                         \right|\alpha,\beta\in\mathbb{R}_{>0}
                            \right\rbrace  & \quad \text{otherwise } \\

  \end{cases}
\]
        The involution $\tau$ is complex conjugation. Note that the name of this hyperfield is used differently in \cite{Vir10b}.
  \item The \emph{tropical hyperfield} \index{hyperfield!tropical} 
        $\mathbb{T}_{+}$ is defined on the set $\mathbb{R}\cup\{-\infty\}$ with multiplication rule 
        defined by $a\odot b=a+b$ (and $-\infty$ as absorbing element) and hyperaddition law given by setting:
         \[ x \boxplus y :=
  \begin{cases}
    \{\max\{x,y\}\}       & \quad \text{if } x\neq y \\
        \{c\in\mathbb{R}\cup\{-\infty\} \mid c\leq x\}       & \quad \text{if } x=y \\

  \end{cases}
\]
        The involution $\tau$ is the identity.
  \item The \emph{triangle hyperfield} \index{hyperfield!triangle} 
        $\mathbb{V}$ is defined on the set $\mathbb{R}_{\geq0}$ with usual multiplication rule and 
        hyperaddition law defined by setting: 
        \[ a\boxplus b=\{c\in\mathbb{R}_{\geq0}\mid |a-b|\leq c\leq a+b\}.\]
        
        The involution $\tau$ is the identity.
 \end{itemize} 
\end{example}

\subsection{Matroids over hyperfields}\label{section1c}
\begin{remark}
Throughout this work, we always assume that a hyperfield $\mathbb{H}$, an involution $\tau$ of $\mathbb{H}$
and a finite ground set $E:=\{1,\ldots,m\}$ are given.
\end{remark}

A \emph{hyperfield vector} is any $X\in\mathbb{H}^{E}$.
The \emph{support} of a hyperfield vector $X$ is the set
\begin{equation*}
 \operatorname{supp}(X):=\left\lbrace e\in E\mid X(e)\neq0\right\rbrace.
\end{equation*}

\begin{definition}[Orthogonal hyperfield vectors]
Two hyperfield vectors $X$ and $Y$ are \emph{orthogonal} 
with respect to $\tau$ (written $X\bot_{\tau}Y$) if 
       \begin{equation}\label{orco}
       0\in\mathop{\mathlarger{\mathlarger{\mathlarger{\boxplus}}}}_{e\in E}X(e)\odot\tau(Y(e)).
       \end{equation}
Two sets $\mathcal{X},\mathcal{Y}$ of hyperfield vectors are orthogonal with respect to $\tau$ --- written $\mathcal{X}\bot_{\tau}\mathcal{Y}$ ---
if $X\bot_{\tau}Y$ for all $X\in\mathcal{X}$ and $Y\in\mathcal{Y}$.
\end{definition}

As explained in \cite{BB16} there exist two different kinds of matroids over a hyperfield $\mathbb{H}$, called 
weak $\mathbb{H}\text{-matroids}$ and strong $\mathbb{H}\text{-matroids}$. We now provide definitions and we recall 
cryptomorphisms for both.

\begin{definition}[Grassmann--Pl\"{u}cker functions; {\cite[Definition 3.6]{BB16}}]\label{GPaxioms}
A rank $d$ \emph{weak Grassmann--Pl\"{u}cker function}
on $E$ with values in $\mathbb{H}$ is a non-zero alternating function 
$\varphi:E^{d}\longrightarrow\mathbb{H}$ whose support is the set of bases of a matroid and such that
\begin{equation*}
         0\in\mathop{\mathlarger{\mathlarger{\mathlarger{\boxplus}}}}_{k=1}^{d+1}
             (-1)^{k}\odot\varphi(x_{1},\ldots,\hat{x}_{k},\ldots,x_{d+1})\odot\varphi(x_{k},y_{1}\,\ldots,y_{d-1})
\end{equation*}
for any two subsets $I=\{x_{1},\ldots,x_{d+1}\}$ and $J=\{y_{1},\ldots,y_{d-1}\}$ of $E$ with $|I\setminus J|\leq3$.

A rank $d$ \emph{strong Grassmann--Pl\"{u}cker function}
on $E$ with values in $\mathbb{H}$ is a non-zero alternating function 
$\varphi:E^{d}\longrightarrow\mathbb{H}$ such that
\begin{equation*}
         0\in\mathop{\mathlarger{\mathlarger{\mathlarger{\boxplus}}}}_{k=1}^{d+1}
             (-1)^{k}\odot\varphi(x_{1},\ldots,\hat{x}_{k},\ldots,x_{d+1})\odot\varphi(x_{k},y_{1}\,\ldots,y_{d-1})
\end{equation*}
for any two subsets $I=\{x_{1},\ldots,x_{d+1}\}$ and $J=\{y_{1},\ldots,y_{d-1}\}$ of $E$.

We say that two weak (resp.\ strong) Grassmann--Pl\"{u}cker functions $\varphi_{1}$ and $\varphi_{2}$ are \emph{equivalent} (written $\varphi_{1}\sim_p\varphi_{2}$)
if $\varphi_{1}=a\odot\varphi_{2}$ for some $a\in\mathbb{H}^{\ast}$.
\end{definition}

The axioms given in Definition \ref{GPaxioms} ensure 
that the support of any weak (resp.\ strong) Grassmann--Pl\"{u}cker function
$\varphi$ is the set of bases of a matroid on $E$ which we call
$M_{\varphi}$. We then call a subset $I\subseteq E$  
\emph{$\varphi\text{-independent}$} if it is an independent set of the matroid $M_{\varphi}$. 
Accordingly, a \emph{$\varphi\text{-basis}$} is a maximal $\varphi\text{-independent}$
set.

In order to state weak (resp.\ strong) $\mathbb{H}\text{-circuits}$
axioms of weak (resp.\ strong) $\mathbb{H}\text{-matroids}$ 
we need at first to recall the notion of modular pair (resp. modular elimination structure).

As suggested by Baker and Bowler in \cite[Section 1.2]{BB16},  modular elimination can be interpreted as follows: if $X$ and $Y$ are hyperfield vectors that are ``sufficiently close'' and there exists an element $e\in E$ such that 
$X(e)=-Y(e)$, then it is possible to ``eliminate'' $e$ by (hyper-)adding $X$ and $Y$, i.e. there is $Z$ with $Z(e)=0$ and $Z(f)\in X(f)\boxplus Y(f)$ for all $f$.

To be more precise, given a family $\mathcal{C}\subseteq\mathbb{H}^{E}$ we say that $X$, $Y\in\mathcal{C}$ form 
a \emph{modular pair} 
if $\operatorname{supp}(X)$, $\operatorname{supp}(Y)$ is a modular pair in the lattice of unions of 
supports of elements of $\mathcal{C}$ \cite{Del11}. More generally (compare \cite[Definition 3.7]{BB16}), 
assume that we have a subset $P$ of $E$, an indexed family 
$(X_{p})_{p\in P}\subseteq\mathcal{C}$ with $\operatorname{supp}(X_{p})\cap P=\{p\}$, and $X\in\mathcal{C}$ with 
$X(p)=-X_{p}(p)$ for all $p\in P$ but $\operatorname{supp}(X)\nsubseteq\bigcup_{p\in P}\operatorname{supp}(X_{p})$. We say 
that $X$ and $(X_{p})_{p\in P}$ give a \emph{modular elimination structure} 
if the height of 
$\operatorname{supp}(X)\cup\bigcup_{p\in P}\operatorname{supp}(X_{p})$ 
in the lattice of unions of supports of elements of $\mathcal{C}$ is exactly $|P|+1$.

\begin{definition}[Weak $\mathbb{H}\text{-circuits}$; {\cite[Definition 3.4]{BB16}}]\label{WHCaxioms}
 A set $\mathcal{C}\subseteq\mathbb{H}^{E}$ is the set of \emph{weak $\mathbb{H}\text{-circuits}$} 
 of a 
 \emph{weak $\mathbb{H}\text{-matroid}$} $\mathcal{M}$ on $E$ if:
 \begin{enumerate}[label=(C\arabic{*})] 
  \item \label{CC0} $(0,\ldots,0)\notin\mathcal{C}$;
  \item \label{CC1} For all $X\in\mathcal{C}$ and all $\alpha\in\mathbb{H}^{\ast}$, $\alpha\odot X\in\mathcal{C}$;
  \item \label{CC2} For all $X,Y\in\mathcal{C}$ such that 
                    $\operatorname{supp}(X)=\operatorname{supp}(Y),$ $X=\alpha\odot Y$ for some $\alpha\in\mathbb{H}$;
  \item \label{CC3} [Weak modular elimination] 
                    For any modular pair $X$, $Y\in\mathcal{C}$ and for any 
                    $e\in E$ with $X(e)=-Y(e)\neq0$, there exists $Z\in\mathcal{C}$ such that $Z(e)=0$ and 
                    $Z(f)\in X(f)\boxplus Y(f)$ for all $f\in E$.
 \end{enumerate}
\end{definition}

\begin{definition}[Strong $\mathbb{H}\text{-circuits}$; {\cite[Definition 3.7]{BB16}}]\label{SHCaxioms}
 A set $\mathcal{C}\subseteq\mathbb{H}^{E}$
 is the set of \emph{strong $\mathbb{H}\text{-circuits}$}
 of a 
 \emph{strong $\mathbb{H}\text{-matroid}$} $\mathcal{M}$ 
 on $E$ if it satisfies \ref{CC0}, \ref{CC1}, \ref{CC2} as well as the 
 following stronger version of the modular elimination axiom \ref{CC3}:
 \begin{enumerate}
  \item[(C4)'] [Strong modular elimination] For any modular elimination structure given by $X\in\mathcal{C}$ and 
               $(X_{p})_{p\in P}\subseteq\mathcal{C}$, there is 
               $Z\in\mathcal{C}$ with $Z(p)=0$ for all $p\in P$ and $Z(f)\in X(f)\boxplus(\boxplus_{p\in P}X_{p}(f))$ for any 
               $f\in E$.
 \end{enumerate}
\end{definition}

If we take $|P|=1$ we immediately notice that (C4)' implies  \ref{CC3}. 
Therefore, a strong $\mathbb{H}\text{-matroid}$ on $E$ is also a weak $\mathbb{H}\text{-matroid}$ on $E$.

If $\mathcal{C}$ is the set of weak (resp.\ strong) $\mathbb{H}\text{-circuits}$ 
of a weak (resp.\ strong) $\mathbb{H}\text{-matroid}$ $\mathcal{M}$ on $E$, 
the set $\{\operatorname{supp}(X)\mid X\in\mathcal{C}\}$ is the set of 
circuits of a matroid $M_{\mathcal{C}}$. 
The rank of $\Moh$ 
is defined to be the rank of the matroid $M_{\mathcal{C}}$.
The following theorem asserts that Definition \ref{GPaxioms}
and Definition \ref{WHCaxioms} encode equivalent data.

\begin{theorem}[{\cite[Theorem 3.13, Theorem 3.17]{BB16}}]
 \label{TheoremA}
 Given a set $E$ and a hyperfield $\mathbb{H}$, there exists a bijection between the set of all equivalence classes of 
 rank $d$ weak (resp.\ strong) Grassmann--Pl\"{u}cker 
 functions on a $E$ with values in $\mathbb{H}$ and the set of all sets of weak (resp.\ strong) 
 $\mathbb{H}\text{-circuit}$ 
 of a rank $d$ weak (resp.\ strong) $\mathbb{H}\text{-matroid}$  on $E$, determined as follows. 
 For a weak (resp.\ strong) Grassmann--Pl\"{u}cker function $\varphi$ and the corresponding 
 set $\mathcal{C}$ of weak (resp.\ strong) $\mathbb{H}\text{-circuits}$:
 \begin{enumerate}
  \item The set of all supports of elements of $\mathcal{C}$ is the set of minimal non-empty $\varphi\text{-dependent}$ 
        sets;
  \item The 
   $\mathbb{H}\text{-circuits}$ $X\in\mathcal{C}$ are determined by the rule
        \begin{equation*}
         \frac{X(x_{i})}{X(x_{0})}=(-1)^{i}\odot
                                   \frac{\varphi(x_{0},x_{1},\ldots,\hat{x}_{k},\ldots,x_{d})}
                                        {\varphi(x_{1},\ldots,x_{d})}
        \end{equation*}
        for all $i=0,\ldots,k$ where $x_{0}\in\operatorname{supp}(X)$ and $\{x_{1},\ldots,x_{d}\}$ is any 
        $\varphi\text{-basis}$ containing $\operatorname{supp}(X)\setminus\{x_{0}\}$.
 \end{enumerate}
\end{theorem}

Thus, we can refer to the rank $d$ weak (resp.\ strong) matroid $\Moh$ over the hyperfield $\mathbb{H}$
(or, for short, $\mathbb{H}\text{-matroid}$) with ground set $E$, 
weak (resp.\ strong) Grassmann--Pl\"{u}cker function
$\varphi$ and weak (resp.\ strong) $\mathbb{H}\text{-circuits}$ $\mathcal{C}$. 
In particular, in this case $M_{\varphi}=M_{\mathcal{C}}$ -- we call this matroid the \emph{underlying matroid}  
of $\Moh$.

In the setting of matroids over hyperfields, duality depends 
on the choice of an involution of the hyperfield (see discussion after Definition \ref{dh}). To be more precise, for a $\mathbb{H}\text{-matroid}$ $\Moh$, any 
involution $\tau$ of $\mathbb{H}$ gives rise to a matroid $\Moh^{(\tau)}$ ``dual'' to $\Moh$ 
as explained in the following statement.

\begin{theorem}[{\cite[Theorem 3.20]{BB16}}]
 \label{Duality}
 Given a finite ground set $E$ with $|E|=m$, a hyperfield $\mathbb{H}$, an involution $\tau$ of $\mathbb{H}$ and a rank $d$ 
 weak (resp.\ strong) $\mathbb{H}\text{-matroid}$ $\Moh$ on $E$ with weak (resp.\ strong) 
 $\mathbb{H}\text{-circuits}$ $\mathcal{C}$ and 
 rank $d$ weak (resp.\ strong) Grassmann--Pl\"{u}cker function $\varphi$, there exists a rank $m-d$ weak (resp.\ strong) 
 $\mathbb{H}\text{-matroid}$ $\Moh^{(\tau)}$ 
 on $E$, called the \emph{dual} $\mathbb{H}\text{-matroid}$ of $\Moh$
 with respect to $\tau$,
 that satisfies the following properties.  \begin{enumerate}
  \item The set $\mathcal{C}^{(\tau)}$ of $\mathbb{H}\text{-circuits}$ of $\Moh^{(\tau)}$ are the elements of 
        $\operatorname{SuppMin}(\mathcal{C}^{\bot}\setminus\{0\})$, where $\operatorname{SuppMin}(S)$ denotes the elements of 
        $S$ of minimal support;
  \item A weak (resp.\ strong) Grassmann--Pl\"{u}cker function $\varphi^{(\tau)}$ for $\Moh^{(\tau)}$ is defined by the 
        formula
        \begin{equation*}
         \varphi^{(\tau)}(x_{1},\ldots,x_{m-d})=\operatorname{sign}(x_{1},\ldots,x_{m-d},x_{1}',\ldots,x_{r}')
                                                \odot
                                                \tau(\varphi(x_{1}',\ldots,x_{r}'))
        \end{equation*}
        where $x_{1}',\ldots,x_{r}'$ is any ordering of $E\setminus\{x_{1},\ldots,x_{m-d}\}$;
  \item The underlying matroid of $\Moh^{(\tau)}$ is the dual of that of $\Moh$;
  \item $(\Moh^{(\tau)})^{(\tau)}=\Moh$.
 \end{enumerate}
 Note that (2) implies in particular that $\Moh^{(\tau)}$ is uniquely determined.
\end{theorem}

The weak (resp.\ strong) $\mathbb{H}\text{-circuits}$ of $\Moh^{(\tau)}$ 
are called the weak (resp.\ strong) $\mathbb{H}\text{-cocircuits}$ 
of $\Moh$ with respect to $\tau$, and vice versa.

We now recall the definition of matroids over hyperfields in terms of ``dual pairs''. 

\begin{definition}[Dual pairs; {\cite[Definition 3.21, Definition 3.23]{BB16}}]
 \label{DualPairs}
 Let $M$ be a matroid with ground set $E$. We say that a collection $\mathcal{C}\subseteq\mathbb{H}^{E}$ 
 is a \emph{circuit coloring} of $M$(with values in $\mathbb{H}$) if:
 \begin{enumerate}[label=(DP\arabic{*})]
  \item \label{DP1} For all $X\in\mathcal{C}$ and all $\alpha\in\mathbb{H}^{\ast}$, $\alpha\odot X\in\mathcal{C}$;
  \item \label{DP2} For all $X,Y\in\mathcal{C}$ with $\operatorname{supp}(X)=\operatorname{supp}(Y),$ $X=\alpha\odot Y$ 
                   for some $\alpha\in\mathbb{H}^{\ast}$;
  \item \label{DP3} The set $\left\lbrace\operatorname{supp}(X)\mid X\in\mathcal{C}\right\rbrace$ is the set of circuits 
                   of $M$.
 \end{enumerate}
 
 We say that $\mathcal{D}\subseteq\mathbb{H}^{E}$ is a \emph{cocircuit coloring} of $M$ if $\mathcal{D}$ is a 
 circuit coloring of $M^{\ast}$, the dual matroid to $M$.
 Moreover, given a circuit coloring $\mathcal{C}$ and a cocircuit coloring $\mathcal{D}$ of $M$ 
 we say that $\mathcal{C}$, $\mathcal{D}$ form a \emph{weakly dual pair} with respect to $\tau$
 if $C\bot_{\tau}D$ for all $C\in\mathcal{C}$, $D\in\mathcal{D}$ with 
 $|\operatorname{supp}(C)\cap\operatorname{supp}(D)|\leq3$.
 Similarly, we say that $\mathcal{C},\mathcal{D}$ is a \emph{strongly dual pair} with respect to $\tau$ 
 if $C\bot_{\tau}D$ for all $C\in\mathcal{C}$, 
 $D\in\mathcal{D}$.
\end{definition}

\begin{remark}\label{rem:colsig}
We depart from the terminology of \cite{BB16} and choose the term {``coloring''} because we want to reserve the word ``signature'' for the objects defined at the beginning of Section \ref{section2b}. Informally, a coloring is a $\mathbb H^*$-orbit of signatures.
\end{remark}

\begin{theorem}[{\cite[Theorem 3.22, Theorem 3.23]{BB16}}]
 \label{TheoremC}
 Given a matroid $M$ with ground set $E$ and a hyperfield $\mathbb{H}$ with an involution $\tau$,
 let $\mathcal{C}$ be a circuit coloring and $\mathcal{D}$ be a cocircuit coloring of $M$.
 Then $\mathcal{C}$ and $\mathcal{D}$ are the set of weak (resp.\ strong) $\mathbb{H}\text{-circuits}$ and 
 $\mathbb{H}\text{-cocircuits}$ with respect to $\tau$ of a weak (resp.\ strong) 
 $\mathbb{H}\text{-matroid}$ $\Moh$ on $E$ with 
 underlying matroid $M$ if and only if they are a weak (resp.\ strong) dual pair with respect to $\tau$.
\end{theorem}

\begin{example}\label{example2}
 Matroids over the hyperfields listed in Example \ref{example1} are all well-studied combinatorial objects.
 In fact, matroids over hyperfields provide a common framework for several notions of matroids that appear in the 
 literature.
 \begin{itemize}
  \item A (weak or strong) matroid over the Krasner hyperfield $\mathbb{K}$ is the same as a matroid in the usual sense;
  \item A (weak or strong) matroid over the hyperfield of signs $\mathbb{S}$ is the same as an oriented matroid;
  \item A weak matroid over the phase hyperfield $\mathbb{P}$ is the same as the notion of complex matroid introduced by 
        Anderson and Delucchi in \cite[Definition 2.4]{AD12}. 
        Notice that in this context the standard duality theory is given by 
        taking the involution $\tau$ of the hyperfield $\mathbb{P}$ to be complex conjugation
        (compare \cite[Definition 2.12]{AD12}). As pointed out by Baker and Bowler in \cite[Appendix A]{BB16}, both notions 
        of weak (compare \cite[Definition 2.4]{AD12}) and strong 
        (compare \cite[Definition 2.3, Definition 2.15]{AD12}) matroids over the phase hyperfield $\mathbb{P}$ are introduced in \cite{AD12}, but 
        they are mistakenly asserted to be equivalent. However, the arguments in the proof of 
        \cite[Proposition 5.6]{AD12} still hold for the weak case;
  \item A (weak or strong) matroid over the tropical hyperfield $\mathbb{T}_{+}$ is the same as a valuated matroid in the 
        sense of Dress and Wenzel \cite{DW92a}.
 \end{itemize}
\end{example}

\begin{remark}\label{weakstrongequal}
 In the context of matroids and oriented matroids this dependence of the duality theory on the involution is 
 hidden, since the Krasner hyperfield $\mathbb{K}$ and the hyperfield of signs $\mathbb{S}$ 
 have the identity as unique involution.
\end{remark}

As pointed out by Baker and Bowler, the notions of weak and strong matroids over hyperfields do 
not agree in general. In particular, they provide in \cite[Section 3.10]{BB16} the following counterexamples:
\begin{itemize}
  \item A weak matroid over the triangle hyperfield $\mathbb{V}$ that is not a strong matroid over $\mathbb{V}$ 
        (compare \cite[Example 3.30]{BB16});
  \item A weak matroid over the phase hyperfield $\mathbb{P}$ that is not a strong matroid over $\mathbb{P}$ 
        (compare \cite[Example 3.31]{BB16}).
\end{itemize}

However, improving on some results of Dress and Wenzel in \cite{DW92b}, 
Baker and Bowler proved that for the special class of \emph{doubly distributive} 
hyperfields there is a coincidence between the concepts of weak and strong matroid over hyperfields \cite[Section 5]{BB16}.

\section{Rescaling classes of matroids over hyperfields}\label{section2}

Let $M$ be a rank $d$ matroid with ground set $E$ and let $\mathbb{H}$ be a given hyperfield.
We want to study the set of weak (resp.\ strong) $\mathbb{H}\text{-matroids}$ with underlying matroid $M$ and
the space of ``rescaling classes'' of such $\mathbb{H}\text{-matroids}$, 
as an analogue of \cite{DS15, GRS95} in the context of hyperfields.

The goal of this section is to give a precise definition of the ``matroid realization spaces'' we will be considering and to prove that such spaces are well-defined (up to homeomorphism in the topological case) as one switches between the cryptomorphic axiomatizations of matroids over hyperfields.

\subsection{Rescaling classes via Grassmann--Pl\"{u}cker functions}\label{section2a}

\begin{definition}
Let $\mathcal{N}_{\mathbb{H}}^{p,w}(M)$ 
(resp. $\mathcal{N}_{\mathbb{H}}^{p,s}(M)$) be
the set of 
weak (resp.\ strong) Grassmann--Pl\"{u}cker functions with underlying matroid $M$.
When $\mathbb{H}$ is given a topology, we topologize $\mathcal{N}_{\mathbb{H}}^{p,w}(M)$ 
and $\mathcal{N}_{\mathbb{H}}^{p,s}(M)$ as subsets of the product $(\mathbb{H})^{E^d}$.
\end{definition}

Recall the equivalence relation $\sim_{p}$  among weak (resp.\ strong) Grassmann--Pl\"{u}cker functions introduced in 
Definition \ref{GPaxioms}.

\begin{definition}\label{TypeBa}
A type $p$ weak (resp.\ strong) $\mathbb{H}\text{-matroid}$ 
 with underlying matroid $M$ is an
equivalence class of the relation $\sim_{p}$ on $\mathcal{N}_{\mathbb{H}}^{p,w}(M)$ 
(resp. $\mathcal{N}_{\mathbb{H}}^{p,s}(M)$). 
The space of type $p$ weak (resp.\ strong) $\mathbb{H}\text{-matroids}$ with underlying matroid $M$ is
$\mathcal{M}^{p,w}_{\mathbb{H}}(M):=\mathcal{N}_{\mathbb{H}}^{p,w}(M)/\sim_{p}$ 
(resp. $\mathcal{M}^{p,s}_{\mathbb{H}}(M):=\mathcal{N}_{\mathbb{H}}^{p,s}(M)/\sim_{p}$), endowed with the quotient topology.
\end{definition}

We now proceed to define the space of rescaling classes of weak (resp.\ strong) $\mathbb{H}\text{-matroids}$
defined in terms of weak (resp.\ strong) Grassmann--Pl\"{u}cker functions. 
\begin{definition}\label{def_GPssim}
Two weak (resp.\ strong) Grassmann--Pl\"{u}cker functions $\varphi_{1}$ and $\varphi_{2} \in
 \mathcal{N}_{\mathbb{H}}^{p,w}(M)$ (resp. $\mathcal{N}_{\mathbb{H}}^{p,s}(M)$) are called \emph{$\approx_{p}\text{-equivalent}$} 
 if there is a function $h:E\longrightarrow\mathbb{H}^{\ast}$
 such that, for all $(x_{1},\ldots,x_{d})\in E^{d}$,
 \begin{equation}
 \label{piropiro1}
  \varphi_{1}(x_{1},\ldots,x_{d})=\left(\bigodot_{j=1}^{d}h(x_{j})\right)\odot
  \varphi_{2}(x_{1},\ldots,x_{d}).
 \end{equation}
\end{definition}

A straightforward computation shows that $\approx_{p}$ is an
equivalence relation between weak (resp.\ strong) Grassmann--Pl\"{u}cker functions.

\begin{definition}
 Two type $p$ weak (resp.\ strong) $\mathbb{H}\text{-matroids}$ $\Phi_{1}$ and $\Phi_{2}$ with underlying matroid $M$ 
 are \emph{$\cong_{p}\text{-equivalent}$} 
 (denoted by $\Phi_{1}\cong_{p}\Phi_{2}$) if there exist 
 weak (resp.\ strong) Grassmann--Pl\"{u}cker functions 
 $\varphi_{1}\in\Phi_{1}$ and $\varphi_{2}\in\Phi_{2}$ such that 
 $\varphi_{1}\approx_{p}\varphi_{2}$. 
\end{definition}

We can now define a \emph{rescaling class}  as an equivalence classes of $\cong_{p}$.

\begin{definition}\label{TypeB}
 Given a matroid $M$ we define the space
 of rescaling classes of type $p$ weak (resp.\ strong) $\mathbb{H}\text{-matroids}$ 
 with underlying matroid $M$ as
 the set $\mathcal{R}^{p,w}_{\mathbb{H}}(M):=\mathcal{M}_{\mathbb{H}}^{p,w}(M)/\cong_{p}$
 (resp. $\mathcal{R}^{p,s}_{\mathbb{H}}(M):=\mathcal{M}_{\mathbb{H}}^{p,s}(M)/\cong_{p}$) 
 of $\cong_{p}\text{-equivalence}$ classes. Again, in the case of a topological hyperfield we endow the space of Grassmann-Plücker functions with the quotient topology.
\end{definition}

\subsection{Rescaling classes via hyperfield circuit and cocircuit signatures}\label{section2b}

A \emph{$\mathbb{H}^{\ast}\text{-circuit}$ signature} 
$\gamma$ of a matroid $M$ is a collection
$\{\gamma C\}_{C\in \mathfrak{C}}$ of functions  
$\gamma C:C\longrightarrow\mathbb{H}^{\ast}$, one for each circuit of $M$. 
In the same way, a \emph{$\mathbb{H}^{\ast}\text{-cocircuit}$ signature}
$\delta$ of a matroid $M$ is a set $\{\delta D\}_{D\in \mathfrak{C}^{\ast}}$ of functions 
$\delta D:D\longrightarrow\mathbb{H}^{\ast}$, one for 
each cocircuit of $M$. We say that a $\mathbb{H}^{\ast}\text{-circuit}$ signature $\gamma$ and a 
$\mathbb{H}^{\ast}\text{-cocircuit}$ signature $\delta$ are \emph{weak orthogonal} (resp. \emph{strong orthogonal}) 
with respect to $\tau$, denoted by $\gamma\bot_{\tau}\delta$, if, for any circuit $C\in\mathfrak{C}$ and 
cocircuit $D\in\mathfrak{C}^{\ast}$ with $|C\cap D|\leq3$ 
(resp. for any circuit and cocircuit), we have
\begin{equation}\label{Orientation}
    0\in\mathop{\mathlarger{\mathlarger{\mathlarger{\boxplus}}}}_{x\in C\cap D}\gamma C(x)\odot\tau(\delta D(x)).
\end{equation}

\begin{definition}\label{def25}
  We denote by
  $\mathcal{N}_{\mathbb{H}}^{\tau,w}(M)$ 
  (resp. $\mathcal{N}_{\mathbb{H}}^{\tau,s}(M)$) 
  the space of pairs $(\gamma,\delta)$ 
  of $\mathbb{H}^{\ast}\text{-circuit}$ and
  $\mathbb{H}^{\ast}\text{-cocircuit}$ signatures of $M$ that are weak (resp.\ strong) orthogonal 
  with respect to $\tau$. 
  This is a subset of the product $(\mathbb{H}^{\ast})^{\sum_{\mathfrak{C}}|C|\times\sum_{\mathfrak{C}^{\ast}}|D|}$ and we topologize it as such, in the case where $\mathbb{H}$ is a topological hyperfield.
\end{definition}

\begin{definition}\label{sim}
  Two pairs $(\gamma_{1},\delta_{1})$ and $(\gamma_{2},\delta_{2})$ of
  $\mathbb{H}^{\ast}\text{-circuit}$ and $\mathbb{H}^{\ast}\text{-cocircuit}$ signatures of $M$ 
  that are weak (resp.\ strong) orthogonal 
  with respect to $\tau$
  are called \emph{$\sim_{\tau}\text{-equivalent}$} 
  (denoted by $(\gamma_{1},\delta_{1})\sim_{\tau}(\gamma_{2},\delta_{2})$) if there exist functions
  $b:\mathfrak{C}\longrightarrow\mathbb{H}^{\ast}$, $C\mapsto b_{C}$, and 
  $l:\mathfrak{C}^{\ast}\longrightarrow\mathbb{H}^{\ast}$, $D\mapsto l_{D}$, such that:
  \begin{itemize}
   \item $\gamma_{1}C(x)=b_{C}\odot\gamma_{2}C(x)$ for any circuit $C\in\mathfrak{C}$ and any $x\in C$;
   \item $\delta_{1}D(y)=l_{D}\odot\delta_{2}D(y)$ for any cocircuit $D\in\mathfrak{C}^{\ast}$ and any $y\in D$.
  \end{itemize}
\end{definition}

One readily verifies that $\sim_{\tau}$ is an equivalence relation on the set 
$\mathcal{N}_{\mathbb{H}}^{\tau,w}(M)$ 
(resp. $\mathcal{N}_{\mathbb{H}}^{\tau,s}(M)$).

\begin{proposition}\label{laborious}
 The function that associates to a pair $(\gamma,\delta)\in\mathcal{N}_{\mathbb{H}}^{\tau,w}(M)$ 
 (resp. $\mathcal{N}_{\mathbb{H}}^{\tau,s}(M)$) the circuit coloring\footnote{See Definiiton \ref{DualPairs} and Remark \ref{rem:colsig}.} 
 \begin{equation*}
  \mathcal{C}_{(\gamma,\delta)}=\left\lbrace 
                                 X\in\mathbb{H}^{E}
                                 \left|
                                  \exists C\in\mathfrak{C},\exists a\in\mathbb{H}^{\ast}\text{ with }
                                  X(j)=\left\lbrace
                                  \begin{array}{lll}
                                     0                 & \text{if} & j\notin C \\
                                     a\odot\gamma C(j) & \text{if} & j\in C \\
                                  \end{array}
                                 \right.
                                \right.
                               \right\rbrace
 \end{equation*}
induces a bijection 
between the quotient set $\mathcal{N}_{\mathbb{H}}^{\tau,w}(M)/\sim_{\tau}$ 
(resp. $\mathcal{N}_{\mathbb{H}}^{\tau,s}(M)/\sim_{\tau}$) and the family of all sets of 
circuits of weak (resp.\ strong) $\mathbb{H}\text{-matroids}$
with underlying matroid $M$. 
\end{proposition}

We postpone the proof of this proposition in order to continue with the definitions we are now able to state.

\begin{definition}[See Definition \ref{DualPairs}]\label{TypeCa}
 A type $\tau$ weak (resp.\ strong) $\mathbb{H}\text{-matroid}$ 
 is an equivalence class of the
 relation $\sim_{\tau}$ on the set $\mathcal{N}^{\tau,w}_{\mathbb{H}}(M)$
 (resp. $\mathcal{N}_{\mathbb{H}}^{\tau,s}(M)$). 

 \noindent
 The space of type $\tau$ weak (resp.\ strong) $\mathbb{H}\text{-matroids}$ 
 with underlying matroid $M$ is the set 
 $\mathcal{M}^{\tau,w}_{\mathbb{H}}(M):=\mathcal{N}^{\tau,w}_{\mathbb{H}}(M)/\sim_{\tau}$ 
 (resp. $\mathcal{M}^{\tau,s}_{\mathbb{H}}(M):=\mathcal{N}^{\tau,s}_{\mathbb{H}}(M)/\sim_{\tau}$). Again, in the topological case we endow these spaces with the quotient topology.
\end{definition}

We now define an equivalence relation on the set 
of type $\tau$ weak (resp.\ strong) $\mathbb{H}\text{-matroids}$ with underlying matroid $M$, in order
to obtain the counterpart of Definition \ref{TypeB}.

\begin{definition}\label{def_ssim}
 Two pairs $(\gamma_{1},\delta_{1})$ and $(\gamma_{2},\delta_{2})$ of
 $\mathbb{H}^{\ast}\text{-circuit}$ and  $\mathbb{H}^{\ast}\text{-cocircuit}$ signatures of $M$ 
 that are weak (resp.\ strong) orthogonal with respect to $\tau$ 
 are called \emph{$\approx_{\tau}\text{-equivalent}$} 
 (denoted $(\gamma_{1},\delta_{1})\approx_{\tau}(\gamma_{2},\delta_{2})$) 
 if there exists a function $h:E\longrightarrow\mathbb{H}^{\ast}$
 such that:
 
 \begin{itemize}
  \item $\gamma_{1}C(x)=h(x)\odot\gamma_{2}C(x)$ for any circuit $C\in\mathfrak{C}$ and any $x\in C$;
  \item   $\delta_{1}D(y)=\tau(h^{-1}(x))\odot\delta_{2}D(y)$ for any cocircuit $D\in\mathfrak{C}^{\ast}$ and any $y\in D$.
 \end{itemize}

 Here $h^{-1}(x)$ stands for the inverse of $h(x)$ in the multiplicative group $\mathbb{H}^{\ast}$.\todo{perhaps change notation for inverse}
\end{definition}

It is easy to see that $\approx_{\tau}$ is an equivalence relation
on $\mathcal{N}_{\mathbb{H}}^{\tau,w}(M)$ (resp. $\mathcal{N}_{\mathbb{H}}^{\tau,s}(M)$).

\begin{definition}\label{cong}
 Two type $\tau$ weak (resp.\ strong) $\mathbb{H}\text{-matroids}$ $\Gamma_{1}$ and $\Gamma_{2}$ 
 with underlying matroid $M$ 
 are \emph{$\cong_{\tau}\text{-equivalent}$} 
 (denoted by $\Gamma_{1}\cong_{\tau}\Gamma_{2}$) if there are
 $(\gamma_{1},\delta_{1})\in\Gamma_{1}$ and $(\gamma_{2},\delta_{2})\in\Gamma_{2}$ such that
 $(\gamma_{1},\delta_{1})\approx_{\tau}(\gamma_{2},\delta_{2})$.
\end{definition}

Again, $\cong_{\tau}$ is obviously an equivalence relation. We conclude  with the definition of the space of rescaling classes of 
type $\tau$ weak (resp.\ strong) $\mathbb{H}\text{-matroids}$ with underlying matroid $M$.

\begin{definition}\label{TypeC}
 The space of 
 rescaling classes 
 of type $\tau$ weak (resp.\ strong)
 $\mathbb{H}\text{-matroids}$ with underlying matroid $M$
 is the set 
 $\mathcal{R}_{\mathbb{H}}^{\tau,w}(M):=\mathcal{M}^{\tau,w}_{\mathbb{H}}(M)/\cong_{\tau}$
 (resp. $\mathcal{R}_{\mathbb{H}}^{\tau,s}(M):=\mathcal{M}^{\tau,s}_{\mathbb{H}}(M)/\cong_{\tau}$), which we endow with the quotient topology if $\mathbb{H}$ is a topological hyperfield.
\end{definition}

\begin{remark}\label{noinvolution1}
 If $\tau_{1}$ and $\tau_{2}$ are involutions of the hyperfield $\mathbb{H}$, there is a natural bijection between the sets 
 $\mathcal{N}_{\mathbb{H}}^{\tau_{1},w}(M)$ and $\mathcal{N}_{\mathbb{H}}^{\tau_{2},w}(M)$ via the map 
 \begin{equation*}
  f_{\tau_{1},\tau_{2}}:\mathcal{N}_{\mathbb{H}}^{\tau_{1},w}(M)\longrightarrow
                        \mathcal{N}_{\mathbb{H}}^{\tau_{2},w}(M)
 \end{equation*}
 that associates to a pair $(\gamma,\delta)$ of $\mathcal{N}_{\mathbb{H}}^{\tau_{1},w}(M)$ the pair 
 $(\tilde{\gamma},\tilde{\delta})$ of $\mathcal{N}_{\mathbb{H}}^{\tau_{2},w}(M)$ defined by 
 \begin{itemize}
  \item $\tilde{\gamma}C(x)=\gamma C(x)$ for any $C\in\mathfrak{C}$ and for any $x\in C$;
  \item $\tilde{\delta}D(y)=\tau_{2}\circ\tau_{1}(\delta D(y))$ for any $D\in\mathfrak{C}^{\ast}$ and for any $y\in D$.
 \end{itemize}
 Moreover, a straightforward check of definitions shows that $f_{\tau_1,\tau_2}$ passes to the quotients and hence induces bijections $F_{\tau_{1},\tau_{2}}$ and $\overline{F}_{\tau_{1},\tau_{2}}$ as in the following commutative diagram.
 \begin{equation*}
  \begin{tikzpicture}[y=2em]
         \node (E) at (0,4) {$\mathcal{N}^{\tau_{1},w}_{\mathbb{H}}(M)$};
         \node (F) at (6,4) {$\mathcal{N}^{\tau_{2},w}_{\mathbb{H}}(M)$};
         \node (A) at (0,2) {$\mathcal{M}^{\tau_{1},w}_{\mathbb{H}}(M)$};
         \node (B) at (6,2) {$\mathcal{M}^{\tau_{2},w}_{\mathbb{H}}(M)$};
         \node (C) at (0,0) {$\mathcal{R}^{\tau_{1},w}_{\mathbb{H}}(M)$};
         \node (D) at (6,0) {$\mathcal{R}^{\tau_{2},w}_{\mathbb{H}}(M)$};
         \draw[<->] (E) -- (F);
         \draw[<->] (A) -- (B);
         \draw[<->] (C) -- (D);
         \draw[->]  (E) -- (A);
         \draw[->]  (F) -- (B);
         \draw[->]  (A) -- (C);
         \draw[->]  (B) -- (D);
         \coordinate [label=left: {$\sim_{\tau_{1}}$}]           (a) at (0,3);
         \coordinate [label=right:{$\sim_{\tau_{2}}$}]           (b) at (6,3);         
         \coordinate [label=left: {$\cong_{\tau_{1}}$}]          (a) at (0,1);
         \coordinate [label=right:{$\cong_{\tau_{2}}$}]          (b) at (6,1);
         \coordinate [label=above:{$\overline{F}_{\tau_{1},\tau_{2}}$}] (c) at (3,0);
         \coordinate [label=above:{$F_{\tau_{1},\tau_{2}}$}]            (d) at (3,2);
         \coordinate [label=above:{$f_{\tau_{1},\tau_{2}}$}]            (d) at (3,4);
  \end{tikzpicture}
 \end{equation*}
The same holds in the strong case.
\end{remark}

\begin{remark}
If $\mathbb H$ is a topological hyperfield (see Definition \ref{def:toph}), the function $f_{\tau_1,\tau_2}$ of Remark \ref{noinvolution2} is indeed a homeomorphism and -- since the other spaces carry the quotient topology -- so are the induced maps $F_{\tau_1,\tau_2}$ and $\overline{F}_{\tau_1,\tau_2}$.
\end{remark}

\begin{proof}[Proof of Proposition \ref{laborious}]
 We prove the statement for weak $\mathbb H$-matroids. The strong case can be treated analogously. The proof has three steps: we first show that the function 
$(\gamma,\delta)\mapsto \mathcal C_{(\gamma,\delta)}$
 is well-defined from $\mathcal N_{\mathbb H} ^{\tau,w}$ to the set of weak $\mathbb H^*$-circuit sets with underlying matroid $M$. Then, we show that this function is surjective and that it factors through the equivalence relation $\sim_\tau$.
 
 \noindent{\em  (i) The function $(\gamma,\delta)\mapsto \mathcal C_{(\gamma,\delta)}$ is well-defined.
 }
Let $(\gamma,\delta)\in\mathcal{N}_{\mathbb{H}}^{\tau,w}(M)$ and consider the set 
 \begin{equation*}\label{cocircuitfrompairs}
  \mathcal{D}_{(\gamma,\delta)}=\left\lbrace 
                                 Y\in\mathbb{H}^{E}
                                 \left|
                                  \exists D\in\mathfrak{C}^{\ast},\exists b\in\mathbb{H}^{\ast}\text{ with }
                                  Y(j)=\left\lbrace
                                  \begin{array}{lll}
                                     0                 & \text{if} & j\notin D \\
                                     b\odot\delta D(j) & \text{if} & j\in D \\
                                  \end{array}
                                 \right.
                                \right.
                               \right\rbrace.
 \end{equation*}
 By Theorem \ref{TheoremC}, it is enough to prove $\mathcal C_{(\gamma,\delta)}\perp_\tau\mathcal D_{(\gamma,\delta)}$, which we now do. Consider $X\in \mathcal C_{(\gamma,\delta)}$ and $Y\in \mathcal D_{(\gamma,\delta)}$. By definition, there is a circuit $C\in\mathfrak{C}$ and a cocircuit $D\in\mathfrak{C}^{\ast}$ as well as  elements $a$, $b\in\mathbb{H}^{\ast}$ such that:
\begin{equation}\label{eq:ort}
\mathop{\mathlarger{\mathlarger{\mathlarger{\boxplus}}}}_{i\in E}X(i)\odot \tau(Y(i))
= a\odot\tau(b)\mathop{\mathlarger{\mathlarger{\mathlarger{\boxplus}}}}_{i\in C\cap D}\gamma C(i)\odot\tau(\delta D(i))
\end{equation}
(using the distributive law). Since $a,b\neq 0$, this implies that 
$X\bot_{\tau}Y$ if and only if $\gamma C\perp_\tau \delta D$ -- but the latter is guaranteed because $(\gamma,\delta)$ is in fact a dual pair.

\noindent{\em (ii) The function $(\gamma,\delta)\mapsto \mathcal C_{(\gamma,\delta)}$ is surjective.} 

 Let us denote by $\mathcal{C}$ and $\mathcal{D}$ the set of 
 weak $\mathbb{H}\text{-circuits}$, resp.\ weak $\mathbb H$-cocircuits, of a weak $\mathbb H$-matroid $\mathcal{M}$ on the ground set $E$ with respect to an involution $\tau$. For every circuit $C$ of the underlying matroid of $\mathcal M$ choose an element $i_C\in C$, and 
 consider the function 
 \begin{equation*}
  \gamma_{(\mathcal{C},\mathcal{D})}C:C\longrightarrow\mathbb{H}^{\ast},
 \quad\quad   i\mapsto\frac{X(i)}{X(i_{C})}
 \end{equation*}
 where $X\in\mathcal{C}$ is any weak $\mathbb{H}\text{-circuit}$ of $\mathcal{M}$ such that 
 $\operatorname{supp}(X)=C$ (this is well-defined by Axiom \ref{CC3}).  
 In the same way, choose an element $j_D\in D$ for every cocircuit $D$ of the underlying matroid and consider the function 
  \begin{equation*}
  \delta_{(\mathcal{C},\mathcal{D})}D:D\longrightarrow\mathbb{H}^{\ast},
\quad\quad
  j\mapsto\frac{Y(j)}{Y(j_{D})},
 \end{equation*}
 where $Y\in\mathcal{D}$ is any weak $\mathbb{H}\text{-cocircuit}$ of $\mathcal{M}$ with respect to $\tau$ such that  
 $\operatorname{supp}(Y)=D$.
 The same computation as in Equation \eqref{eq:ort} above, using distributivity, shows that since $\mathcal{C}$, $\mathcal{D}$ is a weakly dual pair with respect to $\tau$, the pair 
 $(\gamma_{(\mathcal{C},\mathcal{D})},\delta_{(\mathcal{C},\mathcal{D})})$ belongs to $\mathcal{N}_{\mathbb{H}}^{\tau,w}(M)$.
 Since, by definition, $\mathcal{C}=\mathcal{C}_{(\gamma_{(\mathcal{C},\mathcal{D})},\delta_{(\mathcal{C},\mathcal{D})})}$, we conclude that the map 
 $(\gamma,\delta)\mapsto\mathcal{C}_{(\gamma,\delta)}$ is  surjective.

 \noindent{\em (iii) We have $ (\gamma_{1},\delta_{1})\sim_{\tau}(\gamma_{2},\delta_{2})$ if and only if $
  \mathcal{C}_{(\gamma_{1},\delta_{1})}=\mathcal{C}_{(\gamma_{2},\delta_{2})}.$}
 The left-to-right implication is an immediate consequence of Definition \ref{sim}. 
 For the right-to-left implication, let $(\gamma_{1},\delta_{1}),(\gamma_{2},\delta_{2})\in\mathcal{N}_{\mathbb{H}}^{\tau,w}(M)$  such that
 $\mathcal{C}_{(\gamma_{1},\delta_{1})}=\mathcal{C}_{(\gamma_{2},\delta_{2})}$. 
 By definition, this means that for every circuit $C\in\mathfrak{C}$ 
 there exists $b_{C}\in\mathbb{H}^{\ast}$ with $\gamma_{1}C(x)=b_{C}\odot\gamma_{2}C(x)$ for all $x\in C$.

 On the other hand, let $\mathcal{D}_{(\gamma_{1},\delta_{1})}$ and $\mathcal{D}_{(\gamma_{2},\delta_{2})}$ be the sets defined 
 in (i) above. Since $(\gamma_{1},\delta_{1})$ and $(\gamma_{2},\delta_{2})$ belong to 
 $\mathcal{N}_{\mathbb{H}}^{\tau,w}(M)$ in the same way as in  (i) above we obtain that 
 $\mathcal{C}_{(\gamma_{1},\delta_{1})}$, $\mathcal{D}_{(\gamma_{1},\delta_{1})}$ and 
 $\mathcal{C}_{(\gamma_{2},\delta_{2})}$, $\mathcal{D}_{(\gamma_{2},\delta_{2})}$ are dual pairs with respect to $\tau$.
 From $\mathcal{C}_{(\gamma_{1},\delta_{1})}=\mathcal{C}_{(\gamma_{2},\delta_{2})}$ we deduce that 
 $\mathcal{D}_{(\gamma_{1},\delta_{1})}=\mathcal{D}_{(\gamma_{2},\delta_{2})}$. 
Unwrapping the definitions, we then see that for any cocircuit 
 $D\in\mathfrak{C}^{\ast}$ there exists $l_{D}\in\mathfrak{C}^{\ast}$ with $\delta_{1}D(y)=l_{D}\odot\delta_{2}D(y)$ for all 
 $y\in D$. This concludes the proof.
\end{proof}

\subsection{From cryptomorphisms to a one-to-one correspondence}\label{section2c}

In this section we show that the cryptomorphisms of 
Theorem \ref{TheoremA} and Theorem \ref{TheoremC} induce natural correspondences between the spaces of $\mathbb H$-matroids, resp.\ of rescaling classes.

\begin{proposition}\label{common}
 Given a matroid $M$ and a hyperfield $\mathbb{H}$ with an involution $\tau$ there exist bijections $\Upsilon_{\tau}$ and $\overline{\Upsilon_{\tau}}$ that fit into the following commutative square.
 \begin{equation}
    \label{FirstDiagram}
        \begin{tikzpicture}[y=2em]
         \node (A) at (0,2) {$\mathcal{M}^{\tau,w}_{\mathbb{H}}(M)$};
         \node (B) at (3,2) {$\mathcal{M}^{p,w}_{\mathbb{H}}(M)$};
         \node (C) at (0,0) {$\mathcal{R}^{\tau,w}_{\mathbb{H}}(M)$};
         \node (D) at (3,0) {$\mathcal{R}^{p,w}_{\mathbb{H}}(M)$};
         \draw[->] (A) -- (B);
         \draw[->] (A) -- (C);
         \draw[->] (C) -- (D);
         \draw[->] (B) -- (D);
         \coordinate [label=left: {$\cong_{\tau}$}]        (a) at (0,1);
         \coordinate [label=right:{$\cong_{p}$}]                  (b) at (3,1);
         \coordinate [label=above:{$\overline{\Upsilon}_{\tau}$}] (c) at (1.5,0);
         \coordinate [label=above:{$\Upsilon_{\tau}$}]            (d) at (1.5,2);
   \end{tikzpicture}
       \end{equation}
   The same holds for the strong case. In both cases, if $\mathbb{H}$ is a topological hyperfield, then ${\Upsilon}_{\tau}$ and $\overline{\Upsilon}_{\tau}$ are homeomorphisms.
\end{proposition}

We postpone the proof of this proposition in order to conclude our train of thoughts.

\begin{remark}[Compare Remark \ref{noinvolution1}]\label{noinvolution2}
 The spaces $\mathcal{M}^{p,w}_{\mathbb{H}}(M)$ and $\mathcal{R}^{p,w}_{\mathbb{H}}(M)$ 
do not depend on the involution $\tau$.
Thus, Proposition \ref{common} implies that, given involutions $\tau_{1}$ and $\tau_{2}$ of the hyperfield $\mathbb{H}$, the 
spaces $\mathcal{M}^{\tau_{1},w}_{\mathbb{H}}(M)$ and $\mathcal{M}^{\tau_{2},w}_{\mathbb{H}}(M)$, as well as 
$\mathcal{R}^{\tau_{1},w}_{\mathbb{H}}(M)$ and $\mathcal{R}^{\tau_{2},w}_{\mathbb{H}}(M)$,
are in one-to-one correspondence -- and in the case of topological hyperfields these bijections are homeomorphisms. The same holds for the strong case.
\end{remark}

In summary, Proposition \ref{common} allows us to freely switch (up to homeomorphism) between the point of
view of Grassmann--Pl\"{u}cker functions
and that of $\mathbb{H}\text{-circuit}$ and $\mathbb{H}\text{-cocircuit}$ signatures. 
We are now able to define the spaces we will study.

\begin{definition}\label{def:MRhyperfield}
The space $\mathcal{M}_{\mathbb{H}}^{w}(M)$ 
of weak $\mathbb{H}\text{-matroids}$ 
with underlying matroid $M$ is (either of the) the space(s) 
of points identified by the bijection $\Upsilon_{\tau}$ above.

 The space $\mathcal{R}_{\mathbb{H}}^{w}(M)$ 
of rescaling classes of weak $\mathbb{H}\text{-matroids}$ with underlying matroid $M$ 
is the space of points identified by the bijection $\overline{\Upsilon}_{\tau}$ above.

We define similarly the space $\mathcal{M}_{\mathbb{H}}^{s}(M)$ 
of strong $\mathbb{H}\text{-matroids}$ 
with underlying matroid $M$ and the space $\mathcal{R}_{\mathbb{H}}^{s}(M)$ of rescaling classes of 
strong $\mathbb{H}\text{-matroids}$ with underlying matroid $M$.

\end{definition}

\begin{remark}$\,$\label{rem:Gdn}
\begin{itemize}
\item 
{\em On the topological case.} When $\mathbb H$ is a topological hyperfield, Proposition \ref{common} shows that the spaces $\mathcal{M}_{\mathbb{H}}^{s}(M)$, $\mathcal{M}_{\mathbb{H}}^{w}(M)$, $\mathcal{R}_{\mathbb{H}}^{s}(M)$, $\mathcal{R}_{\mathbb{H}}^{w }(M)$ are well-defined up to homeomorphism.
\item 
{\em On realization spaces and hyperfield Grassmannians.} The (topological) spaces $\mathcal{M}_{\mathbb{H}}^{s}(M)$ and $\mathcal{M}_{\mathbb{H}}^{w}(M)$ are called ``realization spaces'' of $M$ over $\mathbb H$ by Anderson and Davis in their work on hyperfield Grassmannians. In this sense, we address their topology via the study of their quotients $\mathcal{R}_{\mathbb{H}}^{s}(M)$, $\mathcal{R}_{\mathbb{H}}^{w }(M)$. Notice that the fiber of the quotient map is $(\mathbb H^* )^{\vert E \vert -1}$ (see e.g.\ \S\ref{section2a}).
\end{itemize}
\end{remark}

\begin{proof}[Proof of Proposition \ref{common}]
Again we only prove the weak case, leaving to the reader the analogous proof of the strong case. 

Choose an ordered basis $B$ of the underlying matroid and consider the function that assigns to any pair $(\gamma,\delta)\in\mathcal{N}_{\mathbb{H}}^{\tau,w}(M)$ a weak 
Grassmann--Pl\"{u}cker function $\varphi_{(\gamma,\delta)}$ with $\varphi_{(\gamma,\delta)}(B)=1$ 
and satisfying
\begin{equation}
  \label{eq:noeuva}
  \frac{\gamma C(x_{i})}{\gamma C(x_{0})} = (-1)^{i}\odot
  \frac{\varphi_{(\gamma,\delta)}(x_{0},\ldots,\widehat{x}_{i},\ldots,x_{d})}{\varphi_{(\gamma,\delta)}(x_{1},\ldots,x_{d})}
\end{equation}
for all circuits $C$ of $M$, all $x_{0}\in C$ and all $i = 0,\ldots,d$, where
$x_{1},\ldots,x_{d}$ is any ordered
basis containing $C\setminus \{x_{0}\}$.  
With Theorem \ref{TheoremA} and Theorem \ref{TheoremC},
the assignment $(\gamma,\delta)\mapsto\varphi_{(\gamma,\delta)}$ induces the desired bijection 
$ \Upsilon_{\tau}:\mathcal{M}_{\mathbb{H}}^{\tau,w}(M)\rightarrow\mathcal{M}_{\mathbb{H}}^{p,w}(M).$

Now in order to obtain the commutative diagram \eqref{FirstDiagram} it is enough to prove the following.

\noindent {\em Claim.} $ \Gamma_{1}\cong_{\tau}\Gamma_{2}$
 if and only if 
 $\Upsilon_{\tau}(\Gamma_{1})\cong_{p}\Upsilon_{\tau}(\Gamma_{2})$.

\noindent{\em Proof of claim.} For the left-to-right direction  assume $\Gamma_{1}\cong_{\tau}\Gamma_{2}$. Thus, there exist representatives 
$(\gamma_{1},\delta_{1})\in\Gamma_{1}$ and $(\gamma_{2},\delta_{2})\in\Gamma_{2}$ and a function 
$h:E\longrightarrow\mathbb{H}^{\ast}$ satisfying the identities listed in Definition \ref{def_ssim}.  
It is enough to prove that, given any representative $\varphi_{2}\in\Upsilon_{\tau}(\Gamma_{2})$, the 
Grassmann--Pl\"{u}cker function $\varphi_{h}$ defined by 
\begin{equation}
 \varphi_{h}(x_{1},\ldots,x_{d})= \left(
                                  \bigodot_{j=1}^{d}h(x_{j})^{-1}
                                 \right)
                                  \odot\varphi_{2}(x_{1},\ldots,x_{d})
\end{equation}               
belongs to $\Upsilon_{\tau}(\Gamma_{1})$, which amounts to the straightforward check that \eqref{eq:noeuva} is satisfied with $\gamma_{1}$ 
on the left-hand side and $\varphi_{h}$ on the right-hand side.

For the right-to-left direction, assume $\Upsilon_{\tau}(\Gamma_{1})\cong_{p}\Upsilon_{\tau}(\Gamma_{2})$. This means that there 
exist representatives $\varphi_{1}\in\Upsilon_{\tau}(\Gamma_{1})$ and $\varphi_{2}\in\Upsilon_{\tau}(\Gamma_{2})$ as well as some function $h:E\longrightarrow\mathbb{H}^{\ast}$ satisfying the identity of Definition \ref{def_GPssim}. 

Let us then consider a representative $(\gamma_{2},\delta_{2})$ of $\Gamma_{2}$. It is enough to prove that the pair 
$(\gamma_{h},\delta_{h})$ of $\mathbb{H}\text{-circuit}$ and $\mathbb{H}\text{-cocircuit}$ signatures of $M$ 
that are orthogonal with respect to $\tau$ defined by 
\begin{itemize}
 \item $\gamma_{h}C(x)=h^{-1}(x)\odot\gamma_{2}C(x)$ for any circuit $C\in\mathfrak{C}$ and any $x\in C$;
 \item $\delta_{h}D(y)=\tau(h(x))\odot\delta_{2}D(y)$ for any cocircuit $D\in\mathfrak{C}^{\ast}$ and any $y\in D$;
\end{itemize}
belongs to $\Gamma_{1}$. Again, this amounts to the checking that \eqref{eq:noeuva} is satisfied 
 with $\gamma_{h}$ on the left-hand side and $\varphi_{1}$ on the right-hand side as well as
 with $\delta_{h}$ on the left-hand side and the dual GP-function $\varphi_{1}^{(\tau)}$ (defined in  Theorem \ref{Duality}) on the right-hand side.

For the topological claim it is enough to check that $\Upsilon_{\tau}$ and its inverse are continuous. This is the case because they are induced in the quotient topology from the functions given by the explicit form in Equation \eqref{eq:noeuva}. These are continuous since $\mathbb H^*$ is a topological group, hence multiplication and inversion are continuous functions.
\end{proof}

\section{Projective classes of matroids over hyperfields}\label{section3}

In this section we define weak (resp.\ strong) $\mathbb{H}\text{-projective}$ classes of a matroid $M$ in terms of circuits and cocircuits.
We then prove that, under certain conditions, there exists a one-to-one correspondence 
between $\mathcal{R}_{\mathbb{H}}^{w}(M)$ (resp. $\mathcal{R}_{\mathbb{H}}^{s}(M)$) 
and the space of weak (resp.\ strong) $\mathbb{H}\text{-projective}$ classes of $M$.

\subsection{Definition of projective classes}\label{section3a}
Generalizing the constructions and definitions of \cite{DS15, GRS95} to the context
of weak (resp.\ strong) $\mathbb{H}\text{-matroids}$ we now introduce the notion of $\mathbb{H}\text{-projective}$ class 
of a matroid.

Let $M$ be a matroid. Given a pair $(\gamma,\delta)$ of $\mathbb{H}\text{-circuit}$ and 
cocircuit signatures of $M$ 
that are weak (resp.\ strong) orthogonal with respect to $\tau$ (compare Section \ref{section2b})
and for any circuit $C$ and cocircuit $D$ of $M$ with $x$, $y\in C\cap D$ set  
\begin{equation}\label{base}
 \begin{pmatrix}
  C & D \\
  x & y \\
 \end{pmatrix}:=\frac{\gamma C(x)\odot\tau(\delta D(x))}{\gamma C(y)\odot\tau(\delta D(y))}.
\end{equation}

The values   
$\left(\begin{smallmatrix}
        C & D \\
        x & y \\
       \end{smallmatrix}
  \right)$ depend only on the $\sim_{\tau}$ and $\approx_{\tau}$ equivalence class of $(\gamma,\delta)$
(see Definition \ref{sim} and Definition \ref{cong}).
These values satisfy the following relations whenever all terms are defined:
  \begin{itemize}
  \item[] 
  \begin{equation}
         \label{PO1}
    \begin{pmatrix}
      C & D \\
      x & x
    \end{pmatrix}
    =1;
  \end{equation}

  \item[]
  \begin{equation}
    \label{PO2}
    \begin{pmatrix}
      C & D \\
      x & y
    \end{pmatrix}\odot
    \begin{pmatrix}
      C & D \\
      y & z
    \end{pmatrix}\odot
    \begin{pmatrix}
      C & D \\
      z & x
    \end{pmatrix}
    =1;
  \end{equation}

  \item[]
  \begin{equation}
    \label{PO3}
    \begin{pmatrix}
      C_{1} & D_{1} \\
      x & y
    \end{pmatrix}\odot
    \begin{pmatrix}
      C_{2} & D_{2} \\
      x & y
    \end{pmatrix}=
    \begin{pmatrix}
      C_{1} & D_{2} \\
      x & y
    \end{pmatrix}\odot
    \begin{pmatrix}
      C_{2} & D_{1} \\
      x & y
    \end{pmatrix};
  \end{equation}
  \item []
  \begin{equation}
    \label{PO4}
    0\in\mathop{\mathlarger{\mathlarger{\mathlarger{\boxplus}}}}_{x\in C\cap D}
          \begin{pmatrix}
            C & D \\
            x & y
          \end{pmatrix}
  \end{equation}
  for all $C\in\mathfrak{C}$, $D\in\mathfrak{C}^{\ast}$ with $|C\cap D|\leq3$ (resp. any $C$ and $D$ in the strong case) and $y\in C\cap D$.  
  \end{itemize}

These properties we will now take as axioms for the abstract definition of projective classes.

\begin{definition}[Projective classes]\label{HPclass}
Let $M$ be a matroid, and write
 \begin{equation*}
   Q_{M}:= \left\lbrace(C,D,x,y)\in\mathfrak{C}\times\mathfrak{C}^{\ast}\times E\times E\left|
       \begin{array}{c}
           C\cap D\neq\emptyset\\
           x,y\in C\cap D\\ 
       \end{array}
       \right.
       \right\rbrace,\quad d_{M}:= \vert Q_{M}\vert,
 \end{equation*}
where $\mathfrak{C}$ and $\mathfrak{C}^{\ast}$ denote the circuit resp.\ cocircuit set of $M$.

 A \emph{weak} (resp. \emph{strong}) $\mathbb{H}\text{-projective}$ class  
 of a matroid $M$ is a function
 $\Pmap_{\mathbb{H}}^{w}:Q_{M}\longrightarrow\mathbb{H}^{\ast}$ 
 whose values, denoted by 
 $\left(\begin{smallmatrix}
         C & D \\
         x & y \\
        \end{smallmatrix}
   \right)$ 
 as a shorthand for $\Pmap_{\mathbb{H}}^{w}(C,D,x,y)$, satisfy conditions
 \eqref{PO1}, \eqref{PO2}, \eqref{PO3} and \eqref{PO4}. A strong $\mathbb H$-projective class is defined accordingly (with the corresponding version of axiom \eqref{PO4}) and denoted $\Pmap_{\mathbb{H}}^{s}$.
        
 \noindent      
 The set of weak (resp.\ strong) $\mathbb{H}\text{-projective}$ classes of $M$ will be denoted by
 $\mathcal{P}_{\mathbb{H}}^{w}(M)$ 
 (resp. $\mathcal{P}_{\mathbb{H}}^{s}(M)$). We will consider it as a subset of 
   $(\mathbb{H}^{\ast})^{d_{M}}$ (and topologized as such, if $\mathbb H$ is a topological hyperfield).
 \end{definition}

\begin{remark}
  When writing $\left(\begin{smallmatrix}
      C & D \\
      x & y \\
    \end{smallmatrix}\right)$ we always assume that
  $(C,D,x,y)\in Q_{M}$. 
\end{remark}

We now state for later reference a property of weak (resp.\ strong) $\mathbb{H}\text{-projective}$ classes.

\begin{proposition}
 Let $\Pmap_{\mathbb{H}}^{w}$ (resp. $\Pmap_{\mathbb{H}}^{s}$) 
 be a weak (resp.\ strong) $\mathbb{H}\text{-projective}$ class of a matroid $M$. Then,

 \begin{equation}
  \label{PO1b}
  \begin{pmatrix}
    C & D \\
    x & y \\
  \end{pmatrix}
  \odot
  \begin{pmatrix}
    C & D \\
    y & x \\
  \end{pmatrix}=1.
 \end{equation}
Moreover, for every 
$C\in\mathfrak{C}(M)$ and $D\in\mathfrak{C}^{\ast}(M)$ 
 such that $C\cap D=\{x,y\}$, we have 
 \begin{equation}
  \label{PO4b}
  \begin{pmatrix}
             C & D\\
             x & y
         \end{pmatrix}
         =-1.
 \end{equation}

\end{proposition}

\begin{proof} Equation \eqref{PO1b} follows from \eqref{PO1} and \eqref{PO2}.
 For \eqref{PO4b} notice that from \eqref{PO1} and \eqref{PO4} we can derive
 \begin{equation*}
   0\in\mathop{\mathlarger{\mathlarger{\mathlarger{\boxplus}}}}_{u\in\{x,y\}}
            \begin{pmatrix}
             C & D \\
             u & y \\
            \end{pmatrix}=
   1\boxplus\begin{pmatrix}
             C & D \\
             x & y \\
         \end{pmatrix}
 \end{equation*}
 and with uniqueness of inverses (H2) the claim follows. 
\end{proof}

\subsection{WAM hyperfields}\label{section3b}\footnote{The name WAM stands as an acronym for ``Whack-A-Mole'':
 The task of simultaneously satisfying all five 
 conditions in \eqref{confer} can feel like playing the well-known game.}

The aim of this section is to provide an algebraic characterization of those hyperfields for which 
there exists a one-to-one correspondence between the spaces of rescaling classes and of projective classes of matroids.

\begin{definition}[WAM hyperfields]\label{fertile}
 A hyperfield $\mathbb{H}$ is \emph{WAM} 
 if for any $f$, $g$, $\mu\in\mathbb{H}^{\ast}$ condition \eqref{confer} 
 below 
 implies $\mu=1$.
 \begin{equation}\label{confer}
  \tag{
  WAM}
  \left\lbrace
  \begin{array}{l}
   \mu\odot\mu=1                                   \\
   0\in1\boxplus f\boxplus g                       \\
   0\in1\boxplus f\boxplus (\mu\odot g)            \\
   0\in1\boxplus (\mu\odot f)\boxplus g            \\
   0\in1\boxplus (\mu\odot f)\boxplus (\mu\odot g) \\
  \end{array}
  \right.
 \end{equation}
\end{definition}

A straightforward check of definitions shows that the hyperfields listed in Example \ref{example1} 
are all WAM.
The following example shows that the class of non-WAM hyperfields is non-empty.

\begin{example}\label{nonWAM2}
 Let $\mathbb{S}_{L}$ be the multivalued algebraic structure defined on the set $\{0,1,-1\}$ with usual multiplication rule and 
 hyperaddition law given by
 \begin{equation*}
  \begin{array}{cccc}
   \toprule
   \boxplus & 0  & 1          & -1         \\
   \midrule
      0     & 0  & 1          & -1         \\
   \midrule 
      1     & 1  & \{1,-1\}   & \{0,1,-1\} \\
   \midrule 
     -1     & -1 & \{0,1,-1\} & \{1,-1\}   \\
   \bottomrule
  \end{array}
 \end{equation*}
 Case-by-case inspection confirms that $\mathbb{S}_{L}$ is a hyperfield.
 To prove that $\mathbb{S}_{L}$ is non-WAM it is enough to notice that
 \begin{equation*}
 (1,1,-1)\in
  \left\lbrace
  (f,g,\mu)\in(\mathbb{S}_{L}^{\ast})^{3}
  \left|
  \begin{array}{l}
   \mu\odot\mu=1                                   \\
   0\in1\boxplus f\boxplus g                       \\
   0\in1\boxplus f\boxplus (\mu\odot g)            \\
   0\in1\boxplus (\mu\odot f)\boxplus g            \\
   0\in1\boxplus (\mu\odot f)\boxplus (\mu\odot g) \\
  \end{array}
  \right.
  \right\rbrace.
 \end{equation*}
 \end{example}
 \begin{remark}
 Comparing all
 possible hyperfield structures on a set of two elements, we can deduce that 
 a non-WAM hyperfield must have at least $3$ elements. 
\end{remark}

\begin{theorem}\label{main}\todo{(C1) (C2) we already used!} 
 For a hyperfield $\mathbb{H}$ with an involution $\tau$ the following conditions are equivalent:
 \begin{enumerate}[label=(\roman{*})]
  \item \label{C1} $\mathbb{H}$ is WAM;
  \item \label{C2} For any matroid $M$ and for any weak (resp.\ strong) $\mathbb{H}\text{-projective}$ class 
        $\Pmap_{\mathbb{H}}^{w}$ (resp. $\Pmap_{\mathbb{H}}^{s}$) of $M$ there 
        exist $\mathbb{H}^{\ast}\text{-circuit}$ and $\mathbb{H}^{\ast}\text{-cocircuit}$ signatures $\gamma$ and 
        $\delta$ of $M$ that are weak (resp.\ strong) orthogonal with respect to $\tau$ and such that:
        \begin{enumerate}
         \item Identity \eqref{base}
               holds for every circuit $C\in\mathfrak{C}$ and cocircuit $D\in\mathfrak{C}^{\ast}$ of $M$ and all 
               $x$, $y\in C\cap D$;
         \item The $\cong_{\tau}\text{-equivalence}$ class of the weak (resp.\ strong) $\mathbb{H}\text{-matroid}$ 
               represented by $(\gamma,\delta)$ is uniquely 
               determined by the given weak (resp.\ strong) $\mathbb{H}\text{-projective}$ class $\Pmap_{\mathbb{H}}^{w}$
               (resp. $\Pmap_{\mathbb{H}}^{s}$).
        \end{enumerate}
 \end{enumerate} 
\end{theorem}

We postpone the proof after a couple of remarks.

\begin{remark}\label{map}
 As already remarked at the beginning of \S\ref{section3a}, the map $F^{w}$ from
 $\mathcal{N}_{\mathbb{H}}^{\tau,w}(M)$ to $\mathcal{P}_{\mathbb{H}}^{w}(M)$
 (resp. $F^{s}$ from $\mathcal{N}_{\mathbb{H}}^{\tau,s}(M)$ to $\mathcal{P}_{\mathbb{H}}^{s}(M)$) that 
 sends a pair $(\gamma,\delta)$ to the weak (resp.\ strong) $\mathbb{H}\text{-projective}$ class of $M$ defined by identity 
 \eqref{base} depends only on the  equivalence class of $(\gamma,\delta)$ with respect to $\sim_{\tau}$ and $\approx_{\tau}$.
 Hence, $F^{w}$ (resp. $F^{s}$) induces a quotient map
 $\mathcal{F}^{w}:\mathcal{R}_{\mathbb{H}}^{w}(M)\longrightarrow\mathcal{P}_{\mathbb{H}}^{w}(M)$ 
 (resp. $\mathcal{F}^{s}:\mathcal{R}_{\mathbb{H}}^{s}(M)\longrightarrow\mathcal{P}_{\mathbb{H}}^{s}(M)$).
 If the hyperfield $\mathbb{H}$ is WAM, Theorem \ref{main} implies that $\mathcal{F}^{w}$ (resp. $\mathcal{F}^{s}$) 
 is a one-to-one correspondence.
 \end{remark}

\begin{remark}[The topological case]
 When $\mathbb H$ is a topological hyperfield, $\mathcal{F}^{s}$ and $\mathcal{F}^{w}$ are homeomorphisms. In order to check this, notice first that continuity of $F^{w}$ and $F^{s}$ is evident from the explicit form of Equation \eqref{base}, and implies continuity of $\mathcal{F}^{w}$ and $\mathcal{F}^{s}$.  
 The continuity of the inverse functions can be checked by inspecting the proof of Theorem \ref{main} as follows: 
  via the explicit formulas given in \cite[Lemma 2.5 and Lemma 2.6]{GRS95} for the reduction steps of Lemma \ref{lemma25} and Lemma \ref{lemma26}, it is enough to check continuity of the inverse in the case when the underlying matroid is uniform of rank $2$ on $4$. This amounts to an inspection of the table in \eqref{eq:tab:U}, keeping in mind that the functions $\tau^{\pm 1}$ are continuous.
\end{remark}

The remainder of this section is devoted to the proof of Theorem \ref{main}.
First we need to check its statement for 
the uniform matroid of rank $2$ over $4$ elements (denoted by $U_{2}(4)$).

\begin{lemma}\label{keylemma}
 Theorem \ref{main} holds with $M=U_{2}(4)$.
\end{lemma}

\begin{proof}
\ref{C1} $\Longrightarrow$ \ref{C2}:\hfill\break Let us assume $\mathbb{H}$ is a WAM hyperfield. The uniqueness part of the claim 
 follows from a generalization of the arguments of \cite{GRS95} to our context. Hence, it suffices to 
 prove the existence of the desired pair $\gamma$, $\delta$. Let $\Pmap_{\mathbb{H}}^{w}$ 
 (resp. $\Pmap_{\mathbb{H}}^{s}$)
 be a weak (resp.\ strong) $\mathbb{H}\text{-projective}$ class of 
 $U_{2}(4)$. From this, we want to construct a $\mathbb{H}^{\ast}\text{-circuit}$ signature $\gamma$ and a 
 $\mathbb{H}^{\ast}\text{-cocircuit}$ signature $\delta$ that are mutually weak (resp.\ strong) orthogonal and such that \eqref{base} 
 holds for all possible arguments. 
 Write $\mathfrak{C}=\{C_{1},C_{2},C_{3},C_{4}\}$ and $\mathfrak{C}^{\ast}=\{D_{1},D_{2},D_{3},D_{4}\}$ where 
 $C_{i}=D_{i}=[4]\setminus\{i\}$, 
 $1\leq i\leq4$. Then $i\neq j$ implies $|C_{i}\cap D_{j}|=2$. From \eqref{PO4b} we have 
 \begin{equation*}
  \begin{pmatrix}
             C_{\sigma(1)} & D_{\sigma(2)}   \\
                \sigma(3)  &    \sigma(4)    \\
  \end{pmatrix}=-1 
 \end{equation*}
 for each permutation $\sigma\in\mathcal{S}_{4}$.
 Hence, if $\{i,j,k,l\}=\{1,2,3,4\}$ using \eqref{PO3} we find 
 \begin{equation}\label{inv}
  \begin{pmatrix}
   C_{i} & D_{i} \\
   j     & l     \\
  \end{pmatrix}=
  \begin{pmatrix}
   C_{k} & D_{k} \\
   j     & l     \\
  \end{pmatrix}^{-1}
\end{equation}
 so that, from \eqref{PO1b} and the previous observations, the only elements of interest are 
 \begin{equation*}
   \begin{array}{lllllll}
                          \begin{pmatrix}
                              C_{1} & D_{1}\\
                                  2 & 3
                          \end{pmatrix}=:a; & \quad & \quad & 
                          \begin{pmatrix}
                              C_{1} & D_{1}\\
                                  3 & 4
                          \end{pmatrix}=:b; & \quad & \quad&
                          \begin{pmatrix}
                              C_{2} & D_{2}\\
                                  1 & 3
                          \end{pmatrix}=:c. \\
   \end{array}
\end{equation*}
A case-by-case inspection shows that
\begin{equation}\label{mu}
\left(a^{-1}\odot b^{-1}\odot c\right)^{2}=1.
\end{equation}
Now, set $f=b$, $g=a\odot b$ and $\mu=a^{-1}\odot b^{-1}\odot c$. Thus, we can rewrite \eqref{mu} as
\begin{equation}\label{MU}
 \mu\odot\mu=1.
\end{equation}
After some computations, we obtain
\begin{equation}
 \label{abcd}
 \begin{array}{ll}
                     \begin{pmatrix}
                         C_{1} & D_{1}\\
                             3 & 4
                     \end{pmatrix}=f; &
                     \begin{pmatrix}
                         C_{1} & D_{1}\\
                             2 & 4
                     \end{pmatrix}=g;\\
                     \quad & \quad\\
                     \begin{pmatrix}
                        C_{2} & D_{2}\\
                            4 & 3
                     \end{pmatrix}=f; &
                     \begin{pmatrix}
                        C_{2} & D_{2}\\
                            1 & 3
                     \end{pmatrix}=\mu\odot g;\\
                     \quad & \quad\\
                     \begin{pmatrix}
                        C_{3} & D_{3}\\
                            1 & 2
                     \end{pmatrix}=\mu\odot f; &
                     \begin{pmatrix}
                        C_{3} & D_{3}\\
                            4 & 2
                     \end{pmatrix}=g;\\
                     \quad & \quad\\
                     \begin{pmatrix}
                        C_{4} & D_{4}\\
                            2 & 1
                     \end{pmatrix}=\mu\odot f; &
                     \begin{pmatrix}
                        C_{4} & D_{4}\\
                            3 & 1
                     \end{pmatrix}=\mu\odot g.\\  
 \end{array}
\end{equation}
Using \eqref{PO4}, from \eqref{abcd} we deduce that the following conditions are simultaneously fulfilled:
\begin{equation}\label{multiconv}
 \begin{array}{ll}
  0\in1\boxplus f\boxplus g;            & 0\in1\boxplus f\boxplus(\mu\odot g);            \\
  0\in1\boxplus (\mu\odot f)\boxplus g; & 0\in1\boxplus (\mu\odot f)\boxplus(\mu\odot g). \\
 \end{array}
\end{equation}
Since $\mathbb{H}$ is a WAM hyperfield, equations \eqref{MU} and \eqref{multiconv} imply $\mu=1$. In particular, 
\eqref{multiconv} reduces to 
\begin{equation*}
 0\in1\boxplus f\boxplus g.
\end{equation*}
Now, we are ready to construct the desired pair $\gamma$, $\delta$. By direct computation (recall Proposition \ref{algbas})
one can check that the set of circuit and cocircuit signatures defined as the columns of the following tables satisfy \eqref{base}.
\begin{equation}\label{eq:tab:U}
\begin{array}{llll}
   \begin{array}{ccccc}
  \toprule
  E & C_{1} & C_{2}  & C_{3} & C_{4}\\
  \midrule
  1 & \bullet  &    1     &     -1      &     1    \\
  2 &      1   &  \bullet &      1      &     f    \\
  3 &      1   &    1     & \bullet     &    -g    \\
  4 &      1   &   -f     &     -g      & \bullet   \\
  \bottomrule
 \end{array} & \quad & \quad &
 \begin{array}{ccccc}
   \toprule
  E & D_{1} & D_{2}  & D_{3} & D_{4}\\
  \midrule
  1 &  \bullet &  \tau(g)  &  -\tau(f) &  1         \\
  2 &  \tau(g) &  \bullet  &        1  &  1         \\
  3 &  \tau(f) &  1        &  \bullet  &  -1        \\
  4 &  1       &  -1       &        -1 &  \bullet   \\
  \bottomrule
 \end{array}\\
\end{array}
\end{equation}

\noindent \ref{C2} $\Longrightarrow$ \ref{C1}:\hfill\break 
Let us suppose that condition \ref{C2} is fulfilled. Let $f$, $g$, 
$\mu\in\mathbb{H}^{\ast}$ such that \eqref{confer} holds. We want to show that $\mu=1$. Let $\Pmap_{\mathbb{H},f,g,\mu}$ 
be the map from $Q_{U_{2}(4)}$ to $\mathbb{H}^{\ast}$ defined in the following way:
\begin{enumerate}
 \item For $1\leq i,j\leq4$ and $k\in C_{i}\cap D_{j}$
       \begin{equation*}
        \Pmap_{\mathbb{H},f,g,\mu}(C_{i},D_{j},k,k)=1;      
       \end{equation*}
 \item For $1\leq i,j\leq4$, $i\neq j$, $k$, $l\in C_{i}\cap D_{j}$ and $k\neq l$
       \begin{equation*}
        \Pmap_{\mathbb{H},f,g,\mu}(C_{i},D_{j},k,l)=-1;      
       \end{equation*}       
 \item \begin{equation*}
        \begin{array}{l}
         \Pmap_{\mathbb{H},f,g,\mu}(C_{1},D_{1},3,4)=f; \\ 
         \Pmap_{\mathbb{H},f,g,\mu}(C_{1},D_{1},2,4)=g; \\
         \Pmap_{\mathbb{H},f,g,\mu}(C_{1},D_{1},3,2)=f\odot g^{-1}; \\

         \Pmap_{\mathbb{H},f,g,\mu}(C_{2},D_{2},4,3)=f; \\ 
         \Pmap_{\mathbb{H},f,g,\mu}(C_{2},D_{2},1,3)=\mu\odot g; \\
         \Pmap_{\mathbb{H},f,g,\mu}(C_{2},D_{2},4,1)=\mu^{-1}\odot f\odot g^{-1}; \\

         \Pmap_{\mathbb{H},f,g,\mu}(C_{3},D_{3},1,2)=\mu\odot f; \\
         \Pmap_{\mathbb{H},f,g,\mu}(C_{3},D_{3},4,2)=g; \\
         \Pmap_{\mathbb{H},f,g,\mu}(C_{3},D_{3},1,4)=\mu\odot f\odot g^{-1}; \\

         \Pmap_{\mathbb{H},f,g,\mu}(C_{4},D_{4},2,1)=\mu\odot f; \\
         \Pmap_{\mathbb{H},f,g,\mu}(C_{4},D_{4},3,1)=\mu\odot g; \\
         \Pmap_{\mathbb{H},f,g,\mu}(C_{4},D_{4},2,3)=f\odot g^{-1}; \\
        \end{array}
       \end{equation*}
 \item For $1\leq i,j\leq4$ and $k$, $l\in C_{i}\cap D_{j}$, if $\Pmap_{\mathbb{H},f,g,\mu}(C_{i},D_{j},k,l)$ is given,
       then 
       \begin{equation*}
        \Pmap_{\mathbb{H},f,g,\mu}(C_{i},D_{j},l,k)=\Pmap_{\mathbb{H},f,g,\mu}(C_{i},D_{j},k,l)^{-1}.
       \end{equation*}
\end{enumerate}
A straightforward check of the definition of $\Pmap_{\mathbb{H},f,g,\mu}$ 
shows that conditions \eqref{PO1}, \eqref{PO2} and \eqref{PO3} are satisfied. Moreover, since \eqref{confer} 
holds for $f$, $g$, $\mu$ it is not hard to see that \eqref{PO4} is also fulfilled. 
Thus, $\Pmap_{\mathbb{H},f,g,\mu}$ is a weak (resp.\ strong) 
$\mathbb{H}\text{-projective}$ class of $U_{2}(4)$. 
By hypothesis \ref{C2}, there exist a 
$\mathbb{H}^{\ast}\text{-circuit}$ signature
$\gamma$ and a $\mathbb{H}^{\ast}\text{-cocircuit}$ signature $\delta$ of $M$
that are weak (resp.\ strong) orthogonal and such that \eqref{base} holds for 
all possible arguments.

Up to the equivalences $\sim_{\tau}$ and $\cong_{\tau}$ 
(compare Definition \ref{sim} and Definition \ref{cong}), we can assume that 
\begin{equation*}
 \begin{array}{llll}
  \gamma C_{1}(2)=1;  & \gamma C_{1}(3)=1;  & \gamma C_{1}(4)=1 & \gamma C_{2}(1)=1;   \\
  \gamma C_{3}(1)=-1; & \gamma C_{3}(2)=1;  & \gamma C_{4}(1)=1 & \delta D_{1}(4)=1;   \\
  \delta D_{2}(4)=-1; & \delta D_{3}(4)=-1; & \multicolumn{2}{l}{\delta D_{4}(3)=-1.}  \\
 \end{array}
\end{equation*}

Hence, from the definition of $\Pmap_{\mathbb{H},f,g,\mu}$ we can use \eqref{base} and Proposition \ref{algbas}
to compute the other values of $\gamma$ and $\delta$.
In particular, one can compute  
$\gamma C_{2}(4)=-f$ and $\delta D_{3}(1)=-\tau(\mu)\odot\tau(f)$ (cf. Section \ref{appendix1}).

By definition of $\Pmap_{\mathbb{H},f,g,\mu}$ we have
\begin{equation*}
 \Pmap_{\mathbb{H},f,g,\mu}(C_{2},D_{3},1,4)=-1.
\end{equation*}
On the other hand, since \eqref{base} holds for all possible arguments we find
\begin{equation*}
 \Pmap_{\mathbb{H},f,g,\mu}(C_{2},D_{3},1,4)=
 \frac{\gamma C_{2}(1)\odot\tau(\delta D_{3}(1))}{\gamma C_{2}(4)\odot\tau(\delta D_{3}(4))}=
 \frac{1\odot(-\mu\odot f)}{-f\odot-1}=-\mu.
\end{equation*}
So that $-\mu=-1$. Multiplying both sides by $-1$, from item \ref{A1} of Proposition \ref{algbas} we get $\mu=1$.
\end{proof}

\begin{remark}\label{noreductive}
 The arguments used in the proof of Lemma \ref{keylemma} hold for both the weak and the strong case. 
 Given any hyperfield $\mathbb{H}$ the spaces $\mathcal{P}_{\mathbb{H}}^{w}(U_{2}(4))$ and 
 $\mathcal{P}_{\mathbb{H}}^{s}(U_{2}(4))$ always coincide, since for any circuit $C$ and any cocircuit $D$ 
 of $U_{2}(4)$ we have $|C\cap D|\leq3$. 
 However, there exists matroids $M$ such that the spaces 
 $\mathcal{P}_{\mathbb{H}}^{w}(M)$ and $\mathcal{P}_{\mathbb{H}}^{s}(M)$ are different 
 (compare Remark \ref{map} and see \cite[Section 3.10]{BB16} for more details).
 This is because the proof of Theorem \ref{main} relies not only on the ``base case'' discussed in Lemma \ref{keylemma},
 but also on the extension arguments of the following Lemma \ref{lemma25} and Lemma \ref{lemma26}.
\end{remark}

\begin{proof}[Proof of Theorem \ref{main}]
 The 
 implication ``\ref{C2} $\Rightarrow$ \ref{C1}'' is a consequence of Lemma \ref{keylemma}.
 To prove the reverse implication, let us assume that $\mathbb{H}$ is a WAM hyperfield. We claim that 
 condition \ref{C2} is fulfilled. The uniqueness part of the statement can be proved by a straightforward generalization 
 of the arguments of \cite[Theorem 1]{GRS95}. To show existence, we need to state some lemmas that enable 
 us to reduce to the case of $U_{2}(4)$. Both can be proved with an easy extension of the arguments of 
 \cite[Lemma 2.5]{GRS95} and \cite[Lemma 2.6]{GRS95} to the hyperfield context.
 
 \begin{lemma}\label{lemma25}
  Let $a\in C$ be an element of a circuit $C\in\mathfrak{C}$ of $M$ with $|C|\leq2$. 
  If condition \ref{C2} holds for the restriction 
  $M'=M[E\setminus\{a\}]$, then it also holds for $M$.
 \end{lemma}
           
 \begin{lemma}\label{lemma26}
  Let $M$ be a connected matroid  without 
  parallel elements and such that $\operatorname{rk}(M^{\ast})>2$.
  If condition \ref{C2} holds for each proper minor of 
  $M$, then it also holds for $M$.
 \end{lemma}
 
 \noindent Now we are ready to complete the proof of Theorem
           \ref{main}. By way of contradiction let $M$ 
           be a matroid with the minimal number of elements, for which
           Theorem \ref{main} fails. 
           Clearly if the statement of Theorem \ref{main} is true for matroids $M_{1}$ and $M_{2}$, it holds for the 
           direct sum $M_{1}\oplus M_{2}$, too. We can thus assume $M$
           to be connected (otherwise there exists a proper
           connected component of $M$
           for which the claim also fails, contradicting the minimality of $M$).
           The minimality assumption and the self-duality of Theorem \ref{main} ensure that neither $M$ nor $M^*$ have parallel elements (Lemma \ref{lemma25}) as well as both $\operatorname{rk}(M^{\ast})\leq 2$ and $\operatorname{rk}(M)\leq 2$ (Lemma \ref{lemma26}). But then $M$ must be (a minor of) $U_2(4)$, violating Lemma \ref{keylemma}.

\end{proof}

\section{Tutte groups and spaces of rescaling classes}\label{section5}

The aim of this section is to  describe
the space of rescaling classes of matroids over hyperfields using inner Tutte groups (see Appendix, Subsection \ref{section4}).
We then focus on the weak case and we study several characterizations of the set of weak hyperfield rescaling classes of a matroid that can be obtained from  Theorem \ref{big} via the 
different presentations of the inner Tutte group provided by Theorem \ref{isoinnerTuttegroups}.
Lastly, we describe how such spaces behave under passing to a sub-hyperfield of a given hyperfield.

\subsection{Main result}\label{section5a}

As a counterpart of \cite[Theorem 6.1]{DW89} and \cite[Theorem 4.4]{DW90} in the context of 
matroids over hyperfield we now prove the following result. The main object here is the group $\mathbb{T}_{M}^{\mathfrak{C},\mathfrak{C}^{\ast}}$ from \cite{DW89}, see Definition \ref{tutteCDdef}.

\begin{theorem}\label{big}
 Let $M$ be a matroid and let $\mathbb{H}$ be a hyperfield with an involution $\tau$.
 Any pair $(\gamma,\delta)$ of $\mathbb{H}\text{-circuit}$ and 
 $\mathbb{H}\text{-cocircuit}$ signatures of $M$  that are strong orthogonal with respect to $\tau$ defines a homomorphism 
 $\Phi:\mathbb{T}_{M}^{\mathfrak{C},\mathfrak{C}^{\ast}}\longrightarrow\mathbb{H}^{\ast}$ that satisfies:
 \begin{enumerate}[label=($\mathcal{S}$)]
  \item\label{bigS} For any circuit $C\in\mathfrak{C}$ and any cocircuit $D\in\mathfrak{C}^{\ast}$ with intersection given by  
        $C\cap D=\{x_{0},x_{1},\ldots,x_{k}\}$, $k\geq2$, and any circuit $C_{j}\in\mathfrak{C}$ such that 
        $C_{j}\cap D=\{x_{0},x_{j}\}$, $1\leq j \leq k$, 
        \begin{equation*}
         1\in\mathop{\mathlarger{\mathlarger{\mathlarger{\boxplus}}}}_{j=1}^{k}
             \Phi\left(\frac{C(x_{j})C_{j}(x_{0})}{C(x_{0})C_{j}(x_{j})}\right)
        \end{equation*}
 \end{enumerate}
 by setting 
 \begin{enumerate}[label=(DH\arabic{*})]
  \item \label{DH2} $\Phi(C(x))=\gamma C(x)$ for any circuit $C\in\mathfrak{C}$ and any $x\in C$;
  \item \label{DH3} $\Phi(D(y))=\tau(\delta D(y))$ for any cocircuit $D\in\mathfrak{C}^{\ast}$ and any $y\in D$.
 \end{enumerate}
 
 Conversely, any homomorphism $\Phi:\mathbb{T}_{M}^{\mathfrak{C},\mathfrak{C}^{\ast}}\longrightarrow\mathbb{H}^{\ast}$ 
 that satisfies \ref{bigS} defines a pair  $(\gamma,\delta)$ of 
 $\mathbb{H}\text{-circuit}$ and 
 $\mathbb{H}\text{-cocircuit}$ signatures of $M$ that are strong orthogonal with respect to $\tau$ by setting
  \begin{enumerate}[label=(DS\arabic{*})]
  \item \label{DS1} $\gamma C(x)=\Phi(C(x))$ for any circuit $C\in\mathfrak{C}$ and any $x\in C$;
  \item \label{DS2} $\delta D(y)=\tau(\Phi(D(y)))$ for any cocircuit $D\in\mathfrak{C}^{\ast}$ and any $y\in D$.
 \end{enumerate}
 
 Two such homomorphisms define the same strong $\mathbb{H}\text{-matroid}$ with underlying matroid $M$ if and only if
 they coincide on the Tutte group $\mathbb{T}_{M}$, while they define the same strong rescaling  
 class of $M$ if and only if they coincide on the inner Tutte group $\mathbb{T}_{M}^{(0)}$.
 
 In particular, there exist one-to-one correspondences between the following sets:
 \begin{enumerate}[label=(GR\arabic{*})]
  \item \label{GR1} $\mathcal{N}_{\mathbb{H}}^{\tau,s}(M)$ and the homomorphisms from 
                    $\mathbb{T}_{M}^{\mathfrak{C},\mathfrak{C}^{\ast}}$ to $\mathbb{H}^{\ast}$ that satisfy \ref{bigS};
  \item \label{GR2} $\mathcal{M}_{\mathbb{H}}^{s}(M)$ and the homomorphisms from 
                    $\mathbb{T}_{M}$ to $\mathbb{H}^{\ast}$ that satisfy \ref{bigS};
  \item \label{GR3} $\mathcal{R}_{\mathbb{H}}^{s}(M)$ and the homomorphisms from 
                    $\mathbb{T}_{M}^{(0)}$ to $\mathbb{H}^{\ast}$ that satisfy \ref{bigS};
 \end{enumerate}

 The same result holds in the weak case if we replace \ref{bigS} by
  \begin{enumerate}[label=($\mathcal{W}$)]
  \item{\label{bigW}} Same statement as \ref{bigS} but with the condition $k\leq 2$.
  
 \end{enumerate}
\end{theorem}

\begin{remark}[On $\epsilon_M$]\label{rem_eps} Any function $\Phi$ defined on the set $\mathbb T_M^{{\mathfrak{C},\mathfrak{C}^{\ast}}}$ must satisfy 
\begin{equation}\label{eq_eps}
\Phi(\epsilon_M)=
\Phi\left(
\frac{C(x)D(x)}{C(y)D(y)}
\right)
\end{equation}
for every circuit $C$ and every cocircuit $D$ of $M$ with $C\cap D=\{x,y\}$, since the arguments on both sides are the same element in $\mathbb T_M^{{\mathfrak{C},\mathfrak{C}^{\ast}}}$ (by Definition \ref{tutteCDdef}).

In particular, any function $\Phi$ defined by \ref{DH2} and \ref{DH3} with respect to a dual pair $(\gamma,\delta)$ of (strong or weak) $\mathbb H^*$-signatures of the circuits and cocircuits of the matroid $M$ must satify
$\Phi(\epsilon_M)=-1$.
\end{remark}

\begin{remark}[The topological case]
If $\mathbb H$ is endowed with a topology, 
 the space of homomorphism $\mathbb{T}_{M}^{{\mathfrak{C},\mathfrak{C}^{\ast}}}$ to $\mathbb{H}^{\ast}$  can be naturally topologized as a subspace of $ Y := (\mathbb{H}^{\ast})^{\sum_{\mathfrak C}\vert C \vert}\times (\mathbb{H}^{\ast})^{\sum_{\mathfrak C^*}\vert D \vert}$ (see Definition \ref{tutteCDdef}). Comparing Definition \ref{def25}, we see that the one-to one correspondence in (GR1) is a restriction of the map $Y\to Y$ given by a cartesian product of the identity on the first factor and the (componentwise) involution on the second factor, hence it is a homeomorphism. The other two correspondences are homeomorphisms because they are obtained by passing to the appropriate quotient (in the case of (GR2)) or subspace (for (GR3)).
\end{remark}

\begin{proof}[Proof of Theorem \ref{big}]
 Let us prove the strong case. The weak case is proved by the same arguments, replacing \ref{bigS} by \ref{bigW}.
 Let $(\gamma,\delta)$ be a pair of 
 $\mathbb{H}\text{-circuit}$ and $\mathbb{H}\text{-cocircuit}$ signatures of $M$ 
 that are strong orthogonal with respect to $\tau$
 and let $\Phi$ be the map defined by \ref{DH2} and \ref{DH3}. 
 
 \begin{claim} $\Phi$ is a group homomorphism from $\mathbb{T}_{M}^{\mathfrak{C},\mathfrak{C}^{\ast}}$ to $\mathbb{H}^{\ast}$.
 \end{claim}
\begin{proof}[Proof of Claim 1] From Remark \ref{rem_eps} we already know that this map sends $\epsilon_{M}$ to $-1$. Hence,  we only need to check that
 \begin{enumerate}[label=(VH\arabic{*})]
  \item \label{VH1} $\Phi(\epsilon_{M}^{2})=\Phi(\epsilon_{M})\odot\Phi(\epsilon_{M})$;
  \item \label{VH2} $\Phi(C(x))\odot\Phi(D(x))=\Phi(\epsilon_{M})\odot\Phi(C(y))\odot\Phi(D(y))$ 
                    for any circuit $C\in\mathfrak{C}$ 
                    and any cocircuit $D\in\mathfrak{C}^{\ast}$ with $C\cap D=\{x,y\}$.
 \end{enumerate}
 Item \ref{VH1} follows immediately from the fact that $\Phi(\epsilon_M)=-1$ (see Remark \ref{rem_eps}). On the other hand, \ref{VH2} follows from the strong orthogonality 
 with respect to $\tau$ between $\gamma$ and $\delta$. 
 To be more precise, let $C\in\mathfrak{C}$ and $D\in\mathfrak{C}^{\ast}$ be a circuit and a 
 cocircuit of $M$ with $C\cap D=\{x,y\}$. Orthogonality implies
 \begin{equation*}
  0\in\gamma C(x)\odot\tau(\delta D(x))\boxplus\gamma C(y)\odot\tau(\delta D(y))
 \end{equation*}
and in turn
 \begin{equation}\label{twoorthogonal}
  \gamma C(x)\odot\tau(\delta D(x))=-\gamma C(y)\odot\tau(\delta D(y))
 \end{equation}
 that is equivalent, by Remark \ref{rem_eps}, \ref{DH2} and \ref{DH3}, to the equality in \ref{VH2}.
\end{proof}
 
 \begin{claim}The homomorphism $\Phi$ fulfills \ref{bigS}.
 \end{claim}
 \begin{proof}[Proof of Claim 2]
 Let us consider a circuit $C\in\mathfrak{C}$ and a cocircuit 
 $D\in\mathfrak{C}^{\ast}$ with $C\cap D=\{x_{0},x_{1},\ldots,x_{k}\}$, $k\geq 1$. 
 Multiplying both sides of the orthogonality condition \eqref{orco} 
  by $(\gamma C(x_{0})\odot\tau(\delta D(x_{0})))^{-1}$, using distributivity and applying \ref{H3} we can rewrite it as
 \begin{equation*}
  -1\in\mathop{\mathlarger{\mathlarger{\mathlarger{\boxplus}}}}_{j=1}^{k}
        \frac{\gamma C(x_{j})\odot\tau(\delta D(x_{j}))}{\gamma C(x_{0})\odot\tau(\delta D(x_{0}))}.
 \end{equation*}
 Now choose, for each $1\leq j\leq k$, a circuit $C_j$ with $C_{j}\cap D=\{x_{0},x_{j}\}$. Using \eqref{twoorthogonal}, multiplying both sides by $-1$ and using distributivity we obtain
 \begin{equation*}
  1\in\mathop{\mathlarger{\mathlarger{\mathlarger{\boxplus}}}}_{j=1}^{k}
      \frac{\gamma C(x_{j})\odot\gamma C_{j}(x_{0})}{\gamma C(x_{0})\odot\gamma C_{j}(x_{j})}.
 \end{equation*}
 With \ref{DH2} this is the same as
 \begin{equation*}
  1\in\mathop{\mathlarger{\mathlarger{\mathlarger{\boxplus}}}}_{j=1}^{k}
      \frac{\Phi(C(x_{j}))\odot\Phi(C_{j}(x_{0}))}{\Phi(C(x_{0}))\odot\Phi(C_{j}(x_{j}))}
 \end{equation*}
 which, since $\Phi$ is a group homomorphism, is equivalent to \ref{bigS}.
 \end{proof}
 
 This concludes the proof of the first part of the theorem, asserting existence and uniqueness of $\Phi$. For the second part,
 let $\Phi$ be a homomorphism from $\mathbb{T}_{M}^{\mathfrak{C},\mathfrak{C}^{\ast}}$ to $\mathbb{H}^{\ast}$ that 
 maps $\epsilon_{M}$ to $-1$ and that fulfills \ref{bigS} 
 and let us consider the pair $(\gamma,\delta)$ of $\mathbb{H}\text{-circuit}$ and 
 $\mathbb{H}\text{-cocircuit}$ signatures of $M$ defined by \ref{DS1} and \ref{DS2}. 
 \begin{claim} The signatures $\gamma$ and $\delta$ are 
 strong orthogonal with respect to $\tau$.
 \end{claim}
 \begin{proof}[Proof of Claim 3] 
Let $C\in\mathfrak{C}$ be a circuit and $D\in\mathfrak{C}^{\ast}$ be a cocircuit. We check condition \eqref{Orientation}. 
 If $C\cap D=\emptyset$ there is nothing to say.
  If $C\cap D=\{u,v\}$, \eqref{Orientation} follows from \ref{VH2} applied to $C$ and $D$, 
        together with \ref{DS1}, \ref{DS2} and the fact that $\Phi$ maps $\epsilon_{M}$ to $-1$.

  Now, if $C\cap D=\{x_{0},x_{1},\ldots,x_{k}\}$, $k>2$, \eqref{Orientation} follows from \ref{bigS}. 
        Indeed, since $\Phi$ is a group homomorphism and using \ref{DS1}, condition \ref{bigS} is equivalent to
        \begin{equation*}
         1\in\mathop{\mathlarger{\mathlarger{\mathlarger{\boxplus}}}}_{j=1}^{k}\frac{\gamma C(x_{j})\odot\gamma C_{j}(x_{0})}
                                                                           {\gamma C(x_{0})\odot\gamma C_{j}(x_{j})}.
        \end{equation*}
        For any choice of circuits $C_1,\ldots,C_k$ with $C_{j}\cap D=\{x_{0},x_{j}\}$,  $1\leq j\leq k$ (see also proof of Claim 2). Since we already know that \eqref{Orientation} holds 
        for a circuit and a cocircuit that intersect in two points, we can substitute $\frac{\tau(\delta D(x_{j}))}{\tau(\delta D(x_{0}))}$ for $\frac{\gamma C_{j}(x_{0})}{\gamma C_{j}(x_{j})}$ throughout and  find 
        \begin{equation*}
         1\in\mathop{\mathlarger{\mathlarger{\mathlarger{\boxplus}}}}_{j=1}^{k}-\frac{\gamma C(x_{j})\odot\tau(\delta D(x_{j}))}
                                                                            {\gamma C(x_{0})\odot\tau(\delta D(x_{0}))}.
        \end{equation*}
        Multiplying both sides by $-\gamma C(x_{0})\odot\tau(\delta D(x_{0}))$ and using distributivity we get 
        \begin{equation*}
         -\gamma C(x_{0})\odot\tau(\delta D(x_{0}))\in\mathop{\mathlarger{\mathlarger{\mathlarger{\boxplus}}}}_{j=1}^{k}
                                                      \gamma C(x_{j})\odot\tau(\delta D(x_{j}))
        \end{equation*}
        which is equivalent to \eqref{Orientation}.
\end{proof}

Claim 3 proves item \ref{GR1} in the theorem's claim. We proceed now with \ref{GR2}, keeping notations as above.

\begin{claim}
Two homomorphisms $\Phi_{1}, \Phi_{2}: \mathbb{T}_{M}^{\mathfrak{C},\mathfrak{C}^{\ast}}\to \mathbb{H}^{\ast}$ satisfying \ref{bigS} define the same rescaling class of $\mathbb{H}\text{-matroids}$ with underlying 
 matroid $M$ if and only if they coincide on the inner Tutte group $\mathbb{T}^{(0)}_{M}$.
 \end{claim}
 
 \begin{proof} Any two homomorphisms that define representatives of the same rescaling class  of $\mathbb H$-matroids must coincide on the inner Tutte group because (e.g., via \ref{DH2}, \ref{DH3}) any two homomorphisms that define $\approx_{\tau}$-equivalent dual pairs must have the same value on the generators of $\mathbb T^{(0)}_M$ given in Remark \ref{easygen} (see Definition \ref{cong}).

For the other implication, let $\Phi_1$, $\Phi_2$ be as in the claim, form their elementwise quotient $\Phi_1/\Phi_2:\mathbb{T}_{M}^{\mathfrak{C},\mathfrak{C}^{\ast}}\to \mathbb H^*$, $x\mapsto \Phi_1(x)/\Phi_2(x)$ and consider the following (horizontal) exact sequence of Abelian groups from \cite[Theorem 1.5]{DW89} (see also Section \ref{section4a} below).

\begin{center}
\begin{tikzcd}
\mathbb T^{(0)}_M \arrow[r,"f"] & 
\mathbb T_M^{\mathfrak C, \mathfrak C^*} \arrow[r,"\Lambda_M"] \arrow[rdd,"\Phi_1 / \Phi_2"'] & 
\operatorname{im} (\Lambda_M) \arrow[r] \arrow[dd,"\alpha",dashrightarrow] & 0 \\
&&& \\
&& \mathbb H^*  &
\end{tikzcd}
\end{center}

Since the $\Phi_i$ coincide on $\mathbb T^{(0)}_M$, we have that $f\circ (\Phi_1/\Phi_2)$ must be constant (hence equal to the identity element). By the universal property of cokernels there exists a homomorphism $\alpha$ as above. 

Now recall (again from Dress and Wenzel's work, restated in Definition \ref{innerdef} below) that $\operatorname{im}(f) \subseteq \mathbb Z^E \times \mathbb Z^{\mathfrak C} \times \mathbb Z^{\mathfrak C^*}$, and in fact the homomorphism $\Lambda_M$ is defined as $\Lambda_M(C(x))=e_x+e_C$ and $\Lambda_M(D(y))=-e_y+e_D$ for all $C\in \mathfrak C$, $x\in C$ and $D\in \mathfrak C^*$, $y\in D$, where $e_x$, $e_C$, $e_D$ are standard basis vectors of $\mathbb Z^E \times \mathbb Z^{\mathfrak C} \times \mathbb Z^{\mathfrak C^*}$.

Therefore, for every $C\in \mathfrak C$ and every $x\in C$,
$$
(\Phi_1/\Phi_2)(C(x))=\alpha\Lambda_M(C(x))=\alpha(e_C)\alpha(e_x)
$$
in the (multiplicative) abelian group $\mathbb H^*$ and hence 
$$
\alpha(e_C)^{-1}\Phi_1(C(x))=\alpha(e_x)\Phi_2(C(x)).
$$
analogously, for every $D\in \mathfrak C^*$ and $y\in C$ we have
$$
\alpha(e_D)^{-1}\Phi_1(D(y))=\tau(\alpha(e_y)^{-1})\Phi_2(D(y)).
$$

In particular, if $\Phi_1$ defines a strong $\mathbb H$-matroid,   the  $\sim_{\tau}$-equivalent $\mathbb H$-matroid determined via Definition \ref{sim} by $b_C:=\alpha(e_C)^{-1}$ and $l_D:=\alpha(e_D)^{-1}$ is $\approx_{\tau}$-equivalent to the strong $\mathbb H$-matroid defined by $\Phi_2$ (Definition \ref{def_ssim} with $h(x):=\alpha(e_x)$). Hence, the $\mathbb H$-matroids defined by $\Phi_1$ and $\Phi_2$ are in the same rescaling class.

 \end{proof}

\begin{claim}
 There is a one-to-one correspondence between the set 
  of homomorphisms from the inner Tutte group $\mathbb T_M^{(0)}$ to $\mathbb H^*$ satisfying \ref{bigS} and rescaling classes of strong $\mathbb H$-matroids with underlying matroid $M$.
\end{claim}

\begin{proof}[Proof of Claim 5] Call $\operatorname{Hom}(\mathbb T_M^{(0)},\mathbb H^*;\mathcal S)$ the set of homomorphisms 
$\mathbb T_M^{(0)}\to \mathbb H^*$ satisfying \ref{bigS}.

\todo{$\Phi_{(\gamma,\delta)}$ etc.?}

It will be enough to check that (i) conditions \ref{DH2}, \ref{DH3} determine a well-defined function $\mathcal R^s_{\mathbb H}(M)\to\operatorname{Hom}(\mathbb T_M^{(0)},\mathbb H^*;\mathcal S)$, and that (ii) conditions \ref{DS1}, \ref{DS2} determine a well-defined function $\operatorname{Hom}(\mathbb T_M^{(0)},\mathbb H^*;\mathcal S)\to \mathcal R^s_{\mathbb H}(M)$. Once this is established, these are clearly mutually inverse functions and exhibit the required bijection.

From Claim 4 we know that any two representatives of the same rescaling class define the same group homomorphism, hence (i) holds. 
Moreover, given any homomorphism $\mathbb T^{(0)}_M\to \mathbb H^*$ satisfying \ref{bigS}, choose an extension of it to the extended Tutte group $\mathbb T_M^{\mathfrak C, \mathfrak C^*}$ (one always exists, see \eqref{ZFree}). Since condition \ref{bigS} only refers to generators of the inner Tutte group, the chosen extension also satisfies \ref{bigS}. Any two such extensions coincide on the inner Tutte group and hence, as in Claim 4, determine the same rescaling class. This proves (ii).
\end{proof}

Claim 5 proves item \ref{GR3} in the theorem's statement. For item \ref{GR2} it is enough to repeat the arguments in Claim 4 and Claim 5 with the Tutte group $\mathbb T_M$ instead of the inner Tutte group $\mathbb T_M^{(0)}$.

\end{proof}

\subsection{Equivalent formulations}\label{section5b}

The result proved in the previous section enables us to present several descriptions of the space of rescaling classes of a matroid. We state this in the following theorem for weak matroids over hyperfields. An analogous theorem for the strong case can be stated (and proved) following the same template, but we do not do this here because for our later application we do not need it.
\begin{theorem}\label{onetoone}
 Given a matroid $M$ 
 and a hyperfield $\mathbb{H}$, there exist one-to-one 
 correspondences 
 \begin{equation}\label{spossato}
  \mathcal{R}_{\mathbb{H}}^{w}(M)
  \stackrel{}{\longleftrightarrow}
  \mathcal{H}_{\mathbb{H}}^{w}(M)
  \longleftrightarrow
  \mathcal{G}_{\mathbb{H}}^{w}(M)
  \longleftrightarrow
  \mathcal{G}_{\mathbb{H}}^{R,w}(M)
 \end{equation}
 where $\mathcal{R}_{\mathbb{H}}^{w}(M)$ is the set of weak $\mathbb{H}\text{-rescaoing}$ classes of $M$ and
 $\mathcal{H}_{\mathbb{H}}^{w}(M)$, $\mathcal{G}_{\mathbb{H}}^{w}(M)$ and $\mathcal{G}_{\mathbb{H}}^{R,w}(M)$ are 
 the sets defined in the list (i)-(iii) below.

 If $\mathbb H$ is a topological hyperfield, these functions are homeomorphisms with respect to the natural topologies.
 \end{theorem}

 \begin{itemize}
  \item[(i)] $\mathcal{H}_{\mathbb{H}}^{w}(M)$ is the set of group homomorphisms 
  $\Phi:\mathbb{T}_{M}^{(0)}\longrightarrow\mathbb{H}^{\ast}$ satisfying:
  \begin{enumerate}[label=(\alph{*})]
   \item \label{a} $\epsilon_{M}\mapsto-1$;
   \item \label{b} For $x_{1},x_{2},x_{3},x_{4}\in L\subseteq E,$ $\operatorname{dim}(L)=1$ 
                   and pairwise distinct circuits $C_{i}\subseteq L\setminus\{x_{i}\}$ (such $C_i$ are unique, see \cite[Lemma 2.7.1]{Fourn}) 
                   \begin{equation*}
                    1\in\Phi\left(\frac{C_{1}(x_{3})C_{2}(x_{4})}{C_{1}(x_{4})C_{2}(x_{3})}\right)
                    \mathop{\mathlarger{\mathlarger{\mathlarger{\boxplus}}}}
                    \Phi\left(\frac{C_{4}(x_{3})C_{2}(x_{1})}{C_{4}(x_{1})C_{2}(x_{3})}\right).
                   \end{equation*}
  \end{enumerate}
  
  \todo{Why the weird indexing $i_1$ etc.?}
  \item[(ii)] $\mathcal{G}_{\mathbb{H}}^{w}(M)$ is the set of $\mathbb{H}^{\ast}\text{-valued}$ functions 
        $\psi_{(C_{i_{1}}C_{i_{2}}|C_{i_{3}}C_{i_{4}})}$ 
        (called $\mathbb{H}^{\ast}$ \emph{cross-ratios}) defined for 
        $C_{i_{1}}$, $C_{i_{2}}$, $C_{i_{3}}$, $C_{i_{4}}\in\mathfrak{C}$ with 
        $\operatorname{dim}(C_{i_{1}}\cup C_{i_{2}}\cup C_{i_{3}}\cup C_{i_{4}})=1$ 
        and $\{C_{i_{1}},C_{i_{2}}\}\cap\{C_{i_{3}},C_{i_{4}}\}=\emptyset$,
        satisfying 
        \begin{equation}\label{CO6}
         1\in\psi_{(C_{i_{1}}C_{i_{2}}|C_{i_{3}}C_{i_{4}})}\boxplus\psi_{(C_{i_{4}}C_{i_{2}}|C_{i_{3}}C_{i_{1}})}
        \end{equation}
        for all pairwise distinct circuits $C_{i_{1}},\ldots,C_{i_{4}}$, as well as
        identities \ref{R3}, \ref{R4}, \ref{R5}, \ref{R6} and \ref{R7} from Definition~\ref{innerTuttegroupTM2}, see Appendix \ref{App:TG}.
        (Here we interpret 
        these identities as equations among $\mathbb{H}^{\ast}\text{-valued}$ functions, under the assumption that 
        the value of $\xi_{M}$ is $-1$.)
        
        With the setup of Appendix \ref{App:TG} one sees that imposing \ref{R3}, \ref{R4}, \ref{R5}, \ref{R6} and \ref{R7} is equivalent to saying that $\psi$ defines a group homomorphism from $\mathbb T^{(2)}$ to the multiplicative group $\mathbb H^*$.

 \item[(iii)] $\mathcal{G}_{\mathbb{H}}^{R,w}(M)$ is the set of $\mathbb{H}^{\ast}\text{-valued}$ functions 
       $\phi_{(C_{j_{1}}C_{j_{2}}|C_{j_{3}}C_{j_{4}})}$ 
       (called \emph{reduced} $\mathbb{H}^{\ast}$-\emph{cross-ratios}) 
       defined for 
       $C_{j_{1}}$, $C_{j_{2}}$, $C_{j_{3}}$, $C_{j_{4}}\in\mathfrak{C}$  with the 
       properties that
       $\operatorname{dim}(C_{j_{1}}\cup C_{j_{2}}\cup C_{j_{3}}\cup C_{j_{4}})=1,$ $j_{1}<j_{2},$ $j_{3}<j_{4}$,  
       $j_{1}<j_{3}$ and satisfying 
       \begin{equation}\label{CR4}
        1\in\phi_{(C_{j_{1}}C_{j_{2}}|C_{j_{3}}C_{j_{4}})}\boxplus
                               \phi_{(C_{j_{1}}C_{j_{3}}|C_{j_{2}}C_{j_{4}})}
       \end{equation}
       for any family of circuits $C_{j_{1}}$, $C_{j_{2}}$, $C_{j_{3}}$, $C_{j_{4}}\in\mathfrak{C}$ 
       such that $j_{1}<j_{2}<j_{3}<j_{4}$ and 
       $\operatorname{dim}(C_{j_{1}}\cup C_{j_{2}}\cup C_{j_{3}}\cup C_{j_{4}})=1$, as well as identities 
       \ref{S3}, \ref{S4} and \ref{S5} from Definition~\ref{innerTuttegroupTMJ0}, see Appendix \ref{App:TG}. 
       (Here we interpret 
        these identities as equations among $\mathbb{H}^{\ast}\text{-valued}$ functions, under the assumption that 
        the symbol $\eta_{M,<_{J}}$ has value $-1$.)

        Again, with the setup of Appendix \ref{App:TG} imposing \ref{S3}, \ref{S4} and \ref{S5} is equivalent to requiring that $\phi$ defines a group homomorphism from $\mathcal T_{M,<J}^{(0)}$ to the multiplicative group $\mathbb H^*$.       
\end{itemize}

\begin{definition}\label{allnumbers}
 Given a matroid $M$ we define the following integer numbers:
 \begin{itemize}
  \item $k_{M}$: the number of reduced cross-ratios (i.e., the number of generators of $\mathcal{T}_{M,<_{J}}^{(0)}$ that are of the form 
                 \ref{Q2});
  \item $n_{M}$: the number of relations \eqref{CR4} (same as the number of relations \ref{S3}).
 \end{itemize}
\end{definition}

\begin{remark}
 \label{lem5}
It is easy to see that there are three ways to totally order a $4$-tuple respecting the conditions of \ref{S3}. Hence, $k_M=3n_M$.

In fact, the set of reduced cross-ratios can be partitioned into triples  so that each relation \eqref{CR4} (resp.\ \ref{S3}) involves cross-ratios from one and only one triple.
\end{remark}

\begin{remark}[Minors of Fano-type]
A glance at Section \ref{App:TG} shows that, if a matroid $M$ has the Fano matroid or its dual as a minor, in its inner Tutte group we have $\epsilon_M=1$ (resp.\ $\sigma_M=1$, $\xi_M=1$). With Remark \ref{rem_eps} we see that then the space of (rescaling classes of) $\mathbb H$-matroids over $M$ is nonempty only for hyperfields $\mathbb H$ satisfying $1=-1$. Compare this with Example \ref{ex:fano}.
\end{remark}

\begin{proof}[Proof of Theorem \ref{onetoone}]
The bijection between $\mathcal{R}_{\mathbb{H}}^{w}(M)$ and $\mathcal{H}_{\mathbb{H}}^{w}(M)$ is in fact an identity (hence a homeomorphism in the topological case). This follows from statement \ref{GR3} because item (b) in the definition of $\mathcal{H}_{\mathbb{H}}^{w}(M)$ imposes the same condition as \ref{bigS}, since for every circuit-cocircuit pair with three-element intersection as in \ref{bigS} there is a corresponding configuration of four circuits and four elements as in (b) such that the values of $\Phi$ on the corresponding ratios are identical -- and vice-versa.

The other two bijections follow from the equivalence of the presentations $\mathbb T^{(2)}_M$, $\mathbb T^{(1)}_M$, $\mathcal{T}_{M,<_{J}}^{(0)}$ of the inner Tutte group (see \cite[Theorem 3, Theorem 4]{GRS95}, Theorem \ref{isoinnerTuttegroups} and
 Lemma \ref{TLL}). Such equivalences provide isomorphisms 
 \begin{equation}\label{OOCG}
  \hom\left(\mathbb{T}_{M}^{(1)},\mathbb{H}^{\ast}\right)\leftrightarrow
  \hom\left(\mathbb{T}_{M}^{(2)},\mathbb{H}^{\ast}\right)\leftrightarrow
  \hom\left(\mathcal{T}_{M,<_{J}}^{(0)},\mathbb{H}^{\ast}\right)
 \end{equation}
 that induce, by restriction, the desired bijections.

These sets of group homomorphisms are naturally topologized as spaces of homomorphisms of topological groups, by regarding the Tutte groups as discrete groups.  Now the bijections in Equation \eqref{OOCG} are induced from isomorphisms of (discrete) topological groups, hence they are homeomorphisms of topological spaces. Finally, the maps in Equation \eqref{spossato} are bijective restrictions of homeomorphisms, hence homeomorphisms themselves.
\end{proof}

\subsection{Embeddings of spaces of rescaling classes}
\label{section5c}

\begin{theorem}\label{rescalingspaceembeddings}
 Given a matroid $M$ and a sub-hyperfield $\mathbb{H}_{1}$ of the hyperfield $\mathbb{H}_{2}$, there exist injections
 \begin{equation*}
 i:\mathcal{M}_{\mathbb{H}_{1}}^{w}(M)\hookrightarrow\mathcal{M}_{\mathbb{H}_{2}}^{w}(M),\quad\quad
  j:\mathcal{R}_{\mathbb{H}_{1}}^{w}(M)\hookrightarrow\mathcal{R}_{\mathbb{H}_{2}}^{w}(M)
 \end{equation*}
that are continuous. The same holds in the strong case.
\end{theorem}

\begin{proof}
Consider the inclusion map of hyperfields representing the sub-hyperfield relationship. It induces an injective group homomorphism $\iota: \mathbb H_1^* \hookrightarrow \mathbb H_2^*$ between the corresponding multiplicative groups. Left-cancellativity of $\iota$ as a monomorphism proves that the induced maps 
$$
\operatorname{hom}(\mathbb T_M,\mathbb H_1^*)\hookrightarrow 
\operatorname{hom}(\mathbb T_M,\mathbb H_2^*),\quad\quad
\operatorname{hom}(\mathbb T^{(1)}_M,\mathbb H_1^*)\hookrightarrow \operatorname{hom}(\mathbb T^{(1)}_M,\mathbb H_2^*)
$$
are injections, and thus so are the functions $i$, $j$ induced by restriction to subsets (compare Theorem \ref{big} and the proof of Theorem \ref{onetoone}). Continuity of $\iota$ implies the continuity of $i$ and $j$ since the corresponding spaces are topologized as subspaces of spaces of homomorphisms of topological groups.
\end{proof}

\subsection{An explicit description of ${\mathcal{G}}_{\mathbb{H}}^{R,w}(M)$ and some examples}

We now provide an explicit description of the set ${\mathcal{G}}_{\mathbb{H}}^{R,w}(M)$ as solution of systems of equations.

\begin{definition}\label{structureset}
  Given a hyperfield $\mathbb{H}$ we set 
  \begin{equation*}
   X_{\mathbb{H}}:=\left\lbrace(u,v)\in\mathbb{H}^{\ast}\times\mathbb{H}^{\ast}\mid1\in u\boxplus v\right\rbrace.
  \end{equation*}
\end{definition}

Recall from Definition \ref{allnumbers} the numbers $k_{M}$ and $n_{M}$. In particular, $n_M$ is the number of relations of type \eqref{CR4} and $k_M$ is the number of reduced $\mathbb H^*$ cross-ratios. 

\begin{definition}\label{def:Bone}
Let  $\mathcal{B}\subseteq(\mathbb{H}^{\ast})^{k_{M}}$ be the space of solutions of \ref{S3}, \ref{S4} and \ref{S5}.
\end{definition}

\todo{I have added proofs of the following two things (So we do not refer to D-Saini)}
\begin{proposition}
\label{geo1}
Let $M$ be a matroid without minors of Fano or dual-Fano type. 
 Up to permuting variables, the set $\mathcal G_{\mathbb{H}}^{R,w}(M)$ equals
        \begin{equation*}
           \mathcal{B}\cap\prod_{j=1}^{n_{M}}\left(X_{\mathbb{H}}\times\mathbb{H}^{\ast}\right)
           \subseteq(\mathbb{H}^{\ast})^{3n_{M}}.
        \end{equation*}
\end{proposition}

\begin{proof} Recall Remark \ref{lem5}, fix an enumeration of the set of relations by indices $j=1,\ldots n_M$ and assign to each $j$ a triple $\alpha_j,\beta_j,\gamma_j$ of reduced cross-ratios, where $\alpha_j$ and $\beta_j$ are those cross-ratios appearing nontrivially in the $j$-th relation \eqref{CR4}. In particular, $k_M=3n_M$ and 
we can enumerate the set of reduced cross-ratios as $\alpha_1,\beta_1,\gamma_1,\alpha_2,\beta_2,\gamma_2,\alpha_3,\ldots$ etc.

If we now understand cartesian products to be taken with respect to the chosen enumerations, we see that

$$
\prod_{j=1}^{n_M}(X_{\mathbb H}\times \mathbb H^*) \subseteq (\mathbb H^*)^{3n_M}
$$
is exactly the subset of all values of the reduced cross-ratios that satisfy all relations \eqref{CR4}. The claim follows.

\end{proof}

\begin{corollary}
\label{geo2}
 Let $F:(\mathbb{H}^{\ast})^{3n_{M}}\longrightarrow(\mathbb{H}^{\ast})^{2n_{M}}$ be the projection map given by 
 $(\alpha_{j},\beta_{j},\gamma_{j})_{1\leq j\leq n_{M}}\mapsto(\alpha_{j},\beta_{j})_{1\leq j \leq n_{M}}$, where we label coordinates as in the proof of Proposition \ref{geo1}. Then the restriction
 \begin{equation*}
  \left.F\right|_{\mathcal{B}\cap\prod_{j=1}^{n_{M}}\left(X_{\mathbb{H}}\times\mathbb{H}^{\ast}\right)}:
                  \mathcal{B}\cap\prod_{j=1}^{n_{M}}\left(X_{\mathbb{H}}\times\mathbb{H}^{\ast}\right)\longrightarrow
                 F(\mathcal{B})\cap\prod_{j=1}^{n_{M}}X_{\mathbb{H}}. 
 \end{equation*}
 is a bijection. If $\mathbb H$ is a topological hyperfield, this map is a homeomorphism. 
\end{corollary}

We postpone the proof of this corollary after some examples and applications.

\begin{example}\label{tropicalexample}
 As previously seen the tropical hyperfield $\mathbb{T}_{+}$ is WAM. Hence, in view of Theorem \ref{main} 
 and Corollary \ref{geo2}, given a matroid $M$ without minors of Fano or dual-Fano type, it is possible to describe the space 
 $\mathcal{R}_{\mathbb{T}_{+}}^{w}(M)$ of rescaling classes of weak $\mathbb{T}_{+}\text{-matroids}$ with underlying matroid 
 $M$ as intersection $\mathcal{L}\cap\prod_{j}^{n_{M}}X_{\mathbb{T}_{+}}$, 
 where $n_{M}$ is the number introduced in Definition \ref{allnumbers},  
 $X_{\mathbb{T}_{+}}$ is the tropical line 
 \begin{equation*}
  X_{\mathbb{T}_{+}}=\left\lbrace
                      (u,v)\in\mathbb{R}^{2}
                      \left|
                       \begin{array}{lll}
                        0=\max\{u,v\}                                    & \text{if} & u\neq v \\
                        0\in\{c\in\mathbb{R}\cup\{-\infty\}\mid c\leq u\} & \text{if} & u=v     \\
                       \end{array}
                      \right.
                     \right\rbrace
 \end{equation*}
 and $\mathcal{L}$ is a linear subspace of $\mathbb{R}^{2n_{M}}$. 
\end{example}

As a consequence of Theorem \ref{onetoone} and Corollary \ref{geo2} we obtain an explicitly computable upper 
bound to the number of rescaling classes of weak $\mathbb{H}\text{-matroids}$ with underlying matroid $M$, 
when $\mathbb{H}$ is a finite WAM hyperfield. 

\begin{proposition}
 \label{finiteH}
 For a matroid $M$ without minors of Fano or dual-Fano type and a finite WAM hyperfield $\mathbb{H}$,
 let $n_{M}$ be the number introduced in Definition \ref{allnumbers} and let $X_{\mathbb{H}}$ be the set 
 introduced in Definition \ref{structureset}. Thus,
 \begin{equation*}
  \left|\mathcal{R}_{\mathbb{H}}^{w}(M)\right|\leq\left|X_{\mathbb{H}}\right|^{n_{M}}.
 \end{equation*}
\end{proposition}

\begin{proof}[Proof of Proposition \ref{finiteH}]
 The assumption that $\mathbb{H}$ is finite ensures us that $\mathcal{R}_{\mathbb{H}}(M)$ and $X_{\mathbb{H}}$ 
 are both finite. Since $\mathbb{H}$ is WAM, from Corollary \ref{geo2} we get
 \begin{equation*}
  \left|\mathcal{R}_{\mathbb{H}}^{w}(M)\right|=\left|F(\mathcal{B})\cap\prod_{j=1}^{n_{M}}X_{\mathbb{H}}\right|\leq 
  \left|\prod_{j=1}^{n_{M}}X_{\mathbb{H}}\right|=\left|X_{\mathbb{H}}\right|^{n_{M}}.
 \end{equation*}
\end{proof}

\begin{example}\label{signboundexample}
\todo{This example changed (E)}
 For the sign hyperfield $\mathbb{S}$ the set $X_{\mathbb{S}}$ consists of three elements (see \cite[Example 1]{GRS95}), and we know that every weak $\mathbb S$-matroid is a strong $\mathbb S$-matroid. 
 Thus, in the context of oriented matroids, Proposition \ref{finiteH} gives an upper bound for the number of 
 reorientation classes of oriented matroids with prescribed underlying matroid:
 $$
 \left|\mathcal{R}_{\mathbb{S}}(M)\right|\leq3^{n_{M}}.
 $$
\end{example}

\begin{proof}[Proof of Corollary \ref{geo2}]
  Let $\widetilde{\mathcal{B}}\subseteq(\mathbb{H}^{\ast})^{3n_{M}}$ be the space of solutions of 
  \ref{S3}.
  $\widetilde{\mathcal{B}}$ is non-empty, since $(1,1,-1,\ldots,1,1,-1)\in\widetilde{\mathcal{B}}$.
  The restriction map 
  \begin{equation*}
   \left.F\right|_{\widetilde{\mathcal{B}}}:\widetilde{\mathcal{B}}\longrightarrow(\mathbb{H}^{\ast})^{2n_{M}}
  \end{equation*}
  is a bijection with inverse
  $G:(\mathbb{H}^{\ast})^{2n_{M}}\longrightarrow\widetilde{\mathcal{B}}$ defined by
  \begin{equation*}
   (u_{j},v_{j})_{1\leq j \leq n_{M}}\mapsto(u_{j},v_{j},-u_{j}^{-1}\odot v_{j})_{1\leq j\leq n_{M}}.
  \end{equation*}
  
  Notice that in the topological case $G$ is continuous as a composition of continuous functions, and $\left.F\right|_{\widetilde{\mathcal{B}}}$ is continuous as restriction of a continuous function. Hence $\left.F\right|_{\widetilde{\mathcal{B}}}$ is a homeomorphism. 

  Since the map 
  \begin{equation*}
   \left.F\right|_{\mathcal{B}\cap\prod_{j=1}^{n_{M}}\left(X_{\mathbb{H}}\times\mathbb{H}^{\ast}\right)}:
                   \mathcal{B}\cap\prod_{j=1}^{n_{M}}\left(X_{\mathbb{H}}\times\mathbb{H}^{\ast}\right)\longrightarrow
                  F\left(\mathcal{B}\cap\prod_{j=1}^{n_{M}}\left(X_{\mathbb{H}}\times\mathbb{H}^{\ast}\right)\right). 
  \end{equation*}
  is a restriction of $\left.F\right|_{\widetilde{\mathcal{B}}}$, it is also bijective (resp.\ a homeomorphism)
  
  To conclude our proof it suffices to prove that
  \begin{equation}\label{DoubleInclusiongeo2}
    F(\mathcal{B})\cap\prod_{j=1}^{n_{M}}X_{\mathbb{H}}= F\left(\mathcal{B}\cap\prod_{j=1}^{n_{M}}\left(X_{\mathbb{H}}
                                            \times\mathbb{H}^{\ast}\right)\right).
  \end{equation} 
  This is a straightforward check, that we nevertheless include for completeness.
 	For the left-to-right inclusion let $P\in F(\mathcal{B})\cap\prod_{j=1}^{n_{M}}X$ and write $P$ in the form 
         $P=(u_{j},v_{j})_{1\leq j\leq m}$. From $P\in F(\mathcal{B})$, there exists $Q\in\mathcal{B}$ such that 
         $P=F(Q)$. Since the points of $\mathcal{B}$ satisfy \ref{S3}, we must have 
         $Q=(u_{j},v_{j},-u_{j}^{-1}\odot v_{j})_{1\leq j\leq m}$, and hence
         \begin{equation}\label{WhereQ}
          Q\in\prod_{j=1}^{n_{M}}\left(X_{\mathbb{H}}\times\mathbb{H}^{\ast}\right)
         \end{equation}
         as required.
          
  	For the right-to-left inclusion, let 
         $P\in F\left(\mathcal{B}\cap\prod_{j=1}^{n_{M}}\left(X_{\mathbb{H}}\times\mathbb{H}^{\ast}\right)\right)$.
         Hence, there is $Q\in\mathcal{B}\cap\prod_{j=1}^{n_{M}}\left(X_{\mathbb{H}}\times\mathbb{H}^{\ast}\right)$ 
         such that $P=F(Q)$.
         Clearly $P\in F(\mathcal{B})$. Write $P$ in the form $P=(u_{j},v_{j})_{1\leq j\leq n_{M}}$. Now, the points of $\mathcal{B}$ satisfy \ref{S3}, so 
         from $Q\in\mathcal{B}$ and $P=F(Q)$, by the definition of $F$, we deduce that $Q=(u_{j},v_{j},-u_{j}^{-1}\odot v_{j})_{1\leq j\leq n_{M}}$.
         This, together with $Q\in\prod_{j=1}^{n_{M}}\left(X_{\mathbb{H}}\times\mathbb{H}^{\ast}\right)$
         and the definition of the map $F$, implies
         \begin{equation}\label{WhereP}
          P\in\prod_{j=1}^{n_{M}}X_{\mathbb{H}}
         \end{equation}
         and concludes the proof.
\end{proof}

\section{Some applications}
\label{sec:App}

  \subsection{Non-$\mathbb C$-realizable uniform weak phased matroids}\label{ss:realizable}
  
  We turn to the case of matroids over the phase hyperfield (a.k.a.\ {\em phased matroids}). These are structures that grew out of the attempt to construct an analogue of oriented matroid theory for complex vectorspaces \cite{AD12,BKRG03}. In this context it is natural to ask about representability  over $\mathbb C$ of a given phased matroid -- a line of research that was already tackled in \cite{BKRG03} for rank $2$.

        In this section we use some topological and geometric properties of the space introduced in Corollary \ref{geo2} 
to prove the existence of non-realizable non-chirotopal uniform weak phased matroids. Let us first recall some terminology.

 A weak phased matroid $\Phi \in \mathcal{M}_{\mathbb{P}}^w(M)$ -- say of rank $d$ on the ground set $[m]$ -- is called $\mathbb C$-realizable (here henceforth just ``{\em realizable}'') if there is a complex matrix $A\in M_{d,m}(\mathbb{C})$  such that   $\Phi$ contains a representative given by the function
\begin{equation}\label{DefAsso}
\varphi_{A}:[m]^{d}\longrightarrow(S^{1}\cup\{0\});\quad  (i_{1},\ldots,i_{d})\mapsto\operatorname{ph}(\det(A^{i_{1}},\ldots,A^{i_{d}})),
\end{equation}
where $A^{j}$ denotes the $j\text{-th}$ column of $A$.
This is equivalent to saying that $\Phi$ is in the image of the function $\mathcal{M}_{\mathbb{C}}^w(M)\to \mathcal{M}_{\mathbb{P}}^w(M)$ induced by the hyperfield homomorphism $\operatorname{ph}:\mathbb C \to \mathbb P$.

A natural class of nonrealizable phased matroids is induced by the class of nonrealizable oriented matroids, via the hyperfield embedding $\iota: \mathbb S\hookrightarrow \mathbb P$ (see Section \ref{section5c}). In order to capture this situation, we call $\Phi \in  \mathcal{M}_{\mathbb{P}}^w(M)$ \emph{chirotopal} if $\Phi\in \iota_* (\mathcal{M}_{\mathbb S}(M))$, i.e., if there exists  
 $\chi\in\Phi$ with $\chi ([m]^d)\subseteq \{0,+1,-1\}$. This $\chi$ is then a chirotope in the sense of oriented matroid theory, see \cite{Zie}, hence the terminology.   
  
  \begin{center}
  
\begin{tikzpicture}
\node at (0,0) (P) {$\mathbb P$};
\node at (-2,0) (C) {$\mathbb C$};
\node at (-1,.2)(ph) {{\small $\operatorname{ph}$}};
\node at (2,0) (K) {$\mathbb K$};
\node at (0,1.5) (S) {$\mathbb S$};
\node at (.1,.8) (i) {{\small $\iota$}};
\draw [->] (C) -- (P);
\draw [->] (P) -- (K);
\draw [->] (S) -- (P);
\end{tikzpicture}
$\quad\quad$
\begin{tikzpicture}
\node at (0,0) (P) {$\mathcal R_{\mathbb P}^w(M)$};
\node at (-2.5,0) (C) {$\mathcal R_{\mathbb C}(M)$};
\node at (-1.2,.2)(ph) {{\small $\operatorname{ph}_*$}};
\node at (2,0) (K) {$M$};
\node at (0,1.5) (S) {$\mathcal R_{\mathbb S}(M)$};
\node at (.3,.8) (i) {{\small $\iota_*$}};
\draw [->] (C) -- (P);
\draw [->] (P) -- (K);
\draw [->] (S) -- (P);
\end{tikzpicture}
  
  \end{center}

We then say that a phasing class $P\in\mathcal{R}_{\mathbb{P}}^w(M)$ is \emph{realizable} 
        if there exists a realizable weak phased matroid $\Phi\in\mathcal{M}_{\mathbb{P}}^w(M)$ such that 
        $P=\pi(\Phi),$ where
        $\pi_{\mathbb P}:\mathcal{M}_{\mathbb{P}}^w(M)\longrightarrow\mathcal{R}_{\mathbb{P}}^w(M)$ is the standard quotient map (see Definition \ref{TypeC}). Analogously we define a \emph{chirotopal} phasing class.

\begin{remark}
Notice that, with \cite[Lemma 1.6]{BKRG03}, the realizability of a phasing class $P\in\mathcal{R}_{\mathbb{C}}(M)$ 
is equivalent to requiring realizability of each weak phased matroid $\Phi$  such that $P=\pi(\Phi)$.
\end{remark}

We now state our theorem and outline the main steps of the proof, postponing some technical 
lemmas to Appendix \ref{appendixBBBYYY}.

\begin{theorem}
  \label{SuperMAIN}
 For $m\geq5$ and $2\leq d\leq m-2$ there exists 
 a non-realizable non-chirotopal weak phased matroid with underlying matroid $U_{d}(m)$.
\end{theorem}

\begin{proof}
Let us consider the natural total ordering of the ground set $E=\{1,\ldots,m\}$, and an enumeration 
 $C_{1},C_{2},\ldots$ of the circuits of $U_{d}(m)$ (e.g. according to the lexicographic order).
 
 The {idea} of the proof is to exploit the topology of the subspace of realizable classes 
 in order to prove that it cannot exhaust the space $\mathcal{R}_{\mathbb{P}}^w(U_{d}(m))$. 
 This is the content of Lemma \ref{RealizabilityCriterion}, which reduces our task to finding a 
 certain type of realizable class, and more precisely to 
 finding  
 a matrix 
 $A\in M_{d,m}(\mathbb{C})$ such that:
 \begin{enumerate}
  \item[(A1)] The weak phirotope $\varphi_{A}$ associated to $A$ has underlying matroid $U_{d}(m)$;
  \item[(A2)] For any family $C_{d_{1}}$, $C_{d_{2}}$, $C_{d_{3}}$, $C_{d_{4}}\in\mathfrak{C}$ of 
              circuits of $U_{d}(m)$ with the properties that $d_{1}<d_{2}<d_{3}<d_{4}$ and 
              $\operatorname{dim}(C_{d_{1}}\cup C_{d_{2}}\cup C_{d_{3}}\cup C_{d_{4}})=1$, one has that
              \begin{equation}\label{dadostar}
              \begin{aligned}
               &\left(\frac{\varphi_{A}(C_{d_{1}}(x_{d_{3}}))\varphi_{A}(C_{d_{2}}(x_{d_{4}}))}
                          {\varphi_{A}(C_{d_{1}}(x_{d_{4}}))\varphi_{A}(C_{d_{2}}(x_{d_{3}}))}
                          ,
                     \frac{\varphi_{A}(C_{d_{1}}(x_{d_{2}}))\varphi_{A}(C_{d_{3}}(x_{d_{4}}))}
                          {\varphi_{A}(C_{d_{1}}(x_{d_{4}}))\varphi_{A}(C_{d_{3}}(x_{d_{2}}))}
               \right) \\
               &\text{is not an element of the set}\left\lbrace(1,1),(1,-1),(-1,1)\right\rbrace \\
              \end{aligned}
              \end{equation}
              where $\{x_{d_{j}}\}=(C_{d_{1}}\cup C_{d_{2}}\cup C_{d_{3}}\cup C_{d_{4}})\setminus C_{d_{j}}$. 
              For $i\neq j$, 
              $\varphi_{A}(C_{d_{i}}(x_{d_{j}}))$ is the evaluation of the weak phirotope $\varphi_{A}$ on the ordered 
              $d\text{-uple}$ defined by the identity 
              $C_{d_{i}}(x_{d_{j}})=C_{d_{i}}\setminus\left\lbrace x_{d_{j}}\right\rbrace$.
              Notice that $U_{d}(m)$ is the uniform matroid, so that $C_{d_{i}}(x_{d_{j}})$ is a basis for all $i\neq j$,  therefore  
              $\varphi_{A}(C_{d_{i}}(x_{d_{j}}))\neq0$. Hence, the expressions in \eqref{dadostar} are well defined.
 \end{enumerate}
 This can be done as follows. Let $\mathcal U$ denote the space of all complex $d\times m$ matrices of the form
 \begin{equation*}
   A=\left(
          \begin{array}{cccc|cccc}
             1   & \quad  & \quad  & \quad &    1   & \ast   & \cdots  & \ast  \\
           \quad & \ddots & 0 & \quad & \vdots &   \vdots      & \quad   &  \vdots     \\
           \quad & 0  & \ddots & \quad &    1   & \ast & \cdots  & \ast \\
           \quad & \quad  & \quad  &   1   &    1   &      1        & \cdots  &   1         \\
          \end{array}
    \right)
 \end{equation*}
 and whose associated matroid is $U_d(m)$.
 Set $N=n_{U_{d}(m)}$ 
 and consider the function $F:\mathcal{U}\longrightarrow\mathbb{C}^{2N}$ whose components are given by pairs
 \begin{equation}\label{dadostarskyantos}
 \left(
  \frac{\det A_{(C_{d_{1}}(x_{d_{3}}))}\det A_{(C_{d_{2}}(x_{d_{4}}))}}
       {\det A_{(C_{d_{1}}(x_{d_{4}}))}\det A_{(C_{d_{2}}(x_{d_{3}}))}}
                          ,
  \frac{\det A_{(C_{d_{1}}(x_{d_{2}}))}\det A_{(C_{d_{3}}(x_{d_{4}}))}}
       {\det A_{(C_{d_{1}}(x_{d_{4}}))}\det A_{(C_{d_{3}}(x_{d_{2}}))}}
       \right)
 \end{equation}
 one for each family of circuits as in $(A2)$, were
 $A_{(C_{d_{i}}(x_{d_{j}}))}$ denotes the $d\times d$ submatrix of $A$
with columns indexed by the ordered 
 $d\text{-uple}$ $C_{d_{i}}(x_{d_{j}})$. The same argument as in \eqref{dadostar} above shows that the expressions 
 in \eqref{dadostarskyantos} are well defined.
 By Corollary \ref{cor:nonC} the components of the (holomorphic) map $F$ are all
 non-constant over the non-empty, connected and open set $\mathcal U$ (Lemma \ref{OpenConnectedSide}) and thus, by Lemma \ref{HoloSide},
 there exists $\widetilde{A}\in\mathcal{U}$ such that for each family $C_{d_1}, \ldots, C_{d_4}$ of circuits appearing in (A2)
 \begin{equation*}
  \left(
  \frac{\det \widetilde{A}_{(C_{d_{1}}(x_{d_{3}}))}\det \widetilde{A}_{(C_{d_{2}}(x_{d_{4}}))}}
       {\det \widetilde{A}_{(C_{d_{1}}(x_{d_{4}}))}\det \widetilde{A}_{(C_{d_{2}}(x_{d_{3}}))}}
                          ,
  \frac{\det \widetilde{A}_{(C_{d_{1}}(x_{d_{2}}))}\det \widetilde{A}_{(C_{d_{3}}(x_{d_{4}}))}}
       {\det \widetilde{A}_{(C_{d_{1}}(x_{d_{4}}))}\det \widetilde{A}_{(C_{d_{3}}(x_{d_{2}}))}}
  \right)
  \in
  \left(
  \mathbb{C}\setminus\mathbb{R}
  \right)^{2}.
 \end{equation*} 
 Then, $\varphi_{\widetilde{A}}$ satisfies \eqref{dadostar} by definition, and the underlying matroid of 
 $\varphi_{\widetilde{A}}$ is $U_{d}(m)$ since $\widetilde{A}\in\mathcal{U}$. 
\end{proof}

  \subsection{Uniform central hyperplane arrangements}\label{Section9}
       As a further application of our methods, and as a concrete token of the fact that weak phased matroids encode significant 
topological properties of arrangements, we now exploit some elementary topological properties of the realization
space of the uniform matroid over $\mathbb{C}$ in order to prove that complex central hyperplane arrangements with same
underlying uniform matroid are lattice-isotopic, hence their complement manifolds are diffeomorphic via Randell's isotopy theorem. This improves on \cite[Theorem 2]{Hat75}, 
where the homotopy type of these arrangements is discussed. Our proof is an adaptation of that of
Lemma \ref{OpenConnectedSide}.

First, we provide a quick review of some basic definitions and results about arrangements to set the appropriate 
background and notations. We refer to the book \cite{OT92} for a  general treatment of hyperplane arrangements, 
and to \cite{FR00} for a survey of their homotopy theory.

A finite collection of affine subspaces 
$\mathcal{A}=\{H_{1},\ldots,H_{m}\}$ in 
$\mathbb{C}^{d}$ is an \emph{arrangement}. Its \emph{complement manifold} 
$M(\mathcal{A})$ is the complement of the union of the $H_{i}$ in $\mathbb{C}^{d}$. 
If each $H_{i}$ is a linear subspace the arrangement is called 
\emph{central}.

Given an arrangement $\mathcal{A}=\{H_{1},\ldots, H_{m}\}$ we assign a \emph{rank} 
to every $I\subseteq\left[m\right]$ by setting
\begin{equation*}
 \operatorname{rk}_{\mathcal{A}}(I) = \operatorname{codim}\bigcap_{i\in I}H_{i}
\end{equation*}
(where we define the empty set to have codimension $d+1$).

Two arrangements $\mathcal{A}=\{H_{1},\ldots, H_{m}\}$ and $\mathcal{B}=\{K_{1},\ldots,K_{m}\}$  
have the same \emph{combinatorial type} if the functions $\operatorname{rk}_{\mathcal A}$ and
$\operatorname{rk}_{\mathcal B}$ are equal.

Given an open interval $(a,b)\subseteq\mathbb{R}$, a \emph{smooth one-parameter family} of arrangements 
is a collection $\{\mathcal{A}_{t}\}_{t\in(a,b)}$ of arrangements 
$\mathcal{A}_{t}=\{H_{1}(t),\ldots,H_{m}(t)\}$ in $\mathbb{C}^{d}$ such that there exist smooth functions from 
$(a,b)$ to $\mathbb{C}$ for the coefficients of the defining equations of the subspaces $H_{i}(t)$.
With a slight abuse of notation we write $\mathcal{A}_{t}$ for $\{\mathcal{A}_{t}\}_{t\in(a,b)}$, omitting the interval 
of parameters $(a,b)$.

\begin{definition}
 A smooth one-parameter family of arrangements 
 $\mathcal{A}_{t}$ 
 is an \emph{arrangement isotopy} if for any $t_{1}$ and $t_{2}$ the arrangements 
 $\mathcal{A}_{t_{1}}$ and $\mathcal{A}_{t_{2}}$ have the same combinatorial type. 
 In this case we say that $\mathcal{A}_{t_{1}}$ and $\mathcal{A}_{t_{2}}$ are \emph{isotopic arrangements}.
\end{definition}

The following theorem, sometimes referred to as ``isotopy theorem'', was proved by Richard Randell.

\begin{theorem}[{\cite[p.\ 556]{Ran89}}]\label{thm:RIT}
 If $\mathcal{A}_{t_1}$ and $\mathcal{A}_{t_1}$ are isotopic arrangements, then  
 the complement manifolds $M(\mathcal{A}_{t_{1}})$ and $M(\mathcal{A}_{t_{2}})$ are diffeomorphic.
\end{theorem}

An arrangement of codimension $1$ subspaces is called a \emph{hyperplane} arrangement.
If $\mathcal{A}=\{H_{1},\ldots,H_{m}\}$ in $\mathbb{C}^{d}$ is a central hyperplane arrangement, 
the function $\operatorname{rk}_{\mathcal A}$ is the rank function of a matroid $M_{\mathcal{A}}$ on the ground set $[m]$. 
The \emph{rank} of $\mathcal{A}$ is by definition the rank of $M_{\mathcal{A}}$.
Thus, a smooth one-parameter family 
$\mathcal{A}_{t}$ of central hyperplane arrangements is a 
lattice isotopy if and  only if $M_{\mathcal{A}_{t_{1}}}=M_{\mathcal{A}_{t_{2}}}$ for any $t_{1}$ and $t_{2}$.

\begin{theorem}\label{UHA}
 Let $\mathcal{A}=\{H_{1},\ldots,H_{m}\}$ and $\mathcal{B}=\{K_{1},\ldots,K_{m}\}$ be central hyperplane 
 arrangements in $\mathbb{C}^{d}$ with same underlying uniform matroid. Then, 
 $\mathcal{A}$ and $\mathcal{B}$ are isotopic arrangements.
\end{theorem}

\begin{proof}
Obviously the only case of interest is when $d\geq1$ and $m\geq1$. Moreover, if $1\leq m<d$ we can reduce to the case 
of arrangements $\widetilde{\mathcal{A}}=\{\widetilde{H}_{1},\ldots,\widetilde{H}_{m}\}$ and 
$\widetilde{\mathcal{B}}=\{\widetilde{K}_{1},\ldots,\widetilde{K}_{m}\}$
in $\mathbb{C}^{m}$,  where
$\widetilde{H}_{j}$ is the quotient of $H_{j}$ by $\bigcap_{r=1}^{m}H_{r}$ and similarly 
$\widetilde{K}_{j}$ is the quotient of $K_{j}$ by $\bigcap_{r=1}^{m}K_{r}$.

Thus, it suffices to prove our statement for $1\leq d\leq m$.
 Pick linear forms $\alpha_{i}$ and $\beta_{i}$ such that 
 $H_{i}=\ker\alpha_{i}$ and $K_{i}=\ker\beta_{i}$. Let us denote by $\alpha_{i}^{j}$ and $\beta_{i}^{j}$ the 
 $j\text{-th}$ component of $\alpha_{i}$ and $\beta_{i}$, respectively. Set $A=(\alpha_{i}^{j})^{T}$ and 
 $B=(\beta_{i}^{j})^{T}$. Now, consider $\mathcal S_{d,m}$, the space of all corresponding matrices over $\mathbb C$. The matrices $A$ and $B$ belong to 
 $\mathcal S_{d,m}$. Hence, it is enough to show that 
  $\mathcal{S}_{d,m}$  
 is an open connected subspace of $M_{d,m}(\mathbb{C})$.
 The same arguments in the proof of Lemma \ref{OpenConnectedSide}  (with $\mathcal S_{d,m}$ in place of $\mathcal U$) imply that this is indeed the case.
\end{proof}

Via Theorem \ref{thm:RIT} we immediately obtain the following.

\begin{corollary}\label{cor:UHA}
Let $\mathcal{A}$ and $\mathcal{B}$ be central hyperplane 
 arrangements in $\mathbb{C}^{d}$ with same underlying uniform matroid. 
 Then, the complement manifolds $M(\mathcal{A})$ and $M(\mathcal{B})$ are diffeomorphic. 
\end{corollary}

\appendix

 \section{Computations for Sections 3 and 4}
  \subsection{Computations for the proof of Lemma \ref{keylemma}}\label{appendix1}
Up to th
equivalences  $\sim_{\tau}$ and $\cong_{\tau}$ (see Definition \ref{sim} and Definition \ref{cong}), we can assume that 
\begin{equation*}
 \begin{array}{llll}
  \gamma C_{1}(2)=1;  & \gamma C_{1}(3)=1;  & \gamma C_{1}(4)=1 & \gamma C_{2}(1)=1;   \\
  \gamma C_{3}(1)=-1; & \gamma C_{3}(2)=1;  & \gamma C_{4}(1)=1 & \delta D_{1}(4)=1;   \\
  \delta D_{2}(4)=-1; & \delta D_{3}(4)=-1; & \multicolumn{2}{l}{\delta D_{4}(3)=-1.}  \\
 \end{array}
\end{equation*}

Hence, from the definition of $\mathbb{P}_{\mathbb{H},f,g,\mu}$ we can use \eqref{base} and Proposition \ref{algbas}
to compute the other values of $\gamma$ and $\delta$.

\begin{equation*}
\begin{aligned}
\delta D_{2}(3)&=1\text{ since}\\
    -1&=\mathbb{P}_{\mathbb{H},f,g,\mu}(C_{1},D_{2},3,4)=
                                        \frac{\gamma C_{1}(3)\odot\tau(\delta D_{2}(3))}
                                             {\gamma C_{1}(4)\odot\tau(\delta D_{2}(4))}=
                                        \frac{1\odot\tau(\delta D_{2}(3))}{1\odot-1};\\ 
\end{aligned}
\end{equation*}

\begin{equation*}
 \begin{aligned}
   \delta D_{3}(2)&=1\text{ since}\\
     -1&=\mathbb{P}_{\mathbb{H},f,g,\mu}(C_{1},D_{3},2,4)=
                                        \frac{\gamma C_{1}(2)\odot\tau(\delta D_{3}(2))}
                                             {\gamma C_{1}(4)\odot\tau(\delta D_{3}(4))}=
                                        \frac{1\odot\tau(\delta D_{3}(2))}{1\odot-1};\\
 \end{aligned}
\end{equation*}

\begin{equation*}
 \begin{aligned}
  \delta D_{4}(2)&=1\text{ since}\\ 
  -1&=\mathbb{P}_{\mathbb{H},f,g,\mu}(C_{1},D_{4},2,3)=
                                        \frac{\gamma C_{1}(2)\odot\tau(\delta D_{4}(2))}
                                             {\gamma C_{1}(3)\odot\tau(\delta D_{4}(3))}=
                                        \frac{1\odot\tau(\delta D_{4}(2))}{1\odot-1};\\
 \end{aligned}
\end{equation*}

\begin{equation*}
 \begin{aligned}
   \delta D_{4}(1)&=1\text{ since}\\
  -1&=\mathbb{P}_{\mathbb{H},f,g,\mu}(C_{3},D_{4},1,2)=
                                        \frac{\gamma C_{3}(1)\odot\tau(\delta D_{4}(1))}
                                             {\gamma C_{3}(2)\odot\tau(\delta D_{4}(2))}=
                                        \frac{-1\odot\tau(\delta D_{4}(1))}{1\odot1};\\                                        
 \end{aligned}
\end{equation*}

\begin{equation*}
 \begin{aligned}
   \gamma C_{2}(3)&=1\text{ since}\\
  -1&=\mathbb{P}_{\mathbb{H},f,g,\mu}(C_{2},D_{4},1,3)=
                                        \frac{\gamma C_{2}(1)\odot\tau(\delta D_{4}(1))}
                                             {\gamma C_{2}(3)\odot\tau(\delta D_{4}(3))}=
                                        \frac{1\odot1}{\gamma C_{2}(3)\odot-1};\\                                        
 \end{aligned}
\end{equation*}

\begin{equation*}
 \begin{aligned}
  \gamma C_{2}(4)&=-f\text{ since}\\
  f^{-1}&=\mathbb{P}_{\mathbb{H},f,g,\mu}(C_{2},D_{2},3,4)=
                                        \frac{\gamma C_{2}(3)\odot\tau(\delta D_{2}(3))}
                                             {\gamma C_{2}(4)\odot\tau(\delta D_{2}(4))}=
                                        \frac{1\odot1}{\gamma C_{2}(4)\odot-1};\\                                  
 \end{aligned}
\end{equation*}

\begin{equation*}
 \begin{aligned}
  \gamma C_{3}(4)&=-g\text{ since}\\
  g^{-1}&=\mathbb{P}_{\mathbb{H},f,g,\mu}(C_{3},D_{3},2,4)=
                                        \frac{\gamma C_{3}(2)\odot\tau(\delta D_{3}(2))}
                                             {\gamma C_{3}(4)\odot\tau(\delta D_{3}(4))}=
                                        \frac{1\odot1}{\gamma C_{3}(4)\odot-1};\\                                     
 \end{aligned}
\end{equation*}

\begin{equation*}
 \begin{aligned}
  \gamma C_{4}(2)&=\mu\odot f\text{ since}\\
  \mu^{-1}\odot f^{-1}&=\mathbb{P}_{\mathbb{H},f,g,\mu}(C_{4},D_{4},1,2)=
                                        \frac{\gamma C_{4}(1)\odot\tau(\delta D_{4}(1))}
                                             {\gamma C_{4}(2)\odot\tau(\delta D_{4}(2))}=
                                        \frac{1\odot1}{\gamma C_{4}(2)\odot1};\\                                     
 \end{aligned}
\end{equation*}

\begin{equation*}
 \begin{aligned}
  \gamma C_{4}(3)&=-\mu\odot g\text{ since}\\
  f\odot g^{-1}&=\mathbb{P}_{\mathbb{H},f,g,\mu}(C_{4},D_{4},2,3)=
                                        \frac{\gamma C_{4}(2)\odot\tau(\delta D_{4}(2))}
                                             {\gamma C_{4}(3)\odot\tau(\delta D_{4}(3))}=
                                        \frac{\mu\odot f\odot1}{\gamma C_{4}(3)\odot-1};\\                                     
 \end{aligned}
\end{equation*}

\begin{equation*}
 \begin{aligned}
  \delta D_{1}(2)&=\tau(g)\text{ since}\\
  g&=\mathbb{P}_{\mathbb{H},f,g,\mu}(C_{1},D_{1},2,4)=
                                        \frac{\gamma C_{1}(2)\odot\tau(\delta D_{1}(2))}
                                             {\gamma C_{1}(4)\odot\tau(\delta D_{1}(4))}=
                                        \frac{1\odot\tau(\delta D_{1}(2))}{1\odot1};\\                                      
 \end{aligned}
\end{equation*}

\begin{equation*}
 \begin{aligned}
  \delta D_{1}(3)&=\tau(f)\text{ since}\\   
  f&=\mathbb{P}_{\mathbb{H},f,g,\mu}(C_{1},D_{1},3,4)=
                                        \frac{\gamma C_{1}(3)\odot\tau(\delta D_{1}(3))}
                                             {\gamma C_{1}(4)\odot\tau(\delta D_{1}(4))}=
                                        \frac{1\odot\tau(\delta D_{1}(3))}{1\odot1};\\                  
 \end{aligned}
\end{equation*}

\begin{equation*}
 \begin{aligned}
  \delta D_{2}(1)&=\tau(\mu)\odot\tau(g)\text{ since}\\ 
  \mu\odot g&=\mathbb{P}_{\mathbb{H},f,g,\mu}(C_{2},D_{2},1,3)=
                                         \frac{\gamma C_{2}(1)\odot\tau(\delta D_{2}(1))}
                                              {\gamma C_{2}(3)\odot\tau(\delta D_{2}(3))}=
                                         \frac{1\odot\tau(\delta D_{2}(1))}{1\odot1};\\
 \end{aligned}
\end{equation*}

\begin{equation*}
 \begin{aligned}
  \delta D_{3}(1)&=-\tau(\mu)\odot\tau(f)\text{ since}\\
   \mu\odot f&=\mathbb{P}_{\mathbb{H},f,g,\mu}(C_{3},D_{3},1,2)=
                                        \frac{\gamma C_{3}(1)\odot\tau(\delta D_{3}(1))}
                                             {\gamma C_{3}(2)\odot\tau(\delta D_{3}(2))}=
                                        \frac{-1\odot\tau(\delta D_{3}(1))}{1\odot1}.\\
 \end{aligned}
\end{equation*}

\subsection{Technical lemma for Theorem~\ref{onetoone}}

\begin{lemma}\label{TLL}
 Let 
 $\eta_{(C_{i_{1}}C_{i_{2}}|C_{i_{3}}C_{i_{4}})}$ 
 be $\mathbb{H}^{\ast}\text{-valued}$ functions 
 defined for circuits $C_{i_{1}}$, $C_{i_{2}}$, $C_{i_{3}}$, $C_{i_{4}}\in\mathfrak{C}$ 
 such that
 \begin{itemize}
  \item $\operatorname{dim}(C_{i_{1}}\cup C_{i_{2}}\cup C_{i_{3}}\cup C_{i_{4}})=1$;
  \item $\{C_{i_{1}},C_{i_{2}}\}\cap\{C_{i_{3}},C_{i_{4}}\}=\emptyset$;
 \end{itemize}
 and satisfying
 \begin{enumerate}[label=(pp\arabic{*})]
  \item \label{pp1} $\eta_{(C_{i_{1}}C_{i_{2}}|C_{i_{3}}C_{i_{3}})}=1$;
  \item \label{pp2} $\eta_{(C_{i_{1}}C_{i_{2}}|C_{i_{3}}C_{i_{4}})}=\psi_{(C_{i_{3}}C_{i_{4}}|C_{i_{1}}C_{i_{2}})}$;
  \item \label{pp3} $\eta_{(C_{i_{1}}C_{i_{2}}|C_{i_{3}}C_{i_{4}})}=\psi_{(C_{i_{1}}C_{i_{2}}|C_{i_{4}}C_{i_{3}})}^{-1}$.
 \end{enumerate}
 Consider a permutation 
 \begin{equation*}
  \sigma=\left(\begin{array}{cccc}
                     1    &     2     &     3     &     4     \\
                \sigma(1) & \sigma(2) & \sigma(3) & \sigma(4) \\
               \end{array}
         \right)\in\mathcal{S}_{4}
 \end{equation*}
Then, relation 
\begin{equation*}
 1\in\eta_{(C_{i_{1}}C_{i_{2}}|C_{i_{3}}C_{i_{4}})}\boxplus\eta_{(C_{i_{4}}C_{i_{2}}|C_{i_{3}}C_{i_{1}})}
\end{equation*}
is equivalent to
\begin{equation*}
 1\in\eta_{(C_{i_{\sigma(1)}}C_{i_{\sigma(2)}}|C_{i_{\sigma(3)}}C_{i_{\sigma(4)}})}\boxplus
     \eta_{(C_{i_{\sigma(4)}}C_{i_{\sigma(2)}}|C_{i_{\sigma(3)}}C_{i_{\sigma(1)}})}
\end{equation*}
\end{lemma}

\begin{proof}
 It is enough to verify our statement for $\sigma$ being an adjacent transposition, since they generate the group $\mathcal{S}_{4}$.
 This can be done by a direct check using properties \ref{pp1}, \ref{pp2}, \ref{pp3}. We illustrate the check for the adjacent transposition $\left(\begin{array}{cc}1 & 2 \\ 2 & 1\end{array}\right)$ leaving the others to the reader.
        \begin{equation*}
         \begin{array}{l}
          1\in\eta_{(C_{i_{2}}C_{i_{1}}|C_{i_{3}}C_{i_{4}})}\boxplus\eta_{(C_{i_{4}}C_{i_{1}}|C_{i_{3}}C_{i_{2}})}
          \Longleftrightarrow \\
          1\in\eta_{(C_{i_{1}}C_{i_{2}}|C_{i_{3}}C_{i_{4}})}^{-1}\boxplus\eta_{(C_{i_{1}}C_{i_{4}}|C_{i_{2}}C_{i_{3}})}
          \Longleftrightarrow \\
          1\in\eta_{(C_{i_{1}}C_{i_{2}}|C_{i_{3}}C_{i_{4}})}^{-1}
              \left(
               -\eta_{(C_{i_{1}}C_{i_{2}}|C_{i_{3}}C_{i_{4}})}^{-1}\odot\eta_{(C_{i_{1}}C_{i_{3}}|C_{i_{2}}C_{i_{4}})}
              \right). \\
         \end{array}
        \end{equation*}
        Multiplying both sides by $-\eta_{(C_{i_{1}}C_{i_{2}}|C_{i_{3}}C_{i_{4}})}$ 
        and using the distributive law this is the same as 
        \begin{equation*}
         -\eta_{(C_{i_{1}}C_{i_{2}}|C_{i_{3}}C_{i_{4}})}\in-1\boxplus\eta_{(C_{i_{1}}C_{i_{3}}|C_{i_{2}}C_{i_{4}})}.
        \end{equation*}
        From \ref{H3} we then find 
        \begin{equation*}
         \begin{array}{l}
          1\in\eta_{(C_{i_{1}}C_{i_{2}}|C_{i_{3}}C_{i_{4}})}\boxplus\eta_{(C_{i_{1}}C_{i_{3}}|C_{i_{2}}C_{i_{4}})}
          \Longleftrightarrow \\
          1\in\eta_{(C_{i_{1}}C_{i_{2}}|C_{i_{3}}C_{i_{4}})}\boxplus\eta_{(C_{i_{4}}C_{i_{2}}|C_{i_{3}}C_{i_{1}})} \\
         \end{array}
        \end{equation*}
        
\end{proof}

\section{On Tutte groups of matroids}\label{section4}\label{App:TG}

We recall the inner Tutte group of a matroid and describe different presentations of this group by generators and relations. The material is borrowed mainly from \cite{DW89, DW90, GRS95,Sai16}.

\subsection{Definition of the inner Tutte group}\label{section4a}

From now on, let $M$ be a matroid with circuits set $\mathfrak{C}$ and cocircuits set $\mathfrak{C}^{\ast}$.
In this section we consider sets of the form 
$F=C_{1}\cup\cdots\cup C_{k}$ with $C_{i}\in\mathfrak{C}$.
Given a maximal chain
\begin{equation*}
 \emptyset\subset F_{0}\subset F_{1}\subset\cdots\subset F_{d}=F
\end{equation*}
the number $d$ depends 
only on $F$ and it is called the \emph{dimension} of $F$. We denote it by
$\operatorname{dim}(F)$. Notice that 
\begin{equation*}
 \operatorname{dim}(F)=|F|-\operatorname{rk}(F)-1.
\end{equation*}

\begin{definition}\label{tutteCDdef}
 The group $\mathbb{T}_{M}^{\mathfrak{C},\mathfrak{C}^{\ast}}$ 
 is defined to be the multiplicative abelian group with 
 formal generators given by the symbols 
 \begin{itemize}
  \item $\epsilon_{M}$;
  \item $C(x)$ for $C\in\mathfrak{C}$ and $x\in C$;
  \item $D(y)$ for $D\in\mathfrak{C}^{\ast}$ and $y\in D$;
 \end{itemize}
 with relations
 \begin{itemize}
  \item $\epsilon_{M}^{2}=1$;
  \item $C(x)D(x)=\epsilon_{M}C(y)D(y)$ for $C\in\mathfrak{C}$, $D\in\mathfrak{C}^{\ast}$ with $\{x,y\}=C\cap D$.
 \end{itemize}
\end{definition}

\begin{definition}[Tutte group]\label{tuttedef}
 The \emph{Tutte group} $\mathbb{T}_{M}$ is 
 the subgroup of $\mathbb{T}_{M}^{\mathfrak{C},\mathfrak{C}^{\ast}}$ generated by 
 \begin{itemize}
  \item $\epsilon_{M}$;
  \item $C(x)C(y)^{-1}$ for $C\in\mathfrak{C}$, $x$, $y\in C$;
  \item $D(x)D(y)^{-1}$ for $D\in\mathfrak{C}^{\ast}$, $x$, $y\in D$.
 \end{itemize}
\end{definition}

\begin{definition}[Inner Tutte group]\label{innerdef}
 Let us consider the group homomorphism
 $\Lambda_{M}:\mathbb{T}_{M}^{\mathfrak{C},\mathfrak{C}^{\ast}}\longrightarrow
  \mathbb{Z}^{|E|}\times\mathbb{Z}^{|\mathfrak{C}|}\times\mathbb{Z}^{|\mathfrak{C}^{\ast}|}$ defined by 
 \begin{itemize}
  \item $\epsilon_{M}\mapsto0$;
  \item $C(x)\mapsto(\mathds{1}_{x},\mathds{1}_{C},0)$;
  \item $D(y)\mapsto(-\mathds{1}_{y},0,\mathds{1}_{D})$;
 \end{itemize}
 where $\mathds{1}$ is the indicator function.
 The \emph{inner Tutte group} 
 $\mathbb{T}_{M}^{(0)}$ of $M$ is the kernel of  the homomorphism $\Lambda_{M}$. 
\end{definition}

\begin{remark}\label{easygen}
 According to \cite[Proposition 2.9 (i)]{Wen89a} the inner Tutte group of a matroid $M$ is generated by $\epsilon_{M}$ and 
 all products of the form 
 \begin{equation*}
  \frac{C_{1}(x)C_{2}(y)}{C_{1}(y)C_{2}(x)}
 \end{equation*}
 for circuits $C_{1}$, $C_{2}\in\mathfrak{C}$ with $\operatorname{dim}(C_{1}\cup C_{2})=1$ and $x$, $y\in C_{1}\cap C_{2}$.
\end{remark}

One can see that 
$\mathbb{T}_{M}^{(0)}\vartriangleleft\mathbb{T}_{M}\vartriangleleft\mathbb{T}_{M}^{\mathfrak{C},\mathfrak{C}^{\ast}}$. 
If $c_{M}$ is the number of connected components of $M$, with \cite[Theorem 1.5]{DW89} we have
\begin{equation}\label{ZFree}
 \begin{array}{ll}
  \mathbb{T}_{M}\cong\mathbb{T}_{M}^{(0)}
                     \times
                     \mathbb{Z}^{|E|-c_{M}}; & 
  \mathbb{T}_{M}^{\mathfrak{C},\mathfrak{C}^{\ast}}\cong\mathbb{T}_{M}^{(0)}
                                                        \times
                                                        \mathbb{Z}^{|E|-c_{M}}
                                                        \times
                                                        \mathbb{Z}^{|\mathfrak{C}|}
                                                        \times
                                                        \mathbb{Z}^{|\mathfrak{C}^{\ast}|}.
 \end{array}                                                         
\end{equation}

Thus, any of these groups is known as soon as we know 
$\mathbb{T}_{M}^{(0)}$.

\subsection{Presentations of the inner Tutte group}\label{section4b}

Throughout this section we assume that, for a matroid $M$ with set of circuits $\mathfrak{C}$, an enumeration 
$\{C_{j}\}_{j\in J}$ of $\mathfrak{C}$ and an arbitrary total order $<_{J}$ on $J$ are given. 

\begin{theorem}[{\cite[Theorem 2.1]{Sai16}}]\label{isoinnerTuttegroups}
 The groups $\mathbb{T}_{M}^{(1)}$, $\mathbb{T}_{M}^{(2)}$ and $\mathcal{T}_{M,<_{J}}^{(0)}$ introduced below in 
 Definition \ref{innerTuttegroupTM1}, Definition \ref{innerTuttegroupTM2} and Definition \ref{innerTuttegroupTMJ0}
 are all isomorphic to the inner Tutte group $\mathbb{T}_{M}^{(0)}$.
\end{theorem}

\begin{definition}\label{innerTuttegroupTM1}
Given a matroid $M$ the group $\mathbb{T}_{M}^{(1)}$ 
is the finitely generated abelian group with generators
\begin{enumerate}[label=(g\arabic{*})]
 \item \label{gen1} $\sigma_{M}$;
 \item \label{gen2} $|C\enskip D,x\enskip y|$, 
                    where $C\in\mathfrak{C}$, $D\in\mathfrak{C}^{\ast}$,  and $x$, $y\in C\cap D$;
\end{enumerate}
and relations
\begin{enumerate}[label=(r\arabic{*})]
 \item \label{rel1} $\sigma_{M}^{2}=1$;
 \item \label{rel2} $|C\enskip D,x\enskip x|=1_{\mathbb{T}_{M}^{(1)}}$;
 \item \label{rel3} $|C\enskip D,x\enskip y||C\enskip D,y\enskip z||C\enskip D,z\enskip x|=1_{\mathbb{T}_{M}^{(1)}}$;
 \item \label{rel4} $|C_{1}\enskip D_{1},x\enskip y||C_{2}\enskip D_{2},x\enskip y|=
                    |C_{1}\enskip D_{2},x\enskip y||C_{2}\enskip D_{1},x\enskip y|$;
 \item \label{rel5} $|C\enskip D,x\enskip y|=\sigma_{M}$ if $C\cap D=\{x,y\}$;
 \item \label{rel6} $|C_{1}\enskip D_{1},x_{2}\enskip x_{3}||C_{2}\enskip D_{2},x_{1}\enskip x_{4}|$ 
                  if 
                  \begin{itemize}
                   \item $\operatorname{dim}(C_{1}\cup C_{2})=1$;
                   \item $C_{1}\cap D_{1}=\{x_{2},x_{3},x_{4}\}$;
                   \item $C_{2}\cap D_{2}=\{x_{1},x_{3},x_{4}\}$;
                   \item $C_{1}\cap D_{2}=C_{2}\cap D_{1}=\{x_{3},x_{4}\}$.
                  \end{itemize}
\end{enumerate}
\end{definition}

\begin{remark}[{\cite[Theorem 3]{GRS95}}]
Notice that the ``natural" assignment $\sigma_M\mapsto \epsilon_M$ and $\vert C,D,x,y \vert \mapsto C(x)D(x)C(y)^{-1}D(y)^{-1}$ defines an isomorphism between $\mathbb{T}_{M}^{(1)}$ and $\mathbb{T}_{M}^{(0)}$.
\end{remark}

\begin{definition}\label{innerTuttegroupTM2}
 Given a matroid $M$ the group $\mathbb{T}_{M}^{(2)}$ is the finitely generated abelian group 
 with formal generators given by the symbols 
\begin{enumerate}[label=(G\arabic{*})]
 \item \label{G1} $\xi_{M}$;
 \item \label{G2} $[C_{i_{1}} C_{i_{2}}|C_{i_{3}} C_{i_{4}}]$, 
                  where $C_{i_{1}}$, $C_{i_{2}}$, $C_{i_{3}}$, $C_{i_{4}}\in\mathfrak{C}$ are circuits of $M$ such that
                  $L=C_{i_{1}}\cup C_{i_{2}}\cup C_{i_{3}}\cup C_{i_{4}}=C_{i_{k}}\cup C_{i_{l}}$ for $k=1,2$, $l=3,4$, 
                  $\operatorname{dim}(L)=1$;
\end{enumerate}
and relations
\begin{enumerate}[label=(R\arabic{*})]
 \item \label{R1} $\xi_{M}^{2}=1_{\mathbb{T}_{M}^{(2)}}$;
 \item \label{R2} $\xi_{M}=1_{\mathbb{T}_{M}^{(2)}}$ if $M$ has minors of Fano or dual-Fano type;
 \item \label{R3} $[C_{i_{1}}C_{i_{2}}|C_{i_{3}}C_{i_{3}}]=1_{\mathbb{T}_{M}^{(2)}}$;
 \item \label{R4} $[C_{i_{1}}C_{i_{2}}|C_{i_{3}}C_{i_{3}}]=[C_{i_{3}}C_{i_{4}}|C_{i_{1}}C_{i_{2}}]$;
 \item \label{R5} $[C_{i_{1}}C_{i_{2}}|C_{i_{3}}C_{i_{4}}]
                  [C_{i_{1}}C_{i_{2}}|C_{i_{4}}C_{i_{5}}]
                  [C_{i_{1}}C_{i_{2}}|C_{i_{5}}C_{i_{3}}]=1_{\mathbb{T}_{M}^{(2)}}$;
 \item \label{R6} $[C_{i_{1}}C_{i_{2}}|C_{i_{3}}C_{i_{4}}]
                  [C_{i_{1}}C_{i_{4}}|C_{i_{2}}C_{i_{3}}]
                  [C_{i_{1}}C_{i_{3}}|C_{i_{4}}C_{i_{2}}]=\xi_{M}$;
 \item \label{R7}$[C_{i_{1}}C_{i_{2}}|C_{i_{6}}C_{i_{9}}]
                  [C_{i_{2}}C_{i_{3}}|C_{i_{4}}C_{i_{7}}]
                  [C_{i_{3}}C_{i_{1}}|C_{i_{5}}C_{i_{8}}]=1_{\mathbb{T}_{M}^{(2)}}$ 
                for every family of circuits $\{C_{i_{1}},\ldots,C_{i_{9}}\}\subseteq\mathfrak{C}$ such that:
               \begin{itemize}
                 \item $\operatorname{dim}(L_{i_{p}})=1$ for $L_{i_p}=C_{i_{q}}\cup C_{i_{r}}$, 
                       where $\{p,q,r\}=\{1,2,3\}$;
                 \item $\operatorname{dim}(P)=2$ where $P=C_{i_{1}}\cup C_{i_{2}}\cup C_{i_{3}}$;
                 \item $C_{i_{s+3}},C_{i_{s+6}}\subseteq L_{i_{s}}$ for $s=1,2,3$;
                 \item $\operatorname{dim}(L_{i_{h}})=1$ for $L_{i_{h}}=C_{i_{3+h}}\cup C_{i_{4+h}}\cup C_{i_{5+h}}$, 
                       $h\in\{1,4\}$;
                 \item $\{C_{i_{1}},C_{i_{2}},C_{i_{3}}\}\cap\{C_{i_{4}},\ldots,C_{i_{9}}\}=\emptyset$.
                \end{itemize}
\end{enumerate}
\end{definition}

\begin{remark}[{\cite[Theorem 4]{GRS95}}]\label{iso20}
The assignment $\xi_M\mapsto \epsilon_M$ and $[C_1,C_2\vert C_3,C_4]\mapsto C_1(x_3)C_1(x_4)^{-1} C_2(x_4)C_2(x_3)^{-1}$, where $x_i\in (\bigcup_{j\in [4]} C_j)\setminus C_i$, defines an isomorphism between $\mathbb{T}_{M}^{(2)}$ and $\mathbb{T}_{M}^{(0)}$.
\end{remark}

\begin{definition}[{\cite[Definition 2.1]{Sai16}}]\label{innerTuttegroupTMJ0}
 We denote by $\mathcal{T}_{M,<_{J}}^{(0)}$ 
 the multiplicative abelian group with formal generators given by the symbols
 \begin{enumerate}[label=(Q\arabic{*})]
  \item \label{Q1} $\eta_{M,<_{J}}$;
  \item \label{Q2} $(C_{j_{1}}C_{j_{2}}|C_{j_{3}}C_{j_{4}})$ 
                   where $C_{j_{1}}$, $C_{j_{2}}$, $C_{j_{3}}$, $C_{j_{4}}\in\mathfrak{C}$ are circuits of $M$ such that
                   $\operatorname{dim}(C_{j_{1}}\cup C_{j_{2}}\cup C_{j_{3}}\cup C_{j_{4}})=1$ 
                   and $j_{1}<j_{2}$, $j_{3}<j_{4}$, $j_{1}<j_{3}$;
 \end{enumerate}
 and relations
  \begin{enumerate}[label=(S\arabic{*})]
   \item \label{S1} $\eta_{M,<_{J}}^{2}=1_{\mathcal{T}_{M,<_{J}}^{(0)}}$;
   \item \label{S2} $\eta_{M,<_{J}}=1_{\mathcal{T}_{M,<_{J}}^{(0)}}$ if $M$ has minors of Fano or dual-Fano type;
   \item \label{S3} $(C_{j_{1}}C_{j_{2}}|C_{j_{3}}C_{j_{4}})
                     (C_{j_{1}}C_{j_{4}}|C_{j_{2}}C_{j_{3}})
                     (C_{j_{1}}C_{j_{3}}|C_{j_{2}}C_{j_{4}})^{-1}=\eta_{M,<_{J}}$ 
                    for any family of circuits $\{C_{j_{1}},C_{j_{2}},C_{j_{3}},C_{j_{4}}\}\subseteq\mathfrak{C}$ with               
                    $\operatorname{dim}(C_{j_{1}}\cup C_{j_{2}}\cup C_{j_{3}}\cup C_{j_{4}})=1$ 
                    and $j_{1}<j_{2}<j_{3}<j_{4}$;
  \item \label{S4} \begin{equation*}
                      \left\lbrace
                     \begin{array}{l}
                 (C_{j_{1}}C_{j_{2}}|C_{j_{3}}C_{j_{4}})
                 (C_{j_{1}}C_{j_{2}}|C_{j_{4}}C_{j_{5}})
                 (C_{j_{1}}C_{j_{2}}|C_{j_{3}}C_{j_{5}})^{-1}=1_{\mathcal{T}_{M,<_{J}}^{(0)}}\\
                 
                 (C_{j_{1}}C_{j_{3}}|C_{j_{2}}C_{j_{4}})
                 (C_{j_{1}}C_{j_{3}}|C_{j_{4}}C_{j_{5}})
                 (C_{j_{1}}C_{j_{3}}|C_{j_{2}}C_{j_{5}})^{-1}=1_{\mathcal{T}_{M,<_{J}}^{(0)}}\\
                 
                 (C_{j_{1}}C_{j_{4}}|C_{j_{2}}C_{j_{3}})
                 (C_{j_{1}}C_{j_{4}}|C_{j_{3}}C_{j_{5}})
                 (C_{j_{1}}C_{j_{4}}|C_{j_{2}}C_{j_{5}})^{-1}=1_{\mathcal{T}_{M,<_{J}}^{(0)}}\\
                 
                 (C_{j_{1}}C_{j_{5}}|C_{j_{2}}C_{j_{3}})
                 (C_{j_{1}}C_{j_{5}}|C_{j_{3}}C_{j_{4}})
                 (C_{j_{1}}C_{j_{5}}|C_{j_{2}}C_{j_{4}})^{-1}=1_{\mathcal{T}_{M,<_{J}}^{(0)}}\\
                 
                 (C_{j_{1}}C_{j_{4}}|C_{j_{2}}C_{j_{3}})
                 (C_{j_{2}}C_{j_{3}}|C_{j_{4}}C_{j_{5}})
                 (C_{j_{1}}C_{j_{5}}|C_{j_{2}}C_{j_{3}})^{-1}=1_{\mathcal{T}_{M,<_{J}}^{(0)}}\\
                 
                 (C_{j_{1}}C_{j_{3}}|C_{j_{2}}C_{j_{4}})
                 (C_{j_{2}}C_{j_{4}}|C_{j_{3}}C_{j_{5}})
                 (C_{j_{1}}C_{j_{5}}|C_{j_{2}}C_{j_{3}})^{-1}=1_{\mathcal{T}_{M,<_{J}}^{(0)}}\\
                 
                 (C_{j_{1}}C_{j_{3}}|C_{j_{2}}C_{j_{5}})
                 (C_{j_{2}}C_{j_{5}}|C_{j_{3}}C_{j_{4}})
                 (C_{j_{1}}C_{j_{4}}|C_{j_{2}}C_{j_{5}})^{-1}=1_{\mathcal{T}_{M,<_{J}}^{(0)}}\\
                 
                 (C_{j_{1}}C_{j_{2}}|C_{j_{3}}C_{j_{4}})
                 (C_{j_{2}}C_{j_{5}}|C_{j_{3}}C_{j_{4}})
                 (C_{j_{1}}C_{j_{5}}|C_{j_{3}}C_{j_{4}})^{-1}=1_{\mathcal{T}_{M,<_{J}}^{(0)}}\\
                 
                 (C_{j_{1}}C_{j_{2}}|C_{j_{3}}C_{j_{5}})
                 (C_{j_{2}}C_{j_{4}}|C_{j_{3}}C_{j_{5}})
                 (C_{j_{1}}C_{j_{4}}|C_{j_{3}}C_{j_{5}})^{-1}=1_{\mathcal{T}_{M,<_{J}}^{(0)}}\\
                 
                 (C_{j_{1}}C_{j_{2}}|C_{j_{4}}C_{j_{5}})
                 (C_{j_{2}}C_{j_{3}}|C_{j_{4}}C_{j_{5}})
                 (C_{j_{1}}C_{j_{3}}|C_{j_{4}}C_{j_{5}})^{-1}=1_{\mathcal{T}_{M,<_{J}}^{(0)}}\\
               \end{array}
               \right.
              \end{equation*}
              where $C_{j_{1}}$, $C_{j_{2}}$, $C_{j_{3}}$, $C_{j_{4}}$, $C_{j_{5}}\in\mathfrak{C}$ are circuits of $M$ with 
              the properties that
              $\operatorname{dim}(C_{j_{1}}\cup C_{j_{2}}\cup C_{j_{3}}\cup C_{j_{4}}\cup C_{j_{5}})=1$ 
              and $j_{1}<j_{2}<j_{3}<j_{4}<j_{5}$;
  \item \label{S5} $\langle C_{j_{1}}C_{j_{2}}|C_{j_{6}}C_{j_{9}}\rangle
               \langle C_{j_{2}}C_{j_{3}}|C_{j_{4}}C_{j_{7}}\rangle
               \langle C_{j_{3}}C_{j_{1}}|C_{j_{5}}C_{j_{8}}\rangle=1_{\mathcal{T}_{M,<_{J}}^{(0)}}$
              for any family of circuits $\{C_{i_{1}},\ldots,C_{i_{9}}\}\subseteq\mathfrak{C}$ as in \ref{R7} with the 
              extra conditions:
              \begin{enumerate}[label=(O\arabic{*})]
                \item \label{O1} $j_{1}<j_{2}<j_{3}$;
                \item \label{O2} $j_{4}\geq j_{7}$, $j_{5}\geq j_{8}$ and $j_{6}\geq j_{9}$ do not all hold at the same time.
               \end{enumerate}
               Here $\langle C_{d_{1}}C_{d_{2}}|C_{d_{3}}C_{d_{4}}\rangle$ are the symbols given by the formula
               \begin{equation}\tag{$\ast$}\label{FM}
                \langle C_{d_{1}}C_{d_{2}}|C_{d_{3}}C_{d_{4}}\rangle=
                   \left\lbrace
                      \begin{array}{lcccc}
                         1_{\mathcal{T}_{M,<_{J}}^{(0)}} & \text{if} & d_{1}=d_{2} & \text{or}   & d_{3}=d_{4}\\
            (C_{d_{1}}C_{d_{2}}|C_{d_{3}}C_{d_{4}})      & \text{if} & d_{1}<d_{2} & d_{3}<d_{4} & d_{1}<d_{3} \\
            (C_{d_{3}}C_{d_{4}}|C_{d_{1}}C_{d_{2}})      & \text{if} & d_{1}<d_{2} & d_{3}<d_{4} & d_{3}<d_{1} \\
            (C_{d_{1}}C_{d_{2}}|C_{d_{4}}C_{d_{3}})^{-1} & \text{if} & d_{1}<d_{2} & d_{4}<d_{3} & d_{1}<d_{4} \\
            (C_{d_{4}}C_{d_{3}}|C_{d_{1}}C_{d_{2}})^{-1} & \text{if} & d_{1}<d_{2} & d_{4}<d_{3} & d_{4}<d_{1} \\
            (C_{d_{2}}C_{d_{1}}|C_{d_{3}}C_{d_{4}})^{-1} & \text{if} & d_{2}<d_{1} & d_{3}<d_{4} & d_{2}<d_{3} \\
            (C_{d_{3}}C_{d_{4}}|C_{d_{2}}C_{d_{1}})^{-1} & \text{if} & d_{2}<d_{1} & d_{3}<d_{4} & d_{3}<d_{2} \\
            (C_{d_{2}}C_{d_{1}}|C_{d_{4}}C_{d_{3}})      & \text{if} & d_{2}<d_{1} & d_{4}<d_{3} & d_{2}<d_{4} \\
            (C_{d_{4}}C_{d_{3}}|C_{d_{2}}C_{d_{1}})      & \text{if} & d_{2}<d_{1} & d_{4}<d_{3} & d_{4}<d_{2} \\
                      \end{array}
                   \right.
               \end{equation}
               and defined for any family of circuits $\{C_{d_{1}},C_{d_{2}},C_{d_{3}},C_{d_{1}}\}\subseteq\mathfrak{C}$
               satisfying $\operatorname{dim}(C_{d_{1}}\cup C_{d_{2}}\cup C_{d_{3}}\cup C_{d_{4}})=1$ and 
               $\{C_{d_{1}},C_{d_{2}}\}\cap\{C_{d_{3}},C_{d_{4}}\}=\emptyset$.
  \end{enumerate}
\end{definition}

\begin{remark}
 If $\{C_{h}\}_{h\in H}$ is another enumeration of $\mathfrak{C}$ with total order $<_{H}$ on $H$, 
 then the groups $\mathcal{T}_{M,<_{H}}^{(0)}$ and $\mathcal{T}_{M,<_{J}}^{(0)}$ are isomorphic. 
 As a consequence, the number of generators of  $\mathcal{T}_{M,<_{J}}^{(0)}$ that are of the form $(C_{j_{1}}C_{j_{2}}|C_{j_{3}}C_{j_{4}})$ is properly defined, neither depending on the choice of the enumeration of circuits of $M$  nor the total ordering of such enumeration. 
\end{remark}

\begin{remark}\label{iso2t}
The assignment $\xi_M\mapsto \eta_{M,<_J}$ and $[C_{i_1}C_{i_2}\vert C_{i_3}C_{i_4}]\mapsto \langle C_{i_1}C_{i_2}\vert C_{i_3}C_{i_4}\rangle $ defines an isomorphism between $\mathbb{T}^{(2)}_M$ and $\mathcal{T}_{M,<_{J}}^{(0)}$. 
\end{remark}

   \section{Technical proofs of Section \ref{ss:realizable}}\label{appendixBBBYYY}
   
   Let us throughout this section fix integers $m,d$ such that   $2\leq d\leq m-2$.  Recall that we consider the uniform matroid $U_d(m)$ on the ground set $E=[m]$.

\begin{definition}\label{def:V}
We denote by $\mathcal{V}$ the set of matrices $A\in M_{d,m}(\mathbb{C})$ of 
the form 
\begin{equation}
    \label{canonicalformNumbers}
  \left(
          \begin{array}{cccc|cccc}
             1   & \quad  & \quad  & \quad &    1   & \ast   & \cdots  & \ast  \\
           \quad & \ddots & 0 & \quad & \vdots &   \vdots      & \quad   &  \vdots     \\
           \quad & 0  & \ddots & \quad &    1   & \ast & \cdots  & \ast \\
           \quad & \quad  & \quad  &   1   &    1   &      1        & \cdots  &   1         \\
          \end{array}
    \right).
 \end{equation}
Moreover, let $\mathcal{U}'\subseteq \mathcal{V}$ be the subset defined by those matrices for which the 
expressions \eqref{dadostar} are well-defined (that is, the determinants in the denominators are non-zero) and  let $\mathcal{U}$ denote the subset of $\mathcal{V}$ of matrices representing the matroid $U_d(m)$.
\end{definition}
Obviously $\mathcal{V}$ can be diffeomorphically 
identified with $\mathbb{C}^{(d-1)(m-d-1)}$, and we have the chain of inclusions 
$\mathcal{U}\subseteq\mathcal{U}'\subseteq \mathcal{V}$. 

\begin{lemma}
 \label{OpenConnectedSide}
  $\mathcal{U}$ is a non-empty open connected subspace of $\mathcal{V}$.
\end{lemma}
 
\begin{proof}
 First of all we prove that $\mathcal{U}$ is non-empty. Since $U_{d}(m)$ is realizable over $\mathbb{C}$ there exists 
 a matrix $A\in M_{d,m}(\mathbb{C})$ with
  \begin{equation*}
       \begin{array}{ccc}
        \det(A^{i_{1}},\ldots,A^{i_{d}})\neq0 & \text{for all} & 1\leq i_{1}<\cdots<i_{d}\leq m, \\
       \end{array}
      \end{equation*}
 where $A^{j}$ denotes the $j\text{-th}$ column of $A$. By elementary linear algebra arguments, there exist matrices 
 $B\in GL_{d}(\mathbb{C})$ and $D=\operatorname{diag}(\delta_{1},\ldots,\delta_{m})\in GL_{m}(\mathbb{C})$ such that the 
 new matrix
 $\widetilde{A}=BAD$ is of the form \eqref{canonicalformNumbers}. 
 Since multiplying a matrix $A$ on the left or on the right side by 
 a non-singular matrix does not change the underlying matroid of $\varphi_{A}$, we conclude that $\widetilde{A}\in\mathcal{U}$, 
 proving that $\mathcal{U}\neq\emptyset$.
 
 Hence, it remains to check that $\mathcal{U}$ is an open connected subspace of $\mathcal{V}$. By definition of  $\mathcal{U}$,
  \begin{equation*}
  \mathcal{U}=\left\lbrace
              A\in\mathcal{V}
              \left|
              \prod_{1\leq i_{1}<\cdots<i_{d}\leq m}
              \det(A^{i_{1}},\ldots,A^{i_{d}})\neq0
              \right.
              \right\rbrace
 \end{equation*}
 and the claim comes from Lemma \ref{LEMLEMLEM} below via the identification of $\mathcal{V}$ with the space
 $\mathbb{C}^{(d-1)(m-d-1)}$.
 Notice that in the polynomial ring with variables corresponding to entries of the matrix none of the $\det(A^{i_{1}},\ldots,A^{i_{d}})\text{'s}$ is the zero polynomial, 
 since $\mathcal{U}\neq\emptyset$. Then, the product 
 $\prod_{1\leq i_{1}<\cdots<i_{d}\leq m}\det(A^{i_{1}},\ldots,A^{i_{d}})$
 is also not the zero polynomial and the hypothesis of Lemma \ref{LEMLEMLEM} are fulfilled.
\end{proof}

\begin{lemma}\label{LEMLEMLEM}
  Let $h\geq1$ and let $F\in\mathbb{C}\left[t_{1},\ldots,t_{h}\right]$ be a polynomial.
  If $F$ is not the zero polynomial, the complement manifold $\mathbb{C}^{h}\setminus\{F=0\}$ 
  is a non-empty open connected subset of 
  $\mathbb{C}^{h}$.
 \end{lemma}
 \begin{proof}
 Since $F$ is not the zero polynomial, the complement is non-empty.
 The regular part of the zero set, namely
 $\{z\in\mathbb{C}^{h}\mid F(z)=0,\operatorname{d}F(z)\neq0\}$
 has complex codimension $1$, thus real codimension $2$. So that its complement is connected. 
 The singular part 
 $\{z\in\mathbb{C}^{h}\mid F(z)=0\text,\operatorname{d}F(z)=0\}$
 is of higher codimension. 
 There are only finitely many such parts, thus the complement is connected.
 \end{proof}

\begin{corollary}\label{cor:nonemptyconnected}
The set $\mathcal{U}'$ is a non-empty open connected subsets of $\mathcal{V}$.
\end{corollary}
\begin{proof}
The proof is analogous to that of Lemma \ref{OpenConnectedSide}.
\end{proof}

\begin{lemma}
 \label{LastSide}
 Let $m,d\in\mathbb{N}$ and assume $2\leq d\leq m-2$. For $d+2\leq i\leq m$ and $1\leq j\leq d-1$ 
 consider a set of variables $\{a_{i}^{j}\}$ and the matrix $A$ defined by
 \begin{equation}
    \label{canonicalform}
  A=\left(
          \begin{array}{cccc|cccc}
             1   & \quad  & \quad  & \quad &    1   & a_{1}^{d+2}   & \cdots  & a_{1}^{m}   \\
           \quad & \ddots & 0      & \quad & \vdots &   \vdots      & \quad   &  \vdots     \\
           \quad & 0      & \ddots & \quad &    1   & a_{d-1}^{d+2} & \cdots  & a_{d-1}^{m} \\
           \quad & \quad  & \quad  &   1   &    1   &      1        & \cdots  &   1         \\
          \end{array}
    \right).
 \end{equation}
 Let $\{i_{1},i_{2},i_{3},i_{4}\}$ and $\{j_{1},\ldots,j_{d-2}\}$ be two subsets of $\{1,\ldots,m\}$ such that
 \begin{enumerate}
  \item[--] $|\{i_{1},i_{2},i_{3},i_{4}\}|=4$;
  \item[--] $|\{j_{1},\ldots,j_{d-2}\}|=d-2$;
  \item[--] $\{i_{1},i_{2},i_{3},i_{4}\}\cap\{j_{1},\ldots,j_{d-2}\}=\emptyset$.
 \end{enumerate}
 Now, let $p$ and $q$ be the elements of the polynomial ring $\mathbb{C}\left[a_{i}^{j}\right]$ 
 defined by
 \begin{enumerate}
  \item[] $p=\det(A^{i_{2}}|A^{i_{4}}|A^{j_{1}}|\cdots|A^{j_{d-2}})
           \det(A^{i_{1}}|A^{i_{3}}|A^{j_{1}}|\cdots|A^{j_{d-2}})$;
  \item[] $q=\det(A^{i_{2}}|A^{i_{3}}|A^{j_{1}}|\cdots|A^{j_{d-2}})
           \det(A^{i_{1}}|A^{i_{4}}|A^{j_{1}}|\cdots|A^{j_{d-2}})$;
 \end{enumerate}
 where $A^{l}$ denotes the $l\text{-th}$ column of $A$. The following properties hold:
 \begin{enumerate}
  \item $p$ and $q$ are not the zero polynomial;
  \item There exist matrices $B_{1},B_{2}\in M_{d,m}(\mathbb{C})$ of the form \eqref{canonicalform} such that 
        \begin{equation*}
         \begin{array}{cccc}
          p(B_{1})\neq0 & q(B_{1})\neq0 & p(B_{2})\neq0 & q(B_{2})\neq0 \\
         \end{array}
        \end{equation*}
        and 
        \begin{equation*}
         \frac{p(B_{1})}{q(B_{1})}\neq\frac{p(B_{2})}{q(B_{2})}
        \end{equation*}
        where, with a slight abuse of notation, we write $p(B)$ and $q(B)$ for the evaluation of the polynomials 
        $p$ and $q$ on the entries of the matrix $B$.
 \end{enumerate}
\end{lemma}

\begin{proof}
In order to prove property (1), notice that we
already know that $\mathcal{U}\neq\emptyset$ so that there exists a matrix $B$ of the form
\eqref{canonicalform} which belongs to $\mathcal{U}$. By definition of the weak phirotope associated to a matrix, 
since $B\in\mathcal{U}$ it follows that $p(B)\neq0$ and $q(B)\neq0$.

Property (2) then follows by a suitable choice of values for the variables $a_{i}^{j}$ (e.g., choosing suitable $x_1,\ldots,x_{d-1}\in \mathbb C$ and setting $a_i^j$ equal to the $j$-th power of $x_i$)  such that the 
determinants in the definition of $p$ and $q$ are Vandermonde-type minors of the matrix $A$ of the form  
\eqref{canonicalform}.
\end{proof}

\begin{corollary}\label{cor:nonC}
Every component of the function defined by the expressions
\eqref{dadostarskyantos} on $\mathcal{U}'$ is non-constant, and the same
holds for its restriction to $\mathcal{U}$.
\end{corollary}
\begin{proof}
The claim for $\mathcal{U}'$ is the second part of the claim of Lemma \ref{LastSide}, 
the claim on the restriction to $\mathcal{U}$ follows from Corollary \ref{cor:nonemptyconnected} 
and the open map theorem.
\end{proof}

\begin{lemma}
 \label{RealizabilityCriterion}
 For the uniform matroid $U_{d}(m)$ on the ground set $[m]$ with $m\geq5$ and $2\leq d\leq m-2$ assume,
 with the hypothesis and notations of Corollary \ref{geo2} with respect to the phase hyperfield (i.e., $\mathbb H = \mathbb P$), that 
 the condition
 \begin{equation}
   \label{condizioneoneone}
 F(\mathcal{B})\cap\prod_{j=1}^{n_{M}}\left[X_{\mathbb{P}}\setminus\left\lbrace\right(1,1),(1,-1),(-1,1)\rbrace\right]\neq\emptyset
 \end{equation}
 is satisfied. Then, there exists a non-realizable non-chirotopal weak phased matroid with underlying matroid 
 $U_{d}(m)$.
\end{lemma}

\begin{remark}
Since we will henceforth only deal with weak phased matroids, we will lighten the notation by suppressing superscripts, thus writing $\mathcal R_{\mathbb H} (M)$ for $\mathcal R_{\mathbb H}^w$, and similarly for $\mathcal N _{\mathbb H} (M)$ and $\mathcal M_{\mathbb H} (M)$.
\end{remark}

\begin{proof}
 Consider the following subspace of $\mathcal{R}_{\mathbb{P}}(U_{d}(m))$: 
 \begin{equation*}
  \mathcal{Z}_{\mathbb{P}}(U_{d}(m))=
  \{P\in\mathcal{R}_{\mathbb{P}}(U_{d}(m))\mid
                                          \text{$P$ is realizable}\}\setminus j(\mathcal{R}_{\mathbb{S}}(U_{d}(m))),
 \end{equation*}
 where $j$ denotes the inclusion induced by the natural hyperfield homomorphism $\mathbb S\hookrightarrow \mathbb P$ as in Theorem \ref{rescalingspaceembeddings}.
 
We will use the characterization of $\mathcal{Z}_{\mathbb{P}}(U_{d}(m))$ below, whose proof,
which relies on work of Ruiz \cite[
Theorem 5.1 and Section 3.3.6 ]{Rui13}, we postpone.

\begin{lemma}
 \label{lastLEMMA} Recall Definition \ref{def:V} and
 let $\mathcal{T}$ be the subset of $\mathcal{V}$ consisting of matrices $A\in M_{d,m}(\mathbb{C})$ of the form 
 \eqref{canonicalform} and such that
 \begin{enumerate}
  \item[(C1)] $\varphi_{A}$ has underlying matroid $U_{d}(m)$;
  \item[(C2)] $\varphi_{A}$ has at least a value in $(S^{1}\cup\{0\})\setminus\{0,1,-1\}$.
 \end{enumerate}
 Let $\Lambda:\mathcal{T}\longrightarrow\mathcal{R}_{\mathbb{P}}(U_{d}(m))$ be the map 
 which associates to a matrix $A\in\mathcal{T}$ the phasing class of the weak phased matroid represented by 
 $\varphi_{A}$. Then,
 \begin{enumerate}
  \item[(P1)] $\mathcal{T}$ is either empty or diffeomorphic to a non-empty open subspace of $\mathcal{V}$; 
  \item[(P2)] $\Lambda$ is continuous;
  \item[(P3)] $\Lambda$ establishes a one-to-one correspondence between $\mathcal{T}$ and 
              $\mathcal{Z}_{\mathbb{P}}(U_{d}(m))$.
 \end{enumerate}
\end{lemma}

Let us continue with the proof of Lemma \ref{RealizabilityCriterion}. 

The homeomorphism $\rho:\mathcal{R}_{\mathbb{P}}(U_{d}(m))\to \mathcal{G}^R_{\mathbb{P}}(U_{d}(m))$ from Theorem \ref{onetoone}, the identity of Proposition \ref{geo1}  and the projection $F$ from Corollary \ref{geo2} fit into the following diagram, which also desplays the homeomorphisms $\Theta:=F\circ \rho$ and $\Sigma:=\Theta\circ\Lambda$, as well as the map $\pi: \mathcal M_{\mathbb{P}}(U_{d}(m))\to \mathcal R_{\mathbb{P}}(U_{d}(m))$ representing the quotient by $\equiv$ (see Proposition \ref{common}).

\begin{equation}\label{defSigma}
\begin{tikzcd}
&
\mathcal T 
\arrow[bend left=0, dashed]{rd}{\Sigma}
\arrow{d}{\Lambda}
& \\
\mathcal M_{\mathbb{P}}(U_{d}(m))
\arrow{r}{\pi}
&
\mathcal{R}_{\mathbb{P}}(U_{d}(m))
\arrow{d}{\rho}
\arrow[bend left=0,dashed]{r}{\Theta}
&
F(\mathcal{B})\cap\prod_{j=1}^{n_{M}}X_{\mathbb P}\\
&
\mathcal{G}^R_{\mathbb{P}}(U_{d}(m))
\arrow[equal]{r}
& \mathcal{B}\cap \prod_{j=1}^{n_M} (X_{\mathbb P}\times\mathbb P^*)
\arrow{u}{F}
\end{tikzcd}
\end{equation}

Observe that $\Theta$ restricts to a homeomorphism between $\mathcal{Z}_{\mathbb{P}}(U_{d}(m))$ and 
 a subspace of $F(\mathcal{B})\cap\prod_{j=1}^{n_{M}}X$. 
   Write from now 
            \newcommand{\mostro}{\widehat{\mathcal X}}
        \begin{equation*}
         \mostro := \prod_{j=1}^{n_{M}}
                                                 [X_{\mathbb P}\setminus\{(1,1),(1,-1),(-1,1)\}].
        \end{equation*}
        Notice that $\{(1,1),(1,-1),(-1,1)\}=X_{\mathbb{S}}$ and thus Proposition \ref{geo1} and Corollary \ref{geo2} applied to $\mathbb H=\mathbb S$, together with Theorem \ref{rescalingspaceembeddings} we see that all rescaling classes of  chirotopal phased matroid are outside $\mostro$. 
 
We distinguish between two cases.
 
 \begin{itemize}
\item[\emph{Case 1.}]
         Assume that 
         \begin{equation}
          \label{emptyintersection}
          \Theta(\mathcal{Z}_{\mathbb{P}}(U_{d}(m)))\cap
          \mostro
          =\emptyset.
         \end{equation}
         By hypothesis \eqref{condizioneoneone}, there exists a non-chirotopal $P\in F(\mathcal{B})\cap\prod_{j=1}^{n_{M}}X_{\mathbb P}$ 
         with $P=\Theta\circ\pi(\Phi)$ for some weak phased 
         matroid $\Phi$ over $U_{d}(m)$. From \eqref{emptyintersection} we know that $\Phi$ cannot be realizable.
 
 \item[\emph{Case 2.}] Now, suppose that
        \begin{equation}
          \label{NONemptyintersection}
          \Theta(\mathcal{Z}_{\mathbb{P}}(U_{d}(m)))\cap
          \mostro  
                  \neq\emptyset
         \end{equation}
        and consider the subset $\mathcal{U}\subseteq\mathcal{T}$ defined by
        \begin{equation*}
        \mathcal{U}=\Sigma^{-1}\left(F(\mathcal{B})\cap\mostro
        \right). 
        \end{equation*}
        Notice that $\mathcal{U}$ is non-empty. Otherwise, 
        \begin{equation*}
         \Sigma(\mathcal{T})\cap\left(F(\mathcal{B})\cap \mostro
          \right)=\emptyset
        \end{equation*}
        in contradiction to 
        \eqref{NONemptyintersection}, since $\Sigma(\mathcal{T})=\Theta(\mathcal{Z}_{\mathbb{P}}(U_{d}(m)))$
        by Lemma \ref{lastLEMMA}.
        Moreover, $\mathcal{U}$ is an open subset of $\mathcal{T}$. This follows from the continuity of $\Sigma$ and 
        from the fact that 
        $\mostro=F(\mathcal{B})\cap\prod_{j=1}^{n_{M}}[X_{\mathbb P}\setminus\{(1,1),(1,-1),(-1,1)\}]$ is an open subset of 
        $F(\mathcal{B})\cap\prod_{j=1}^{n_{M}}X_{\mathbb P}$.
        Therefore, point $(P1)$ in Lemma \ref{lastLEMMA} implies that $\mathcal{U}$ is diffeomorphic to a non-empty 
        open subset of $\mathcal{V}$. In particular, via the identification of $\mathcal{V}$ with 
        $\mathbb{C}^{(d-1)(m-d-1)}$ we conclude that $\mathcal{U}$ is a differentiable 
        manifold of dimension $2(d-1)(m-d-1)$.
        
   \noindent     On the other hand, 
        $\mostro$ is a 
        non-empty open subset of $F(\mathcal{B})$.
        From Definition \ref{def:Bone} and item (iii) after Theorem \ref{onetoone} we then conclude that the dimension of $F(\mathcal B)$ equals the rank of the group $\mathbb T^{(0)}_M$, and by \cite[Theorem 8.1]{DW89}  this rank is    $\binom{m}{d}-m$. 
        Thus, 
        \begin{equation*}
         F(\mathcal{B})\cap\mostro
        \end{equation*}
        is a differentiable manifold of dimension $\binom{m}{d}-m$. 
        
    \noindent    Consider the map
        \begin{equation}\label{Smooth}
         \left.\Sigma\right|_{\mathcal{U}}:\mathcal{U}
         \longrightarrow 
         F(\mathcal{B})\cap\mostro
        \end{equation}
        From \eqref{DefAsso} and the construction of the map $\Sigma$ in Diagram \eqref{defSigma}, given a matrix $A\in\mathcal{T}$, the 
        components of $\Sigma(A)$ are rational expressions of polynomials and 
        absolute values of polynomials in the entries of the matrix $A$ (recall that $\rho$ is induced by the change of presentation of the inner Tutte group from $\mathbb T^{(0)}_M$ to $\mathcal{T}_{M,<_{J}}^{(0)}$, e.g. as explicitly described in Remark \ref{iso20} and Remark \ref{iso2t}). Since the underlying matroid of $\varphi_A$ is uniform (by assumption (C1)), all these rational expressions are well defined.
        Thus, \eqref{Smooth} is a smooth map of differentiable manifolds.
Now consider the following inequality.\footnote{For a proof, show $\binom{m}{d}>m+2(d-1)(m-d-1)$ starting with $\binom{m}{d}\geq \binom{m}{2}>\frac{m^2-2m+4}{2}$ (the first inequality because $2\leq d\leq m-2$, the second since $m\geq 5$).}
        \begin{equation*}\label{tool}
         \dim F(\mathcal{B})\cap\mostro= \binom{m}{d}-m>2(d-1)(m-d-1)=\dim \mathcal U.
        \end{equation*}   
        This allows us to apply \cite[Chapter 3, Proposition 1.2]{Hir76} to the map $\Sigma\vert_{\mathcal U}$, from which we deduce that there exists an element 
        \begin{equation*}\label{MorsySardy}
         P\in\left(F(\mathcal{B})\cap\mostro
         \right)
         \setminus
         \Sigma(\mathcal{U}).           
        \end{equation*}
        with $P=\Theta\circ\pi(\Phi)$ for some phased matroid $\Phi$ over $U_{d}(m)$.
        In particular, 
        \begin{equation}\label{lastgiusta}
         P\in F(\mathcal{B})\cap\mostro. 
        \end{equation}
        By definition of $\mathcal{U}$ we must have
         \begin{equation}\label{lasty}
          P\in\left(F(\mathcal{B})\cap\mostro 
          \right)
          \setminus
          \Theta(\mathcal{Z}_{\mathbb{P}}(U_{d}(m))).
         \end{equation}
         To see this,   assume 
         $P\in\Theta(\mathcal{Z}_{\mathbb{P}}(U_{d}(m)))$. Hence, from point (P3) in Lemma \ref{lastLEMMA}, we have 
         \begin{equation*}
          P\in\Theta(\mathcal{Z}_{\mathbb{P}}(U_{d}(m)))\setminus\Sigma(\mathcal{U})=
                 \Sigma(\mathcal{T})\setminus\Sigma(\mathcal{U})=
                 \Sigma(\mathcal{T}\setminus\mathcal{U}).
         \end{equation*}
         Thus, 
         \begin{equation*}
          P\in\left(F(\mathcal{B})\cap\prod_{j=1}^{n_{M}}X_{\mathbb P}\right)\setminus
               \left(F(\mathcal{B})\cap\prod_{j=1}^{n_{M}}[X_{\mathbb P}\setminus\{(1,1),(1,-1),(-1,1)\}]\right)
         \end{equation*}
         contradicting \eqref{lastgiusta}.
         Now \eqref{lasty} says that $\Phi$ is non-chirotopal (since $P\in \mostro$) and non-realizable ($P\not\in \Theta(\mathcal{Z}_{\mathbb{P}}(U_{d}(m)))$) , so          the claim follows.
         \end{itemize}
\end{proof}

\begin{proof}[Proof of Lemma \ref{lastLEMMA}]$\,$ \\
\emph{Proof of property $(P1)$.}
       If $\mathcal{T}$ is empty there is nothing to say. Let us assume $\mathcal{T}$ non-empty. With \eqref{DefAsso},         condition (C1) is equivalent to 
      \begin{equation}\label{eqC1}
       \begin{array}{ccc}
        \det(A^{i_{1}},\ldots,A^{i_{d}})\neq0 & \text{for all} & 1\leq i_{1}<\cdots<i_{d}\leq m. \\
       \end{array}
      \end{equation}
      Similarly, condition $(C2)$ is equivalent to 
     \begin{equation}\label{eqC2}
       \begin{array}{ccc}
 \det(A^{i_{1}},\ldots,A^{i_{d}})\in\mathbb{C}\setminus\mathbb{R} & \text{for some} & 1\leq i_{1}<\cdots<i_{d}\leq m.\\
       \end{array}
     \end{equation}
     Condition \eqref{eqC1} and \eqref{eqC2} define an open subset of $\mathcal{V}$. 
     
 \noindent \emph{Proof of property $(P2)$.} Let $\mathcal{N}_{\mathbb{P}}^{p}(M)$ as in Definition \ref{TypeBa} and let 
       $\Omega:\mathcal{T}\longrightarrow\mathcal{N}_{\mathbb{P}}^{p}(M)$ be the function $A\mapsto\varphi_{A}$ which 
       maps a matrix $A\in\mathcal{T}$ to the weak phirotope $\varphi_{A}$ associated to $A$ defined as in \eqref{DefAsso}.  Comparing the definitions it is apparent that $\Omega$ is continuous.
       Now, consider the projections 
       \begin{equation*}
        \pi_{\sim_{p}}:\mathcal{N}^{p}_{\mathbb{P}}(U_{d}(m))
                      \longrightarrow\mathcal{M}^{p}_{\mathbb{P}}(U_{d}(m))
       \end{equation*}
       and
       \begin{equation*}
        \pi_{\cong_{p}}:\mathcal{M}^{p}_{\mathbb{P}}(U_{d}(m))
                       \longrightarrow\mathcal{R}^{p}_{\mathbb{P}}(U_{d}(m)).
       \end{equation*}
       From the definition of the topological spaces 
       $\mathcal{M}^{p}_{\mathbb{P}}(U_{d}(m))$ and 
       $\mathcal{R}^{p}_{\mathbb{P}}(U_{d}(m))$
       (compare Definition \ref{TypeBa} and Definition \ref{TypeB}) these projections
       are continuous.
       Thus, $\Lambda$ is continuous since it is 
       composition of the continuous functions $\pi_{\cong_{p}}$, $\pi_{\sim_{p}}$ and $\Omega$.
 
 \noindent \emph{Proof of property $(P3)$.} By definition of $\mathcal{T}$ and $\mathcal{Z}_{\mathbb{P}}(U_{d}(m))$ 
        it follows that 
        $\Lambda(A)$ is an element of $\mathcal{Z}_{\mathbb{P}}(U_{d}(m))$.
        We need to check that
        $\Lambda:\mathcal{T}\longrightarrow\mathcal{Z}_{\mathbb{P}}(U_{d}(m))$ is  bijective.
        
          \noindent \emph{$\Lambda$ is surjective.}
          Let $P\in\mathcal{Z}_{\mathbb{P}}(U_{d}(m))$. Since $P$ is realizable, there exists a matrix 
          $B\in M_{d,m}(\mathbb{C})$ such that $P=\pi_{\cong_p}\circ \pi_{\sim_p}(\varphi_{B})$. 
          It is not hard to see that 
          there exists $A$ of the form \eqref{canonicalform} such that 
          $\pi_{\cong_{p}}\circ\pi_{\sim_{p}}(\varphi_{A})=\pi_{\cong_{p}}\circ\pi_{\sim_{p}}(\varphi_{B})$. 
          From the definition of 
          $\mathcal{Z}_{\mathbb{P}}(U_{d}(m))$ it is not hard to see that $A\in\mathcal{T}$.
          
          \noindent \emph{$\Lambda$ is injective.} 
          Let $A,B\in\mathcal{T}$ such that $\Lambda(A)=\Lambda(B)$. By definition of $\cong_{p}$, there exist
          $a\in S^{1}$ and a function $h:[m]\longrightarrow S^{1}$ such that for $(x_{1},\ldots,x_{d})\in [m]^{d}$ 
          \begin{equation}\label{piromadman}
            \varphi_{A}(x_{1},\ldots,x_{d})=a\left(\prod_{j=1}^{d}h(x_{j})\right)
            \varphi_{B}(x_{1},\ldots,x_{d}).
          \end{equation} 
          Since $A$ and $B$ are of the form \eqref{canonicalform}, we can explicitly compute $\varphi_A$ and $\varphi_B$ on selected arguments using Equation \eqref{DefAsso}. In particular, for $1\leq l \leq d$ we can consider the $d$-tuple $(x^{(l)}_{1},\ldots,x^{(l)}_{d})$  
 defined by setting
 \begin{equation*}
   x^{(l)}_{j}=\left\lbrace
               \begin{array}{ccc}
                  j   & \text{if} & j\neq l; \\
                  d+1 & \text{if} & j=l.     \\
               \end{array}
               \right.
 \end{equation*}
Then, $\varphi_A(x^{(l)}_{1},\ldots,x^{(l)}_{d})=\varphi_B(x^{(l)}_{1},\ldots,x^{(l)}_{d})$ and hence
 \begin{equation*}
  1=a\frac{h(1)\cdots h(d+1)}{h(l)}
 \end{equation*}
 for $1\leq l \leq d$. Therefore $h(1)=h(2)=\cdots=h(d)$.
 
 Moreover, for all $d\leq l \leq m$ the form \eqref{canonicalform} for $A$ and $B$ gives 
 \begin{equation}\label{eq:AB1} \varphi_A(1,2,\ldots,d-1,l)=\varphi_B(1,2,\ldots,d-1,l)=1
 \end{equation} and thus
 $1=ah(1)\cdots h(d-1)h(l)$. In particular, $h(d)=\ldots = h(m)$.
 
 All in all, we have proved that $h$ is constant on $[m]$, say equal to some $h_0\in \mathbb H^*$. Thus, $\varphi_A=ah_0^d \varphi_B$, and evaluating on the $d$-tuple $(1,\ldots,d)$ we see that it must be $ah_0^{d}=1$. So 
          \begin{equation}\label{piromadmanequal}
           \varphi_{A}=\varphi_{B}
          \end{equation}
which, together with the proof of
          \cite[Theorem 5.1]{Rui13} -- where the matrices $A$ and $B$ are reconstructed explicitly from $\varphi_A$ and $\varphi_B$ -- implies  $A=B$.
\end{proof}

\begin{lemma}
 \label{HoloSide}
 Let $h\geq1$ and let $\Omega$ be a non-empty open connected subset of $\mathbb{C}^{h}$. Let $k\geq1$ and let 
 $F:\Omega\longrightarrow\mathbb{C}^{k}$ be a holomorphic map. Assume that none of the components 
 $F_{1},\ldots,F_{k}$ of $F$ is constant. Then, there exists $z_{0}\in\Omega$ such that 
 $F(z_{0})\in\left(\mathbb{C}\setminus\mathbb{R}\right)^{k}$.
\end{lemma}

\begin{proof}
 Set $\Omega^{(0)}=\Omega$ and let $F_{1}$ be the first component of $F$. Since $\Omega$ 
 is a non-empty open connected 
 subset of $\mathbb{C}^{h}$ and $F_{1}$
 is a holomorphic non-constant function on $\Omega$, 
 the open map theorem implies that there is $z_{0}^{(1)}\in\Omega^{(0)}$ such that 
 $F_{1}(z_{0}^{(1)})\in\mathbb{C}\setminus\mathbb{R}$. 
 By continuity of $F_{1}$ it is possible to find an open connected neighborhood $\Omega^{(1)}$ of $z_{0}^{(1)}$ in 
 $\Omega^{(0)}$ with $F_{1}(\Omega^{(1)})\subseteq\mathbb{C}\setminus\mathbb{R}$.
 
 Now, let $F_{2}$ be the second component of $F$. Since $\Omega$ is connected and $F_{2}$ is non-constant on 
 $\Omega$, again from the open map theorem it follows that $F_{2}$ can not be constant on $\Omega^{(1)}$.
 So that, as above, there exists $z_{0}^{(2)}\in\Omega^{(1)}$ and an open connected neighborhood $\Omega^{(2)}$ 
 of $z_{0}^{(2)}$ in $\Omega^{(1)}$ such that $F_{2}(\Omega^{(2)})\subseteq\mathbb{C}\setminus\mathbb{R}$.
 
 Hence, we can recursively find for each component $F_{j}$ of $F$ a point $z_{0}^{(j)}$ and an open connected 
 neighborhood $\Omega^{(j)}$ of $z_{0}^{(j)}$ in $\Omega^{(j-1)}$ such that 
 $F_{j}(\Omega^{(j)})\subseteq\mathbb{C}\setminus\mathbb{R}$. Set $z_{0}=z_{0}^{(k)}$. By construction 
 $z_{0}\in\Omega^{j}$ for each $j=1,\ldots,k$. Thus, $F_{j}(z_{0})\in\mathbb{C}\setminus\mathbb{R}$ for $j=1,\ldots,k$.
\end{proof}

\bibliographystyle{plain}
\bibliography{Full2}

\end{document}